\documentclass[12pt]{amsart}
\usepackage[pdftex]{graphicx}
\usepackage{amsmath}
\usepackage{amssymb}
\usepackage{amsfonts}
\usepackage{amsthm}
\usepackage{bbm}
\usepackage{url}
\usepackage{mathtools}
\usepackage{ascmac}
\usepackage{caption}
\usepackage{subcaption}
\usepackage{aliascnt}
\usepackage[breaklinks=true,pdftex]{hyperref}
\usepackage[nameinlink]{cleveref}
\usepackage{appendix}
\usepackage{tikz-cd}
\usetikzlibrary{fit}
\date{}
\author{Kan Nagano}
\address{Department of Mathematics, Institute of Science Tokyo, 2-12-1, Ookayama, Meguro-ku, Tokyo, 152-8551, Japan}
\email{nagano.k.ah@m.titech.ac.jp}
\makeatother
\theoremstyle{definition}
\newtheorem{definition}{Definition}[section]

\newaliascnt{example}{definition}
\newtheorem{example}[example]{Example}
\aliascntresetthe{example}

\newaliascnt{construction}{definition}
\newtheorem{construction}[construction]{Construction}
\aliascntresetthe{construction}
\theoremstyle{plain}
\newaliascnt{theorem}{definition}
\newtheorem{theorem}[theorem]{Theorem}
\aliascntresetthe{theorem}

\newaliascnt{lemma}{definition}
\newtheorem{lemma}[lemma]{Lemma}
\aliascntresetthe{lemma}

\newaliascnt{corollary}{definition}
\newtheorem{corollary}[corollary]{Corollary}
\aliascntresetthe{corollary}

\newaliascnt{proposition}{definition}
\newtheorem{proposition}[proposition]{Proposition}
\aliascntresetthe{proposition}

\newaliascnt{conjecture}{definition}
\newtheorem{conjecture}[conjecture]{Conjecture}
\aliascntresetthe{conjecture}

\theoremstyle{remark}
\newaliascnt{remark}{definition}
\newtheorem{remark}[remark]{Remark}
\aliascntresetthe{remark}

\crefname{figure}{Figure}{Figures}
\crefname{remark}{Remark}{Remarks}
\crefname{corollary}{Corollary}{Corollarys}
\crefname{definition}{Definition}{Definitions}
\crefname{theorem}{Theorem}{Theorems}
\crefname{conjecture}{Conjecture}{Conjectures}
\crefname{lemma}{Lemma}{Lemmas}
\crefname{example}{Example}{Examples}
\crefname{proposition}{Proposition}{Propositions}
\crefname{section}{Section}{Sections}
\crefname{appendix}{Appendix}{Appendices}
\crefname{construction}{Construction}{Constructions}

\DeclareMathOperator{\Camb}{Camb}
\DeclareMathOperator{\Tors}{Tors}
\DeclareMathOperator{\Hasse}{Hasse}
\newcommand{\edge}{\ \text{---}\ }

\title{Flip of lattices}

\begin{document}

\begin{abstract}
In this paper, we introduce a new combinatorial operation, called a flip, on arbitrary partially ordered sets. We define a mutation to be a flip that maps a lattice to a lattice.
We study properties of flips, and give a necessary and sufficient condition for a flip to be a mutation. We introduce locally mutable lattices and mutable lattices in terms of flips, and prove that mutable lattices are semidistributive.
We show that type-A and type-B Cambrian lattices are locally mutable, and those associated with the finite-type Coxeter quivers with different orientations are related also by the sequence of mutations. Finally we introduce a new class of lattices, called Ordovician lattices, as the lattices obtained from Cambrian lattices by iterated mutations. We provide conjectures on the structure of Ordovician lattices and on the compatibility between our mutation and the mutation in the theory of cluster algebras.
\end{abstract}

\maketitle

\setcounter{tocdepth}{1}
\tableofcontents

\section{Introduction} 
N.~Reading constructed lattices associated with finite-type Coxeter quivers in 2006, and named them Cambrian lattices \cite{MR2258260}.
Its definition is inspired by the theory of cluster algebras \cite{MR1887642,MR2004457,MR2110627,MR2295199}.
The Cambrian lattices are vast generalization of classical Tamari lattices \cite{MR146227}. 
S.~Ladkani introduced the operation \emph{Flip-Flop} on arbitrary partially ordered sets \cite{ladkani2007universalderivedequivalencesposets}.
The idea of its definition stems from quiver representation theory.
It was proved that the Cambrian lattices associated with finite-type Coxeter quivers with different orientations were related by a sequence of Flip-Flops \cite{ladkani2007universalderivedequivalencesposets}.
Cambrian lattices have been studied by both combinatorial and tilting-theoretic approaches \cite{MR2258260,MR4579952}.

In this paper, we introduce a new combinatorial operation, called a \emph{flip}, which describes a connection between two posets.
The definition of flips is purely order-theoric.
We define a new class of lattices, which we call Ordovician lattices, by applying flips to Cambrian lattices.
We prove that Ordovician lattices, in general, cannot be obtained from the weak order of finite Coxeter groups in the same way as Cambrian lattices.
We present several observations suggesting that Ordovician lattices are associated with infinite Coxeter groups.
Since the weak order of infinite Coxeter groups are not lattices in general, the construction of Cambrian lattices cannot be extended to this case. 
In \cite{barkley2025affinetamarilattice}, G.~Barkley and C.~Defant construct affine version of Tamari lattice by the affine Dyer lattice. However, their construction essentially depends on the specific property of Coxeter group of type $\tilde{A}_n$ \cite{barkley2025affinetamarilattice,MR3943754,barkley2025affineextendedweakorder}.
We study Ordovician lattices without considering infinite Coxeter groups, and conjecture its algebraic property.

\subsection{Main results}
We review the main terminology, results, and conjectures of this paper.
\begin{definition}[\cref{flip pair,definition of flip}]
Let $L$ be a poset. We assume that $A,B\subset L$ satisfies the following:
\[ L=A\cup B,\  \emptyset = A\cap B, \]
\[ \forall x \in A , \forall y \in B \ \ x\ngeq y. \] 
Then we define a set with a relation $(L',\leq_{L'})$ by $L'=L$ as a set, and
\begin{align*}
  x \leq_{L'} y \ \Leftrightarrow & \ \exists x_0,x_1,\cdots x_n, \\
  & \text{such that} \  x_0=x,\ x_n=y,\ x_0 \prec_{L'} x_1 \prec_{L'} \cdots \prec_{L'} x_n.
\end{align*}
Here, $\prec_{L'}$ is defined as:
\begin{align*}
  x \prec_{L'} y \ \Leftrightarrow & \ (x,y\in A \ \text{and} \  x \prec_{L} y) \\
  \ & \text{or} \ (x,y\in B \ \text{and} \  x \prec_{L} y) \\
  \ & \text{or} \  (x\in B, y \in A \ \text{and} \  y \prec_{L} x)
\end{align*}
where $\prec_{L}$ stands for a cover relation of $L$.
We refer to $(L',\leq_{L'})$ as the flip of $L$ by a flip pair $(A,B)$.
We denote this flip by $\mu_{(A,B)}$.
Roughly speaking, $(L',\leq_{L'})$ is obtained from $(L,\leq_L)$ by reversing all cover relations between $A$ and $B$.
\end{definition}

We present an example of a flip. In \cref{fig:MutationBetweenA3TamariandAffineTamari}, the two black components represent $A$ and $B$, and the edges represent cover relations of $L$.
Then the flip $\mu_{(A,B)}$ reverses all cyan lines.
Flips modify the structure of posets only slightly.
  \begin{figure}[htbp]
\centering
\begin{tabular}{cc}
  \begin{minipage}[t]{0.3\linewidth}
  \centering
\begin{tikzpicture}[scale=0.3]

\node[name=A0,draw,inner sep=2pt,circle] at (3,-4) {};
\node[name=A1,draw,inner sep=2pt,circle] at (5,0) {};
\node[name=A2,draw,inner sep=2pt,circle] at (1,0) {};
\node[name=A3,draw,inner sep=2pt,circle] at (5,2) {};
\node[name=A4,draw,inner sep=2pt,circle] at (1,4) {};
\node[name=A5,draw,inner sep=2pt,circle] at (1,6) {};
\node[name=A6,draw,inner sep=2pt,circle] at (3,8) {};
\node[name=B0,draw,inner sep=2pt,circle] at (-3,-6) {};
\node[name=B1,draw,inner sep=2pt,circle] at (-1,-4) {};
\node[name=B2,draw,inner sep=2pt,circle] at (-1,-2) {};
\node[name=B3,draw,inner sep=2pt,circle] at (-5,0) {};
\node[name=B4,draw,inner sep=2pt,circle] at (-1,2) {};
\node[name=B5,draw,inner sep=2pt,circle] at (-5,2) {};
\node[name=B6,draw,inner sep=2pt,circle] at (-3,6) {};

\draw[black,<-] (A0)--(A1);
\draw[black,<-] (A0)--(A2);
\draw[black,<-] (A1)--(A3);
\draw[black,<-] (A1)--(A5);
\draw[black,<-] (A2)--(A4);
\draw[black,<-] (A4)--(A5);
\draw[black,<-] (A3)--(A6);
\draw[black,<-] (A5)--(A6);

\draw[black,<-] (B0)--(B1);
\draw[black,<-] (B0)--(B3);
\draw[black,<-] (B1)--(B2);
\draw[black,<-] (B1)--(B5);
\draw[black,<-] (B2)--(B4);
\draw[black,<-] (B3)--(B5);
\draw[black,<-] (B4)--(B6);
\draw[black,<-] (B5)--(B6);

\draw[cyan,->] (A0)--(B0);
\draw[cyan,->] (A2)--(B2);
\draw[cyan,->] (A3)--(B3);
\draw[cyan,->] (A4)--(B4);
\draw[cyan,->] (A6)--(B6);

\end{tikzpicture}
  \end{minipage} &
  
  \begin{minipage}[t]{0.3\linewidth}
\centering
  \begin{tikzpicture}[scale=0.3]
\node[name=A0,draw,inner sep=2pt,circle] at (3,-8) {};
\node[name=A1,draw,inner sep=2pt,circle] at (5,-4) {};
\node[name=A2,draw,inner sep=2pt,circle] at (1,-4) {};
\node[name=A3,draw,inner sep=2pt,circle] at (5,-2) {};
\node[name=A4,draw,inner sep=2pt,circle] at (1,0) {};
\node[name=A5,draw,inner sep=2pt,circle] at (1,2) {};
\node[name=A6,draw,inner sep=2pt,circle] at (3,4) {};
\node[name=B0,draw,inner sep=2pt,circle] at (-3,-6) {};
\node[name=B1,draw,inner sep=2pt,circle] at (-1,-4) {};
\node[name=B2,draw,inner sep=2pt,circle] at (-1,-2) {};
\node[name=B3,draw,inner sep=2pt,circle] at (-5,0) {};
\node[name=B4,draw,inner sep=2pt,circle] at (-1,2) {};
\node[name=B5,draw,inner sep=2pt,circle] at (-5,2) {};
\node[name=B6,draw,inner sep=2pt,circle] at (-3,6) {};

\draw[black,<-] (A0)--(A1);
\draw[black,<-] (A0)--(A2);
\draw[black,<-] (A1)--(A3);
\draw[black,<-] (A1)--(A5);
\draw[black,<-] (A2)--(A4);
\draw[black,<-] (A4)--(A5);
\draw[black,<-] (A3)--(A6);
\draw[black,<-] (A5)--(A6);

\draw[black,<-] (B0)--(B1);
\draw[black,<-] (B0)--(B3);
\draw[black,<-] (B1)--(B2);
\draw[black,<-] (B1)--(B5);
\draw[black,<-] (B2)--(B4);
\draw[black,<-] (B3)--(B5);
\draw[black,<-] (B4)--(B6);
\draw[black,<-] (B5)--(B6);

\draw[cyan,<-] (A0)--(B0);
\draw[cyan,<-] (A2)--(B2);
\draw[cyan,<-] (A3)--(B3);
\draw[cyan,<-] (A4)--(B4);
\draw[cyan,<-] (A6)--(B6);

\end{tikzpicture}
  \end{minipage}
\end{tabular}
\caption{}
\label{fig:MutationBetweenA3TamariandAffineTamari}
\end{figure}

\begin{theorem}[\cref{flip and minimum}]
  Let $L$ be a connected poset, and let $x\in L$.
  Then there uniquely exists a poset $L'$ obtained from $L$ by a finite sequence of flips whose least element is $x$.
\label{flip and minimum chapter 1}
\end{theorem}
We introduce a certain nonnegative integer $D_L(x,y,z)$ determined from $x,y,z\in L$ and the structure of $L$. See \cref{sec:Invariants of flips} for its precise definition.
As a corollary of \cref{flip and minimum chapter 1}, we show the following:
\begin{corollary}[\cref{D and sequence of flips}]
  Let $L_1=(L,\leq_{L_1})$ and $L_2=(L,\leq_{L_2})$ be connected posets with the underlying set $L$. Then the following are equivalent:
  \begin{enumerate}
    \item $L_2$ is obtained from $L_1$ by applying a finite sequence of flips.
    \item $D_{L_1}(x,y,z)=D_{L_2}(x,y,z)$ for all $x,y,z\in L$.
  \end{enumerate}
\end{corollary}
Let $L$ be a poset. We define a graph $G(L)$ as follows (cf.~\cref{the flip graph G}):
  \begin{itemize}
    \item A vertex is labelled by $L'$ such that $L'$ is obtained from $L$ by applying a finite sequence of flips and such that $L'$ has the least element. 
    \item An edge joins two vertices $L_1$ and $L_2$ if a flip $L_1 \xmapsto[]{\mu_{(A,B)}} L_2$ exists.
  \end{itemize}
\begin{theorem}[\cref{flip graph and Hasse quiver}]
  Let $L$ be a connected poset with the least element.
  Then $G(L)$ is isomorphic to the underlying undirected graph of $\Hasse(L)$.
\end{theorem}

A flip $\mu_{(A,B)}$ is said to be a \emph{mutation} when both $L$ and $\mu_{(A,B)}(L)$ are lattices.
A lattice $L$ may have two or more mutations. These mutations can be labelled by atoms (resp.~coatoms) of $L$.
When a mutation corresponds to an atom $a$, we denote it by $\mu_{a}$.
We define $L$ to be \emph{locally mutable} if $L$ can be mutated by all atoms (resp.~coatoms). 
Moreover, we define $L$ to be \emph{mutable} if all lattices obtained by a finite number of mutations from $L$ are locally mutable lattices.

\begin{theorem}[\cref{mutable implies semidistributive}]
  Let $L$ be a mutable lattice. Then $L$ is semidistributive.
  \label{mutable implies semidistributive chapter 1}
\end{theorem}
The semidistributive property holds for Cambrian lattices, and it is studied from both representation-theoretic and lattice-theoretic viewpoints \cite{MR4579952,MR4280379}.
\cref{mutable implies semidistributive chapter 1} suggests that the structure of mutable lattices might be governed by lattice theory and representation theory.

The next theorem shows that a mutation describes the connection between Cambrian lattices.
It is already known by using Flip-Flop, but our approach is not based on quiver representation theory \cite{ladkani2007universalderivedequivalencesposets,MR4492558}.
The following theorem relates the mutation of lattices and that of quivers.
\begin{theorem}[\cref{mutation of general type Cambrian lattice}]
  Let $\overrightarrow{B}$ be a finite-type Coxeter quiver and $\Camb(\overrightarrow{B})$ the corresponding Cambrian lattice.
  If $i\in \overrightarrow{B}$ is a sink or a source, then $\mu_{\eta((i,i+1))}(\Camb(\overrightarrow{B})) \cong \Camb(\mu_{i}(\overrightarrow{B})) $.
  Here, $\mu_{i}(\overrightarrow{B})$ is the BGP-reflection of $\overrightarrow{B}$ on $i$.
  Moreover, if $\overrightarrow{B},\overrightarrow{B'}$ are finite-type Coxeter quivers which share the same underlying undirected graph, then $\Camb(\overrightarrow{B})$ is mapped to a lattice isomorphic to $\Camb(\overrightarrow{B'})$ by a finite sequence of mutations.
\end{theorem}

\begin{theorem}[\cref{mutation of general type Cambrian lattice}]
  Cambrian lattices of finite-type are locally mutable.
\end{theorem}
We can define a new class of lattices obtained by a finite sequence of mutations from Cambrian lattices. We call them \emph{Ordovician lattices}.
\begin{conjecture}[\cref{conj:structure of Ordovician lattice}]
  Cambrian lattices of finite-type are mutable.
\end{conjecture}
We confirm this conjecture in the cases of types $A_1$, $A_2$, $A_3$, $B_2$, and $B_3$.
The next theorem shows that, in general, Ordovician lattices cannot be obtained from the weak order of any finite Coxeter groups in the same way as Cambrian lattices are.
\begin{theorem}[\cref{general type Ordovician lattice and the infinite case}]
  Let $L$ be a Dynkin type Cambrian lattice, $a_i \in L$ be an atom such that $i$ is neither a sink nor a source, and $L'=\mu_{a_i}(L)$. Then $L'$ cannot be expressed as a quotient of the weak order of any finite Coxeter groups.
\end{theorem}
We further study the relations between the mutation of lattices and the mutation in cluster theory. 
\begin{theorem}[\cref{general type of Ordovician lattice and ideal polygon}]
  Let $Q$ be a weighted quiver of Dynkin type, $L=\Camb(\overrightarrow{B}(Q))$, $a_i \in L$ an atom, and $L'=\mu_{a_i}(L)$.
  Then, \cref{General type Cambrian lattice and ideal polygon} can be applied to $L'$, and the resulting quiver $\overrightarrow{B'}$ is isomorphic to $\overrightarrow{B}(\mu_{i}(Q))$, where $i$ is the vertex of $Q$ that corresponds to $a_i$.
\end{theorem}
\begin{conjecture}[\cref{conj:mutation and quiver with potential}]
  Let $(Q,W)$ be a quiver with potential where $Q$ is a quiver such that $Q$ has no loops or 2-cycles. 
  Denote by $\mathcal{P}(Q,W)$ the Jacobian algebra associated with $(Q,W)$.
  Assume that $\dim (\mathcal{P}(Q,W)) < \infty$.
  Then there exists a mutation that maps $L=\Tors(\mathcal{P}(Q,W))$ to a lattice isomorphic to 
$L'=\Tors(\mathcal{P}(\mu_i(Q,W)))$.
Here, $\Tors(A)$ denotes the lattice of torsion classes of the category of finitely generated left $A$-modules, and $\mu_i(Q,W)$ is the mutation of $(Q,W)$ at a vertex $i$ of $Q$ \textnormal{\cite{MR2480710}}.
\end{conjecture} 
If both $Q,Q'$ are type-A Coxeter quivers, then both $\Tors(\mathcal{P}(Q,0))$ and $\Tors(\mathcal{P}(Q',0))$ are Cambrian lattices \cite{MR4579952}, and this is contained in the theorem stated above.
Hence, if this conjecture holds, Ordovician lattices can be described in terms of representation theory.

\subsection{Organization}
The organization of this paper is outlined below.

In \cref{sec:Orders and Lattices}, we collect preliminaries for this paper.
In \cref{sec:Flips and mutations}, we introduce flips and mutations. We give a necessary and sufficient condition for a flip to be a mutation.
In \cref{sec:Invariants of flips}, we introduce an invariant of flips.
In \cref{sec:Mutable lattices}, we introduce mutable lattices and investigate its properties.
In \cref{sec:Flips of the weak order of finite Coxeter groups}, we discuss the flip of weak order of a finite Coxeter group.
In \cref{sec:Flips of Cambrian lattices}, we study the flip of Cambrian lattices. We also introduce Ordovician lattices which includes Cambrian lattices.
In \cref{sec:A graphical description of flips of Cambrian lattices}, we discuss a graphical description of Cambrian lattices of type-A and type-B.
In \cref{sec:Lattice structures on the A_3 associahedron}, we state all lattice structures on $A_3$ associahedron. We also describe the relationships induced by mutations.

\section{Orders and lattices}
\label{sec:Orders and Lattices}
In this section, we review the basics of lattices, referring to \cite{MR1902334,Sokuron}.
In all sections, \textbf{we consider only those posets that have finite elements}.

A relation $\leq$ on a set $L$ is a partial order if the following conditions hold:
\begin{enumerate}
  \item $\forall x\in L \ x\leq x$
  \item $\forall x,y,z\in L \ x\leq y,y\leq z \Rightarrow x\leq z$
  \item $\forall x,y\in L \ x\leq y\leq x \Rightarrow x=y$
\end{enumerate}
We call $(L,\leq)$ \emph{partially ordered set (poset)}.

Let $x$ and $y$ be elements of a poset $L$. If $x\lneq y$ and $x\leq z\leq y\Rightarrow z\in \{x,y\}$ , then we call $y$ \emph{covers} $x$ , and we write $x\prec y$.

  For elements $a,b\in L$ that satisfy $a\leq b$, we define the \emph{interval} $[a,b]$ to be $\{x\in L\ | \ a\leq x\leq b\}$.

\begin{definition}
We call a sequence $x=p_0 \prec p_1 \prec p_2 \prec \cdots \prec p_{n-1} \prec p_n = y$ a length $n$ \emph{covering chain} from $x$ to $y$. 
\end{definition}
If $x=y$, covering chain from $x$ to $y$ is $\{x\}$, which length is $0$.
Let $L,L'$ be posets.
A map $f:L\to L'$ is an \emph{order homomorphism} if $\forall x,y\in L,\ x\leq_{L}y \ \Rightarrow \ f(x)\leq_{L'}f(y) $.
A map $f$ is an \emph{order isomorphism} if $f$ is bijective and both $f$ and $f^{-1}$ are order homomorphisms.

We denote the \emph{maximum} (resp.~the \emph{greatest element}) of $L$ by $1_L$ (or simply $1$).
Similarly, we denote the \emph{minimum} (resp.~the \emph{least element}) of $L$ by $0_L$ (or simply $0$).
An element $x\in L$ is an \emph{atom} (resp.~\emph{coatom}) if $x\succ 0$ (resp.~$x\prec 1$).

\begin{proposition}
  Let $L,L'$ be posets, and let $f:L\to L'$ be a surjective order homomorphism. If $L$ has the maximum $1_{L}$, $L'$ has the maximum $f(1_{L})$.
  \begin{proof}
    Since $f$ is a order homomorphism, we have $f(1_{L})\geq f(x)$ for all $x\in L$. Then $f(1_{L})$ is the maximum because $f$ is surjective.
  \end{proof}
\end{proposition}

A \emph{lattice} is a poset in which every pair of elements $x,y\in L$ has both a \emph{join} (i.e., least upper bound) and a \emph{meet} (i.e., greatest lower bound). 
We denote a join of $x,y$ by $x\vee y$, and a meet of $x,y$ by $x\wedge y$.

Every finite lattice has the maximum and the minimum element.

Let $L,L'$ be lattices. A map $f:L\to L'$ is a \emph{lattice homomorphism} if $\forall x,y \in L \ f(x\vee y)=f(x)\vee f(y),\ f(x\wedge y)=f(x)\wedge f(y)$.

A map $f$ is a \emph{lattice isomorphism} if $f$ is a lattice homomorphism and bijective.
If $f$ is a lattice isomorphism, then $f^{-1}$ is too.
Every lattice homomorphism is an order homomorphism. 

\begin{proposition}
  Let $f:L\to L'$ be surjective lattice homomorphism, and let $x\prec_{L} y$. Then $f(x)=f(y)$ or $f(x) \prec_{L'} f(y)$.
  \begin{proof}
    We prove the statement by contradiction. By assumption, there exists an element $z'$ satisfying $f(x)\lneq z' \lneq f(y)$.
    Let $z$ be an element satisfying $f(z)=z'$, and let $w$ as $(z\vee x)\wedge y$.
    Then $x\leq w\leq y$, hence $w=x$ or $w=y$.
    On the other hand, equation $f(w)=(f(z)\vee f(x))\wedge f(y)=f(z)$ can be derived, which is a contradiction.
  \end{proof}  
\end{proposition}

\begin{definition}
  Let $L$ be a lattice. A subset $S\subset L$ is an \emph{ideal} if the following conditions hold:
  \begin{enumerate}
    \item $\forall s,t\in S, s\vee t\in S$
    \item $\forall s\in S, x\in L, \ x\leq s\Rightarrow x\in S$
  \end{enumerate}
\end{definition}

If $f:L\to L'$ is surjective lattice homomorphism and $S\subset L$ is an ideal, $f(S)\subset L'$ is an ideal.

\begin{definition}
  A sequence $x_0,x_1,\cdots,x_n$ in $L$ is said to be a \emph{covering path} from $x_0$ to $x_n$ if $x_i\prec x_{i+1} \ or \ x_i \succ x_{i+1}$ for all $i$.
  We say $x_0,x_1,\cdots,x_n$ has $k$ ascending covers where $k=\#\{i\ |\ x_i\prec x_{i+1}\}$.
  Similarly, we say $x_0,x_1,\cdots,x_n$ has $k$ descending covers where $k=\#\{i\ |\ x_i\succ x_{i+1}\}$.
\end{definition}

A \emph{quiver} is a directed graph.
Let $L$ be a poset. Define $\Hasse(L)$ be a quiver as follows;
\begin{enumerate}
  \item Each vertex is labelled by an element of $L$.
  \item Draw an edge $x\to y$ if $x\succ y$.
\end{enumerate}
We call $\Hasse(L)$ \emph{Hasse quiver of $L$}.

A lattice $L$ is a \emph{polygon} \cite{MR3645055} if undirected graph of $\Hasse(L)$ is a cycle.
A polygon $L$ has exactly two covering chains from $0$ to $1$.
We define a \emph{$(m,n)$-polygon} to be a polygon whose two covering chains from $0$ to $1$ have lengths $m$ and $n$, respectively.

\begin{lemma}
\label{cycle in lattice}
Let $L$ be a lattice.
  \begin{enumerate}
    \item Let $(a,b,c,d,a)$ be a covering path in $L$. Then $\{a,b,c,d\}$ is a $(2,2)$-polygon.
    \item Let $(a,b,c,d,e,a)$ be a covering path in $L$. Then $\{a,b,c,d,e\}$ is a $(2,3)$-polygon.
    \item Let $(a_1,a_2,\cdots,a_n,a_1)$ be a covering path in $L$ which has at most two descending covers. Then $\{a_1,a_2,\cdots,a_n\}$ is a $(2,n-2)$-polygon.
  \end{enumerate}
  \begin{proof}
    (3) 
    We prove the statement by contradiction. By assumption, there exists an element $j\neq 1,2,n$ satisfying that $(a_1,a_2),(a_j,a_{j+1})$ are descending covers. Hence $a_2\leq a_1\wedge a_j \leq a_j$. If $a_j\geq a_1$, then $a_{j+1}\lneq a_1 \lneq a_j$, which contradicts to $a_{j+1}\prec a_j$. We thus get $a_1\wedge a_j=a_2$, which contradicts the fact that $a_{j+1}\nleq a_2$ is lower than $a_1,a_j$.
    (1) and (2) follows immediately from (3).
  \end{proof}
\end{lemma}

\section{Flips and mutations}
\label{sec:Flips and mutations}
In this section, we define a \emph{flip} for posets, then study its properties.
In particular, we call it a \emph{mutation} if it is between lattices, and study its necessary and sufficient condition.

\begin{definition}
\label{flip pair}
Let $L$ be a poset. Assume that $A,B\subset L$ have the following properties: \relax
\[ L=A\cup B,\  \emptyset = A\cap B, \]
\[ \forall x \in A , \forall y \in B \ \ x\ngeq y. \] 
Then we call $(A,B)$ a \emph{flip pair} of $L$. Similarly, we call $A$ \emph{footwall}, and $B$ \emph{hanging wall}.
\end{definition}

\begin{definition}
\label{definition of flip}
Let $(A,B)$ be a flip pair of $L$.
Then we define a set with a relation $(L',\leq_{L'})$ by $L'=L$ as a set, and
\begin{align*}
  x \leq_{L'} y \ \Leftrightarrow & \ \exists x_0,x_1,\cdots x_n, \\
  & \text{such that} \  x_0=x,\ x_n=y,\ x_0 \prec_{L'} x_1 \prec_{L'} \cdots \prec_{L'} x_n.
\end{align*}
Here, $\prec_{L'}$ is defined as:
\begin{align*}
  x \prec_{L'} y \ \Leftrightarrow & \ (x,y\in A \ \text{and} \  x \prec_{L} y) \\
  \ & \text{or} \ (x,y\in B \ \text{and} \  x \prec_{L} y) \\
  \ & \text{or} \  (x\in B, y \in A \ \text{and} \  y \prec_{L} x)
\end{align*}
We refer to $(L',\leq_{L'})$ as the \emph{flip} of $L$ by a flip pair $(A,B)$.
Roughly speaking, $(L',\leq_{L'})$ is obtained from $(L,\leq_L)$ by reversing all cover relations between $A$ and $B$.
We denote $(L',\leq_{L'})$ by $\mu_{(A,B)}(L,\leq_L)$; we also use $L'=\mu_{(A,B)}(L)$ and $L\xmapsto[]{\mu_{(A,B)}}L'$.
\end{definition}

When there is no risk of confusion, we denote $\mu_{(A,B)}$ by $\mu$.

\begin{remark}
  When necessary to avoid conflicts with other notation, we denote filp by \emph{faulting flip}.
  The set $L'$, constructed from $(L,\leq_L)$ via a flip $\mu_{(A,B)}$, inherits the labeling of $L$.
  When we say that $L$ and $L'$ are isomorphic (denoted $L\cong L'$), we disregard the labeling.
  In contrast, when we write $L=L'$, we mean that, if both are equipped with labelings, they are identical including the labeling.
\end{remark}

\begin{lemma} Suppose that $L\xmapsto[]{\mu_{(A,B)}}L'$. Then the following hold:
\begin{enumerate}
  \item $(L',\leq_{L'})$ is a poset.
  \item If $x,y\in A$, then $x\leq_{L}y \Leftrightarrow x\leq_{L'}y$.
  \item If $x,y\in B$, then $x\leq_{L}y \Leftrightarrow x\leq_{L'}y$.
  \item $\prec_{L'}$ defines a cover relation of $(L',\leq_{L'})$.
  \[ \textnormal{i.e.,} \ on \ L',\  y \ covers \ x \ \Leftrightarrow \  x\prec_{L'} y \]
  \item Underlying undirected graphs of $\Hasse(L)$ and $\Hasse(L')$ are identical, including the labeling.
\end{enumerate}
\begin{proof}
(1) By definition, no two elements $x\in A, y\in B$ satisfy $x\leq_{L'}y$. Since $A$ and $B$ are posets, $(L',\leq_{L'})$ is a poset.

(2),(3) $(\Rightarrow)$ By definition, both the restriction of $\leq_{L}$ to $A$ and to $B$ are invariant under $\mu_{(A,B)}$.

$(\Leftarrow)$ By assumption, we take $x_0=x\ ,\ x_n=y\ ,\ x_0 \prec_{L'} x_1 \prec_{L'} \cdots \prec_{L'} x_n$.

By definition of $\prec_{L'}$, if $x_i\in A$ holds for some $i$, then $x_j\in A$ holds for any $j\geq i$.

Hence, if $x_0\in A$, then this sequence are consisting of elements of $A$.
Therefore, $x\leq_{L}y$.

The case $x_n\in B$ is the same.

(4) $(\Rightarrow)$ By definition.
$(\Leftarrow)$
We take arbitrary $x\prec_{L'} y$.
From (2) and (3), the claim holds when $x,y\in A$ or $x,y\in B$. 
Hence, we may assume $x\in B,\ y\in A$.
Every sequence $x_0=x\ ,\ x_n=y\ ,\ x_0 \prec_{L'} x_1 \prec_{L'} \cdots \prec_{L'} x_n$ contains exactly one $i$ such that $x_i\in B,x_{i+1}\in A$.
Since \[x_i\succ_{L} x_{i+1} \leq_{L} x_n \prec_{L} x_0 \leq_{L} x_i \] holds, $x_i=x_0,x_{i+1}=x_n$ follows from covering property.
Therefore $n=1$, and $\prec_{L'}$ indeed defines a cover relation of $(L',\leq_{L'})$.
(5) follows from (4).
\end{proof}
\label{After flip}
\end{lemma}

\begin{lemma}
In the above setting, $(B,A)$ is a flip pair of $L'$. Moreover, $L'\xmapsto[]{\mu_{(B,A)}}L$. 
\begin{proof}
The first assertion follows from the proof of (1) of \cref{After flip}, and the second is easy to check.
\end{proof}
\end{lemma}

We refer to $L'\xmapsto[]{\mu_{(B,A)}}L$ as \emph{dual flip} of $L\xmapsto[]{\mu_{(A,B)}}L'$.

\begin{definition}
Let $(A,B)$ be a flip pair of $L$.
We define \emph{fault plane}s $\partial A,\partial B$ as follows:
\[ \partial A = \{ x\in A \ | \  \exists y \in B, \ \& \  x \prec_{L} y \}\]
\[ \partial B = \{ y\in B \ | \  \exists x \in A, \ \& \  x \prec_{L} y \}\]
\end{definition}

\begin{remark}
We comment on why the new concept is named a flip, referring to the name of \emph{Flip-Flop} introduced by S.~Ladkani \cite{ladkani2007universalderivedequivalencesposets}.

Let $A,B$ be posets, and let $f:A\to B$ be an order homomorphism.
We define $(A\sqcup B,\leq_{+}^{f})$ and $(A\sqcup B,\leq_{-}^{f})$ to agree with the original relations in $A$ and $B$, respectively, and to be given as follows under the conditions $x\in A$ and $y\in B$:
\[ x\leq_{+}^{f} y \Leftrightarrow f(x)\leq y\ \ ,\ \  y\leq_{-}^{f} x \Leftrightarrow y \leq f(x) \]
A Flip-Flop is exchanging two posets $(A\sqcup B,\leq_{+}^{f}),(A\sqcup B,\leq_{-}^{f})$.

Flip-Flops and flips are independent concepts.
Flips cannot change the underlying undirected graph of its Hasse quiver, where Flip-Flops can.
On the other hand, for any element $x\in A$, Flip-Flops yield an element $y\in B$ such that the relation on a poset satisfies $x\leq y$ in one case and $y\leq x$ in the other, whereas flips does not necessarily do so.
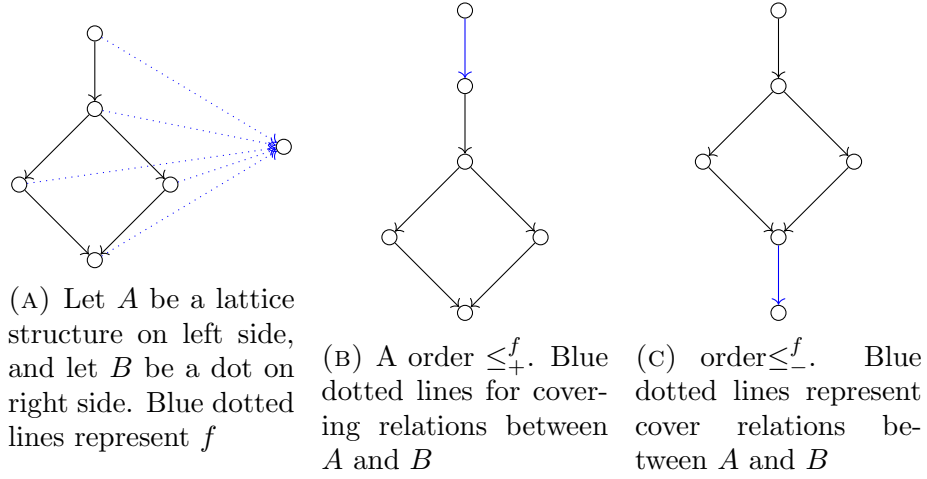
\begin{figure}
  \centering
  \begin{tabular}{ccc}
  \begin{minipage}{0.3\linewidth}
  \centering
    \begin{tikzpicture}[scale=0.5]
\node[draw,shape=circle,inner sep=2pt] (X0) at (0,0) {};
\node[draw,shape=circle,inner sep=2pt] (Xp) at (-2,2) {};
\node[draw,shape=circle,inner sep=2pt] (Xq) at (2,2) {};
\node[draw,shape=circle,inner sep=2pt] (Xr) at (0,4) {};
\node[draw,shape=circle,inner sep=2pt] (X1) at (0,6) {};
\node[draw,shape=circle,inner sep=2pt] (Y) at (5,3) {};

\draw[<-] (X0)--(Xp);
\draw[<-] (X0)--(Xq);
\draw[<-] (Xp)--(Xr);
\draw[<-] (Xq)--(Xr);
\draw[<-] (Xr)--(X1);

\draw[->,dotted,blue=10] (X0)--(Y);
\draw[->,dotted,blue=10] (Xp)--(Y);
\draw[->,dotted,blue=10] (Xq)--(Y);
\draw[->,dotted,blue=10] (Xr)--(Y);
\draw[->,dotted,blue=10] (X1)--(Y);

    \end{tikzpicture}
    \subcaption{Let $A$ be a lattice structure on left side, and let $B$ be a dot on right side. Blue dotted lines represent $f$}
  \end{minipage}&
  \begin{minipage}{0.3\linewidth}
  \centering
    \begin{tikzpicture}[scale=0.5]
\node[draw,shape=circle,inner sep=2pt] (X0) at (0,0) {};
\node[draw,shape=circle,inner sep=2pt] (Xp) at (-2,2) {};
\node[draw,shape=circle,inner sep=2pt] (Xq) at (2,2) {};
\node[draw,shape=circle,inner sep=2pt] (Xr) at (0,4) {};
\node[draw,shape=circle,inner sep=2pt] (X1) at (0,6) {};
\node[draw,shape=circle,inner sep=2pt] (Y) at (0,8) {};

\draw[<-] (X0)--(Xp);
\draw[<-] (X0)--(Xq);
\draw[<-] (Xp)--(Xr);
\draw[<-] (Xq)--(Xr);
\draw[<-] (Xr)--(X1);
\draw[<-,blue] (X1)--(Y);

    \end{tikzpicture}
    \subcaption{A order $\leq^f_{+}$. Blue dotted lines for covering relations between $A$ and $B$}
  \end{minipage}&
  \begin{minipage}{0.3\linewidth}
  \centering
    \begin{tikzpicture}[scale=0.5]
\node[draw,shape=circle,inner sep=2pt] (X0) at (0,0) {};
\node[draw,shape=circle,inner sep=2pt] (Xp) at (-2,2) {};
\node[draw,shape=circle,inner sep=2pt] (Xq) at (2,2) {};
\node[draw,shape=circle,inner sep=2pt] (Xr) at (0,4) {};
\node[draw,shape=circle,inner sep=2pt] (X1) at (0,6) {};
\node[draw,shape=circle,inner sep=2pt] (Y) at (0,-2) {};

\draw[<-] (X0)--(Xp);
\draw[<-] (X0)--(Xq);
\draw[<-] (Xp)--(Xr);
\draw[<-] (Xq)--(Xr);
\draw[<-] (Xr)--(X1);
\draw[<-,blue] (Y)--(X0);

    \end{tikzpicture}
  \subcaption{order$\leq^f_{-}$. Blue dotted lines represent cover relations between $A$ and $B$}
  \end{minipage}
  \end{tabular}
  \caption{An operation that is a Flip-Flop but not a flip. The underlying undirected graphs of $\leq^f_{+}$ and $\leq^f_{-}$ are different.}
  \label{fig:Flip-Flop but not flip}
\end{figure}

\begin{figure}
  \centering
  \begin{tabular}{cc}
  \begin{minipage}{0.4\linewidth}
  \centering
    \begin{tikzpicture}
\node[draw,shape=circle,inner sep=2pt] (X0) at (0,0) {};
\node[draw,shape=circle,inner sep=2pt] (Xa) at (-1,2) {};
\node[draw,shape=circle,inner sep=2pt] (X1) at (0,4) {};
\node[draw,shape=circle,inner sep=2pt] (Y0) at (2,1) {};
\node[draw,shape=circle,inner sep=2pt] (Ya) at (3,3) {};
\node[draw,shape=circle,inner sep=2pt] (Y1) at (2,5) {};

\draw[<-] (X0)--(Xa);
\draw[<-] (Xa)--(X1);
\draw[<-] (Y0)--(Ya);
\draw[<-] (Ya)--(Y1);

\draw[<-,blue] (X0)--(Y0);
\draw[<-,blue] (X1)--(Y1);
    \end{tikzpicture}
  \end{minipage}&
  \begin{minipage}{0.4\linewidth}
  \centering
    \begin{tikzpicture}
\node[draw,shape=circle,inner sep=2pt] (X0) at (0,1) {};
\node[draw,shape=circle,inner sep=2pt] (Xa) at (-1,3) {};
\node[draw,shape=circle,inner sep=2pt] (X1) at (0,5) {};
\node[draw,shape=circle,inner sep=2pt] (Y0) at (2,0) {};
\node[draw,shape=circle,inner sep=2pt] (Ya) at (3,2) {};
\node[draw,shape=circle,inner sep=2pt] (Y1) at (2,4) {};

\draw[<-] (X0)--(Xa);
\draw[<-] (Xa)--(X1);
\draw[<-] (Y0)--(Ya);
\draw[<-] (Ya)--(Y1);

\draw[->,blue] (X0)--(Y0);
\draw[->,blue] (X1)--(Y1);
    \end{tikzpicture}
  \end{minipage}
  \end{tabular}
  \caption{An operation that is a flip but not a Flip-Flop. The direction of blue dotted lines are changed by this operation}
  \label{fig:flip but not Flip-Flop}
\end{figure}
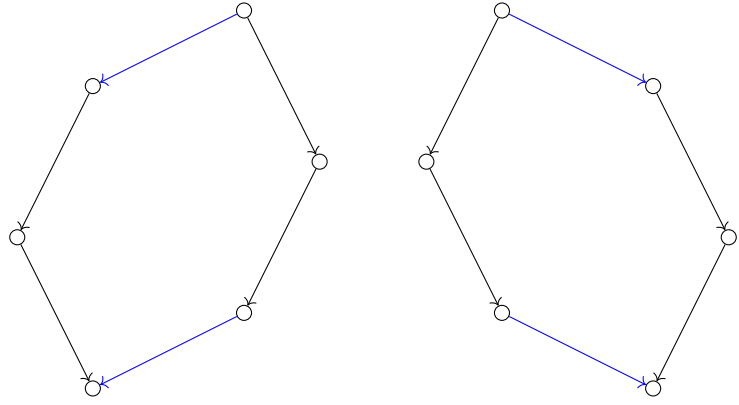
\end{remark}

\begin{example}
We show examples of a flip in \cref{fig:image of a flip,fig:mutation}.
\end{example}

  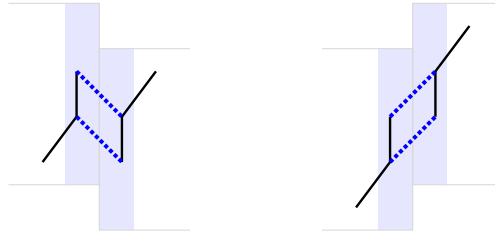
\begin{figure}[htbp]
\centering
\begin{tabular}{cc}
  \begin{minipage}[t]{0.3\linewidth}
  \centering
  \begin{tikzpicture}[scale=0.3]

\coordinate (A1) at (-4,0);
\coordinate (A2) at (0,0);
\coordinate (A3) at (0,-8);
\coordinate (A4) at (-4,-8);
\coordinate (A5) at (-1.5,0);
\coordinate (A6) at (-1.5,-8);
\coordinate (B1) at (4,-2);
\coordinate (B2) at (0,-2);
\coordinate (B3) at (0,-10);
\coordinate (B4) at (4,-10);
\coordinate (B5) at (1.5,-2);
\coordinate (B6) at (1.5,-10);

\draw[draw=none,fill=blue!10] (A5)--(A2)--(A3)--(A6);
\draw[draw=none,fill=blue!10] (B5)--(B2)--(B3)--(B6);
\draw[gray!30] (A1)--(A2)--(A3)--(A4);
\draw[gray!30] (B1)--(B2)--(B3)--(B4);

\coordinate (X1) at (-1,-3);
\coordinate (X2) at (-1,-5);
\coordinate (X3) at (-2.5,-7);
\coordinate (Y1) at (2.5,-3);
\coordinate (Y2) at (1,-5);
\coordinate (Y3) at (1,-7);

\draw[line width=0.9pt] (X1)--(X2)--(X3);
\draw[line width=0.9pt] (Y1)--(Y2)--(Y3);

\draw[densely dotted,blue,line width=1.5pt] (X1)--(Y2);
\draw[densely dotted,blue,line width=1.5pt] (X2)--(Y3);

\end{tikzpicture}
  \end{minipage} &
  
  \begin{minipage}[t]{0.3\linewidth}
  \centering
  \begin{tikzpicture}[scale=0.3]

\coordinate (A1) at (-4,-2);
\coordinate (A2) at (0,-2);
\coordinate (A3) at (0,-10);
\coordinate (A4) at (-4,-10);
\coordinate (A5) at (-1.5,-2);
\coordinate (A6) at (-1.5,-10);
\coordinate (B1) at (4,0);
\coordinate (B2) at (0,0);
\coordinate (B3) at (0,-8);
\coordinate (B4) at (4,-8);
\coordinate (B5) at (1.5,-0);
\coordinate (B6) at (1.5,-8);

\draw[draw=none,fill=blue!10] (A5)--(A2)--(A3)--(A6);
\draw[draw=none,fill=blue!10] (B5)--(B2)--(B3)--(B6);
\draw[gray!30] (A1)--(A2)--(A3)--(A4);
\draw[gray!30] (B1)--(B2)--(B3)--(B4);

\coordinate (X1) at (-1,-5);
\coordinate (X2) at (-1,-7);
\coordinate (X3) at (-2.5,-9);
\coordinate (Y1) at (2.5,-1);
\coordinate (Y2) at (1,-3);
\coordinate (Y3) at (1,-5);

\draw[line width=0.9pt] (X1)--(X2)--(X3);
\draw[line width=0.9pt] (Y1)--(Y2)--(Y3);

\draw[densely dotted,blue,line width=1.5pt] (X1)--(Y2);
\draw[densely dotted,blue,line width=1.5pt] (X2)--(Y3);

\end{tikzpicture}
  \end{minipage}
\end{tabular}
\caption{An image of a flip}
\label{fig:image of a flip}
\end{figure}

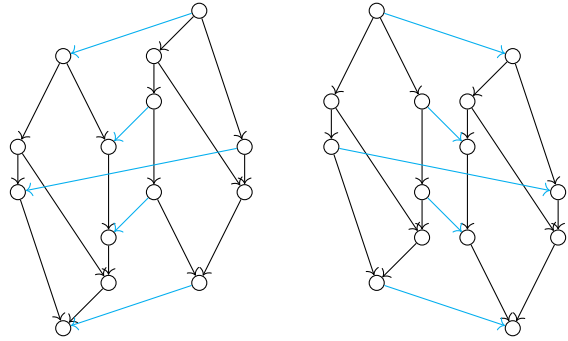
\begin{figure}[htbp]
\centering
\begin{tabular}{cc}
  \begin{minipage}[t]{0.3\linewidth}
  \centering
\begin{tikzpicture}[scale=0.3]

\node[name=A0,draw,inner sep=2pt,circle] at (3,-4) {};
\node[name=A1,draw,inner sep=2pt,circle] at (5,0) {};
\node[name=A2,draw,inner sep=2pt,circle] at (1,0) {};
\node[name=A3,draw,inner sep=2pt,circle] at (5,2) {};
\node[name=A4,draw,inner sep=2pt,circle] at (1,4) {};
\node[name=A5,draw,inner sep=2pt,circle] at (1,6) {};
\node[name=A6,draw,inner sep=2pt,circle] at (3,8) {};
\node[name=B0,draw,inner sep=2pt,circle] at (-3,-6) {};
\node[name=B1,draw,inner sep=2pt,circle] at (-1,-4) {};
\node[name=B2,draw,inner sep=2pt,circle] at (-1,-2) {};
\node[name=B3,draw,inner sep=2pt,circle] at (-5,0) {};
\node[name=B4,draw,inner sep=2pt,circle] at (-1,2) {};
\node[name=B5,draw,inner sep=2pt,circle] at (-5,2) {};
\node[name=B6,draw,inner sep=2pt,circle] at (-3,6) {};

\draw[black,<-] (A0)--(A1);
\draw[black,<-] (A0)--(A2);
\draw[black,<-] (A1)--(A3);
\draw[black,<-] (A1)--(A5);
\draw[black,<-] (A2)--(A4);
\draw[black,<-] (A4)--(A5);
\draw[black,<-] (A3)--(A6);
\draw[black,<-] (A5)--(A6);

\draw[black,<-] (B0)--(B1);
\draw[black,<-] (B0)--(B3);
\draw[black,<-] (B1)--(B2);
\draw[black,<-] (B1)--(B5);
\draw[black,<-] (B2)--(B4);
\draw[black,<-] (B3)--(B5);
\draw[black,<-] (B4)--(B6);
\draw[black,<-] (B5)--(B6);

\draw[cyan,->] (A0)--(B0);
\draw[cyan,->] (A2)--(B2);
\draw[cyan,->] (A3)--(B3);
\draw[cyan,->] (A4)--(B4);
\draw[cyan,->] (A6)--(B6);

\end{tikzpicture}
  \end{minipage} &
  
  \begin{minipage}[t]{0.3\linewidth}
  \centering
  \begin{tikzpicture}[scale=0.3]

\node[name=A0,draw,inner sep=2pt,circle] at (3,-8) {};
\node[name=A1,draw,inner sep=2pt,circle] at (5,-4) {};
\node[name=A2,draw,inner sep=2pt,circle] at (1,-4) {};
\node[name=A3,draw,inner sep=2pt,circle] at (5,-2) {};
\node[name=A4,draw,inner sep=2pt,circle] at (1,0) {};
\node[name=A5,draw,inner sep=2pt,circle] at (1,2) {};
\node[name=A6,draw,inner sep=2pt,circle] at (3,4) {};
\node[name=B0,draw,inner sep=2pt,circle] at (-3,-6) {};
\node[name=B1,draw,inner sep=2pt,circle] at (-1,-4) {};
\node[name=B2,draw,inner sep=2pt,circle] at (-1,-2) {};
\node[name=B3,draw,inner sep=2pt,circle] at (-5,0) {};
\node[name=B4,draw,inner sep=2pt,circle] at (-1,2) {};
\node[name=B5,draw,inner sep=2pt,circle] at (-5,2) {};
\node[name=B6,draw,inner sep=2pt,circle] at (-3,6) {};

\draw[black,<-] (A0)--(A1);
\draw[black,<-] (A0)--(A2);
\draw[black,<-] (A1)--(A3);
\draw[black,<-] (A1)--(A5);
\draw[black,<-] (A2)--(A4);
\draw[black,<-] (A4)--(A5);
\draw[black,<-] (A3)--(A6);
\draw[black,<-] (A5)--(A6);

\draw[black,<-] (B0)--(B1);
\draw[black,<-] (B0)--(B3);
\draw[black,<-] (B1)--(B2);
\draw[black,<-] (B1)--(B5);
\draw[black,<-] (B2)--(B4);
\draw[black,<-] (B3)--(B5);
\draw[black,<-] (B4)--(B6);
\draw[black,<-] (B5)--(B6);

\draw[cyan,<-] (A0)--(B0);
\draw[cyan,<-] (A2)--(B2);
\draw[cyan,<-] (A3)--(B3);
\draw[cyan,<-] (A4)--(B4);
\draw[cyan,<-] (A6)--(B6);

\end{tikzpicture}
  \end{minipage}
\end{tabular}
\caption{An example of a mutation}
\label{fig:mutation}
\end{figure}

We provide another elementally property of a flip.
Before stating the proposition, we introduce an operation called \emph{BGP-reflection}.
Let $Q$ be a quiver, and $i\in Q$ be a sink (i.e., no edge starts at $i$) or a source (i.e., no edge ends at $i$).
Let $\mu_i(Q)$ be the graph obtained from $Q$ by reversing the direction of every edge incident to $i$.
We call $\mu_i$ the BGP-reflection on $i$.

\begin{proposition}
The flip $L\xmapsto[]{\mu_{(A,B)}}L'$ can be written by a finite sequence of BGP-reflections from $\Hasse(L)$ to $\Hasse(L')$.
\begin{proof}
  Take $A=\{a_1,a_2,a_3,\cdots, a_k\}$ such that $i\lneq j \Rightarrow a_i\ngeq a_j$.
  We prove by mathematical induction that BGP-reflection can be performed from $a_1$ through $a_k$.
  Base case:
  We verify that the statement holds for $i=1$.
  We can apply a BGP-reflection to $a_1$ since $a_1$ is minimal.
  Induction hypothesis:
  Assume that the statement holds for $i\leq j$. We show that the statement also holds for $i=j+1$.
  After applying a BGP-reflection to $a_j$, $a_{j+1}$ is sink because all outgoing edges from $a_{j+1}$ on $L$ is reversed, whereas incoming edges to $a_{j+1}$ remain unchanged.
  After all operations are completed, only the edges between $A$ and $B$ reverse their direction.
\end{proof}
\begin{figure}
\centering
\begin{tabular}{ccc}
\begin{minipage}{0.3\linewidth}
\centering
    \begin{tikzpicture}[scale=0.4]
\node[draw,circle,inner sep=2pt] (A0) at (0,0) {};
\node[draw,circle,inner sep=2pt] (Ap) at (-2,2) {};
\node[draw,circle,inner sep=2pt] (Aq) at (2,2) {};
\node[draw,circle,inner sep=2pt] (A1) at (0,4) {};
\node[draw,circle,inner sep=2pt] (B0) at (5,2) {};
\node[draw,circle,inner sep=2pt] (Bp) at (3,4) {};
\node[draw,circle,inner sep=2pt] (Bq) at (7,4) {};
\node[draw,circle,inner sep=2pt] (B1) at (5,6) {};

\draw[<-] (A0)--(Ap);
\draw[<-] (A0)--(Aq);
\draw[<-] (Ap)--(A1);
\draw[<-] (Aq)--(A1);
\draw[<-] (B0)--(Bp);
\draw[<-] (B0)--(Bq);
\draw[<-] (Bp)--(B1);
\draw[<-] (Bq)--(B1);

\draw[<-,cyan] (A0)--(B0);
\draw[<-,cyan] (Ap)--(Bp);
\draw[<-,cyan] (Aq)--(Bq);
\draw[<-,cyan] (A1)--(B1);
    
  \end{tikzpicture}
  
  \end{minipage}&
  \begin{minipage}{0.3\linewidth}
  \centering
    \begin{tikzpicture}[scale=0.4]

\node[draw,circle,inner sep=2pt,fill=blue] (A0) at (0,0) {};
\node[draw,circle,inner sep=2pt] (Ap) at (-2,2) {};
\node[draw,circle,inner sep=2pt] (Aq) at (2,2) {};
\node[draw,circle,inner sep=2pt] (A1) at (0,4) {};
\node[draw,circle,inner sep=2pt] (B0) at (5,2) {};
\node[draw,circle,inner sep=2pt] (Bp) at (3,4) {};
\node[draw,circle,inner sep=2pt] (Bq) at (7,4) {};
\node[draw,circle,inner sep=2pt] (B1) at (5,6) {};

\draw[->,blue] (A0)--(Ap);
\draw[->,blue] (A0)--(Aq);
\draw[<-] (Ap)--(A1);
\draw[<-] (Aq)--(A1);
\draw[<-] (B0)--(Bp);
\draw[<-] (B0)--(Bq);
\draw[<-] (Bp)--(B1);
\draw[<-] (Bq)--(B1);

\draw[->,blue] (A0)--(B0);
\draw[<-,cyan] (Ap)--(Bp);
\draw[<-,cyan] (Aq)--(Bq);
\draw[<-,cyan] (A1)--(B1);
    
  \end{tikzpicture}
  
  \end{minipage}&
  \begin{minipage}{0.3\linewidth}
  \centering
    \begin{tikzpicture}[scale=0.4]

\node[draw,circle,inner sep=2pt] (A0) at (0,0) {};
\node[draw,circle,inner sep=2pt,fill=blue] (Ap) at (-2,2) {};
\node[draw,circle,inner sep=2pt] (Aq) at (2,2) {};
\node[draw,circle,inner sep=2pt] (A1) at (0,4) {};
\node[draw,circle,inner sep=2pt] (B0) at (5,2) {};
\node[draw,circle,inner sep=2pt] (Bp) at (3,4) {};
\node[draw,circle,inner sep=2pt] (Bq) at (7,4) {};
\node[draw,circle,inner sep=2pt] (B1) at (5,6) {};

\draw[<-,blue] (A0)--(Ap);
\draw[->] (A0)--(Aq);
\draw[->,blue] (Ap)--(A1);
\draw[<-] (Aq)--(A1);
\draw[<-] (B0)--(Bp);
\draw[<-] (B0)--(Bq);
\draw[<-] (Bp)--(B1);
\draw[<-] (Bq)--(B1);

\draw[->,cyan] (A0)--(B0);
\draw[->,blue] (Ap)--(Bp);
\draw[<-,cyan] (Aq)--(Bq);
\draw[<-,cyan] (A1)--(B1);
    
  \end{tikzpicture}
  \end{minipage}\\
  \begin{minipage}{0.3\linewidth}
  \centering
    \begin{tikzpicture}[scale=0.4]

\node[draw,circle,inner sep=2pt] (A0) at (0,0) {};
\node[draw,circle,inner sep=2pt] (Ap) at (-2,2) {};
\node[draw,circle,inner sep=2pt,fill=blue] (Aq) at (2,2) {};
\node[draw,circle,inner sep=2pt] (A1) at (0,4) {};
\node[draw,circle,inner sep=2pt] (B0) at (5,2) {};
\node[draw,circle,inner sep=2pt] (Bp) at (3,4) {};
\node[draw,circle,inner sep=2pt] (Bq) at (7,4) {};
\node[draw,circle,inner sep=2pt] (B1) at (5,6) {};

\draw[<-] (A0)--(Ap);
\draw[<-,blue] (A0)--(Aq);
\draw[->] (Ap)--(A1);
\draw[->,blue] (Aq)--(A1);
\draw[<-] (B0)--(Bp);
\draw[<-] (B0)--(Bq);
\draw[<-] (Bp)--(B1);
\draw[<-] (Bq)--(B1);

\draw[->,cyan] (A0)--(B0);
\draw[->,cyan] (Ap)--(Bp);
\draw[->,blue] (Aq)--(Bq);
\draw[<-,cyan] (A1)--(B1);
    
  \end{tikzpicture}
  \end{minipage}&
  \begin{minipage}{0.3\linewidth}
  \centering
    \begin{tikzpicture}[scale=0.4]

\node[draw,circle,inner sep=2pt] (A0) at (0,0) {};
\node[draw,circle,inner sep=2pt] (Ap) at (-2,2) {};
\node[draw,circle,inner sep=2pt] (Aq) at (2,2) {};
\node[draw,circle,inner sep=2pt,fill=blue] (A1) at (0,4) {};
\node[draw,circle,inner sep=2pt] (B0) at (5,2) {};
\node[draw,circle,inner sep=2pt] (Bp) at (3,4) {};
\node[draw,circle,inner sep=2pt] (Bq) at (7,4) {};
\node[draw,circle,inner sep=2pt] (B1) at (5,6) {};

\draw[<-] (A0)--(Ap);
\draw[<-] (A0)--(Aq);
\draw[<-,blue] (Ap)--(A1);
\draw[<-,blue] (Aq)--(A1);
\draw[<-] (B0)--(Bp);
\draw[<-] (B0)--(Bq);
\draw[<-] (Bp)--(B1);
\draw[<-] (Bq)--(B1);

\draw[->,cyan] (A0)--(B0);
\draw[->,cyan] (Ap)--(Bp);
\draw[->,cyan] (Aq)--(Bq);
\draw[->,blue] (A1)--(B1);
    
  \end{tikzpicture}
  \end{minipage}&
  \begin{minipage}{0.3\linewidth}
  \centering
    \begin{tikzpicture}[scale=0.4]

\node[draw,circle,inner sep=2pt] (A0) at (0,0) {};
\node[draw,circle,inner sep=2pt] (Ap) at (-2,2) {};
\node[draw,circle,inner sep=2pt] (Aq) at (2,2) {};
\node[draw,circle,inner sep=2pt] (A1) at (0,4) {};
\node[draw,circle,inner sep=2pt] (B0) at (5,2) {};
\node[draw,circle,inner sep=2pt] (Bp) at (3,4) {};
\node[draw,circle,inner sep=2pt] (Bq) at (7,4) {};
\node[draw,circle,inner sep=2pt] (B1) at (5,6) {};

\draw[<-] (A0)--(Ap);
\draw[<-] (A0)--(Aq);
\draw[<-] (Ap)--(A1);
\draw[<-] (Aq)--(A1);
\draw[<-] (B0)--(Bp);
\draw[<-] (B0)--(Bq);
\draw[<-] (Bp)--(B1);
\draw[<-] (Bq)--(B1);

\draw[->,cyan] (A0)--(B0);
\draw[->,cyan] (Ap)--(Bp);
\draw[->,cyan] (Aq)--(Bq);
\draw[->,cyan] (A1)--(B1);
    
  \end{tikzpicture}
  \end{minipage}
  
\end{tabular}
  
  \caption{Flip written by a finite sequence of BGP-reflections. Black lines represent two lattices $A,B$. The direction of cyan lines are changed.}
  \label{fig:factorization into BGP-reflection}
\end{figure}
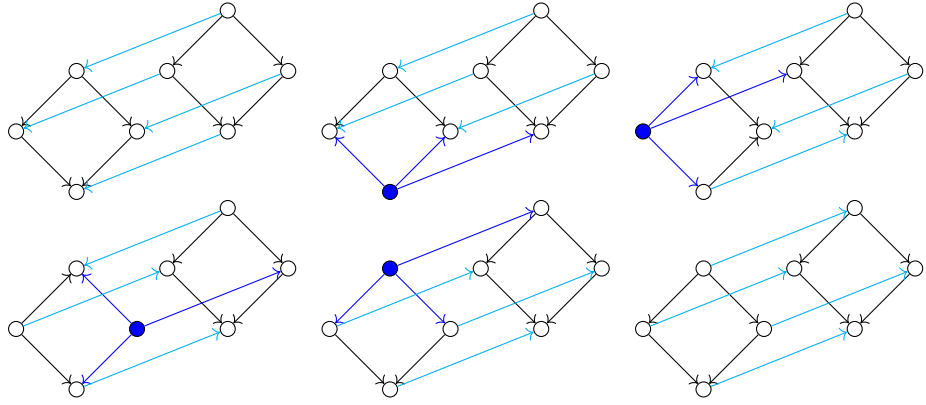
\end{proposition}

In what follows, we consider in particular the case where $L$ is a lattice.

\begin{lemma}
Let $(A,B)$ be a flip pair of a lattice $L$. Then the following hold:
\begin{enumerate}
  \item $A$ is closed under $\wedge$, and $B$ is closed under $\vee$.
  \item $A$ contains $0$, and $B$ contains $1$.
\end{enumerate}
\begin{proof} (1) Let $x,y$ be elements such that $x\in A$ and $y\leq x$. By definition of a flip, $y\in A$. A similar argument holds for $B$.
(2) If $A$ does not contain $0$, then $A$ is empty set, which contradicts to the definition of a flip. A similar argument holds for $B$.
\end{proof}
\label{A contains 0}
\end{lemma}

\begin{definition}
  Let $L\xmapsto[]{\mu_{(A,B)}}L'$ be a flip of a lattice $L$. We now introduce the following definitions:
  \begin{enumerate}
    \item $\mu_{(A,B)}$ is a \emph{mutation} if $L'$ is a lattice.
    \item $\mu_{(A,B)}$ satisfies the \emph{atom-coatom(AC) condition} if $A$ contains a coatom $a'$ and $B$ contains an atom $a$.
    \item $\mu_{(A,B)}$ satisfies the \emph{sublattice condition} if $A$ is closed under $\vee$ and $B$ is closed under $\wedge$.
    \item $\mu_{(A,B)}$ satisfies the \emph{$\partial$-sublattice condition} if $\partial A$ is closed under $\vee$ and $\partial B$ is closed under $\wedge$.
  \end{enumerate}
\end{definition}

\begin{remark}
  The term mutation is named after the mutation of a quiver (cf.~\cref{conj:mutation and quiver with potential}).
\end{remark}

\begin{proposition}
  Let $L\xmapsto[]{\mu_{(A,B)}}L'$ be a flip of a lattice $L$ satisfying the AC condition.
  Then, $a'\in \partial A$ and $a\in \partial B$.
  Here, $a$ and $a'$ are taken to satisfy the AC condition.
  \begin{proof}
    By \cref{A contains 0}, $0\prec a$ is in $\partial B$. A similar argument applies to $a'$.
  \end{proof}
\end{proposition}

\begin{proposition}
  Let $L\xmapsto[]{\mu_{(A,B)}}L'$ be a flip of a lattice $L$ satisfying the $\partial$-sublattice condition.
  Then, $\mu_{(A,B)}$ satisfies the sublattice condition.
  \begin{proof}
    We prove the sublattice condition by contradiction.
    Let $x,y\in A,\ z\in B$ satisfy $x\vee y=z$.
    For any covering chain $x=x_0,\cdots,x_n=z$, there exists $i$ such that $A\ni x_i\prec x_{i+1}\in B$.
    Then $x_i\in \partial A$. By the same argument, we obtain $y_j\in \partial A$ satisfying $y\leq y_j \leq z$.
    Hence $x\vee y \leq x_i\vee y_j \in \partial A \subset A$. Since the definition of a flip, $A$ is downward closed.
    Therefore $x\vee y\in A$, which is a contradiction.
  \end{proof}
  \label{p-sublattice implies sublattice}
\end{proposition}

The following theorem is the main theorem of this section.
We give a necessary and sufficient condition for a flip to be a mutation.
\begin{theorem}
\label{mutation iff}
  Let $\mu_{(A,B)}$ be a flip of a lattice $L$. The following are equivalent:
  \begin{enumerate}
    \item $\mu_{(A,B)}$ is a mutation.
    \item $\mu_{(A,B)}$ satisfies AC and $\partial$-sublattice conditions.
  \end{enumerate}
\end{theorem}

Hereafter, we give some lemmas to prove.
The implication $(1) \Rightarrow (2)$ will be proved in \cref{necessary condition for mutation}, and $(2) \Rightarrow (1)$ will be proved in \cref{sufficient condition for mutation}.

\begin{definition}
Let $a\in L$ be an atom, and let $a'\in L$ be a coatom.
We say that $(a,a')$ is an \emph{AC-correspondence} on $L$ if $L=\{ x \geq a | \ x \in L \} \sqcup \{ x \leq a' | \ x \in L \}$.
\end{definition}

\begin{lemma}
Let $(a,a')$ be an AC-correspondence on $L$. Then $(\{ x \leq a' | \ x \in L \},\{ x \geq a | \ x \in L \})$ is a flip pair of $L$.
\begin{proof} By definition of AC-correspondence, $a\nleq a'$.
\end{proof}
\end{lemma}

Hereafter, we take a flip pair $A=\{ x \leq a' | \ x \in L \},\ B=\{ x \geq a | \ x \in L \} $ if $(a,a')$ is an AC-correspondence, unless we mentioned.
Added to this, we denote the flip by this flip pair by $\mu_a$ or $\mu_{a'}$.

\begin{definition}
Let $(a,a')$ be an AC-correspondence on $L$. For an arbitrary $x\in L$, we define $\partial_{\downarrow}$ and $\partial_{\uparrow}$ as follows:
\[\partial_{\downarrow}(x)=(x\vee a)\wedge a',\ \partial_{\uparrow}(x)=(x\wedge a')\vee a\]
\end{definition}

\begin{lemma}
\label{mutation implies AC and sublattice}
Let $L\xmapsto[]{\mu_{(A,B)}}L'$ be a mutation. Then it satisfies AC and sublattice conditions.
\begin{proof} The sublattice condition: Since any $p,q\in A$ there exists $p\vee_{L'} q$ on $L'$, $p\vee_{L'} q$ is in $A$.
By definition of flip, $p\vee_{L'} q$ is upper than $p,q$ on $L$, hence upper than $p\vee_{L} q$ on $L$.
Therefore $p\vee_{L} q$ is in $A$ since $A$ is downward closed.

The AC condition: Let $b'$ be the minimum of $L'$ (since every finite lattice has a minimum).
$b'$ is also the minimum of $B$ on $L$.
We can take covering chain from $0$ to $b'$ on $L$, hence there exists an element of $A$ covered by $b'$.
We show there is a contradiction when $p$ is not $0$.
In this situation $b'$ do not cover $0$, then $0$ do not cover $b'$ on $L$. 
Since $0$ is minimum of $A$, if we take a covering chain from $b'$ to $0$ on $L'$, we can take $q\in B$ which covered by $0$.

However,$0\lneq p\lneq b'\leq q$ on $L$, hence $q$ cannot cover $0$ on $L$.
\end{proof}
\end{lemma}

\begin{lemma}
\label{AC and sublattice implies AC-correspondence}
Let $\mu_{(A,B)}$ be a flip of a lattice $L$ which satisfies AC and sublattice conditions.
Then $A=\{ x \leq a' | \ x \in L \},\ B=\{ x \geq a | \ x \in L \} $.
Here, $a$ and $a'$ are taken to satisfy the AC condition.
\begin{proof} If $A$ contains $p\nleq a'$, then $p\vee a'=1$ (since $a'$ is coatom). Hence $A$ contains $1$, which is a contradiction.
On the other hand, $A$ is downward closed, then any $p\leq a'$ is contained by $A$.
\end{proof}
\end{lemma}

\begin{lemma}
Let $(a,a')$ be an AC-correspondence of a lattice $L$. For arbitrary $x\in L$, $\partial_{\downarrow}(\partial_{\downarrow}(x))=\partial_{\downarrow}(x)$ and $\partial_{\uparrow}(\partial_{\uparrow}(x))=\partial_{\uparrow}(x)$ hold.
\begin{proof} Indeed, we have the following equation for all lattices:
\[(((p\vee q)\wedge r)\vee q)\wedge r=(p\vee q)\wedge r\]
We show it in detail.
Since $(p\vee q)\wedge r \leq p\vee q$ and $q \leq p\vee q$, we obtain $((p\vee q)\wedge r)\vee q\leq p\vee q$.
Since $\wedge r$ is an order homomorphism, $LHS\leq RHS$ holds.
On the other hand, since $((p\vee q)\wedge r)\vee q\geq (p\vee q)\wedge r$ and $r \geq (p\vee q)\wedge r$, we obtain $LHS\geq RHS$.
\end{proof}
\end{lemma}

\begin{lemma}
Let $(a,a')$ be an AC-correspondence of a lattice $L$. Then $\partial A=\partial_{\downarrow}(L)$ and $\partial B=\partial_{\uparrow}(L)$ hold.
\begin{proof} For any $p\in \partial A$, we can take $q\in B$ that covers $p$.
Since $p\lneq p\vee a \leq q$, it follows that $p\vee a=q$.
The same argument shows $p=q\wedge a'$, and thus $\partial_{\downarrow}(p)=p, \partial_{\downarrow}(q)=p$.
On the other hand, $\partial_{\downarrow}(x)=a'\wedge (a\vee x)$ holds for any $x\in L$.
We take an arbitrary covering chain from $a\vee x$ to $\partial_{\downarrow}(x)$.
In this chain, if an element $p\in A$ precedes to $\partial_{\downarrow}(x)$, then $p$ is lower than $a'$.
Hence, $p$ is lower than $a'\wedge (a\vee x)$, which contradicts $\partial_{\downarrow}(x)=a'\wedge (a\vee x)$.
Therefore $\partial_{\downarrow}(x)$ is covered by an element of $B$.
\end{proof}
\label{AC-correspondence and fault plane}
\end{lemma}

\begin{lemma}
  Let $(a,a')$ be an AC-correspondence of a lattice $L$, and let $x\in \partial A$. Then $\partial_{\downarrow}(x)=x$.
  \begin{proof}
    By \cref{AC-correspondence and fault plane}, we can take an element $y$ such that $x=f(y)$. Thus, $f(x)=f(f(y))=f(y)=x$.
  \end{proof}
  \label{partial on fault plane}
\end{lemma}

\begin{lemma}
Let $(a,a')$ be an AC-correspondence of a lattice $L$. Then $\vee a$ is a bijective order homomorphism from $\partial A$ to $\partial B$, and $\wedge a'$ is a bijective order homomorphism from $\partial B$ to $\partial A$. Moreover, these two morphisms are inverse to each other, then $\partial A,\ \partial B$ are order isomorphic.
\begin{proof} $\partial_{\downarrow}(L)\vee a \ \subset \  \partial_{\uparrow}(L),\ \partial_{\uparrow}(L)\wedge a' \ \subset \  \partial_{\downarrow}(L),\ \partial_{\downarrow}(\partial_{\downarrow}(L))=\partial_{\downarrow}(L)$ leads to bijective.
Also, from elementally properties of lattices, these are order homomorphisms. \cref{partial on fault plane} shows that these are inverse mappings.
\end{proof}
\label{order isomorphic fault planes}
\end{lemma}

\begin{lemma}
  Let $(a,a')$ be an AC-correspondence of a lattice $L$, and let $x \in \partial A,$ and $y \in \partial B$. Then the following are equivalent:
  \begin{enumerate}
    \item $x\leq y$
    \item $x\leq \partial_{\downarrow}(y)$
    \item $\partial_{\uparrow}(x)\leq y$
  \end{enumerate}
  \begin{proof}
    (1) implies (2) and (3), since $\partial_{\uparrow},\partial_{\downarrow}$ are order homomorphisms.
    Conversely, (2) or (3) implies (1), since $x\leq \partial_{\uparrow}(x)$ and $\partial_{\downarrow}(y)\leq y$.
  \end{proof}
\end{lemma}

\begin{lemma}
\label{join and fault plane}
  Let $(a,a')$ be an AC-correspondence of a lattice $L$, and let $x\in \partial A$. Then, $x\prec x\vee a$. Moreover, $x\vee a$ is the only element in $\partial B$ that covers $x$. The dual statement also holds.
  \begin{proof}
    Let us choose an arbitrary covering chain $x_0,\cdots, x_n$ from $x$ to $x\vee a$.
    There exists $i$ satisfying $x_i\in A,x_{i+1}\in B$.
    In this situation, $x_{i+1}\geq x,a$. Then $x\vee a \leq x_{i+1}\geq x\vee a$ and hence $x_{i+1}=x\vee a$.
    On the other hand, $(x\vee a)\wedge a'=x$ holds. Since $x_i\leq a',x\vee a$, it follows that $x\leq x_i\leq x$. Hence $x_i=x$.
    Therefore $x=x_i\prec x_{i+1}=x\vee a$.
    The dual statement is proved similarly.
  \end{proof}
\end{lemma}

\begin{lemma}
\label{AC-correspondence and semilattice}
Let $(a,a')$ be an AC-correspondence of a lattice $L$. Then $\partial A$ is closed under $\wedge$, and $\partial B$ is closed under $\vee$.
\begin{proof}
By \cref{AC-correspondence and fault plane}, we have $\partial A=\partial_{\downarrow}(L)$.
We show that $\partial A$ is closed under $\wedge$ by the following computation.
\begin{align*}
((x\vee a)\wedge a')\wedge ((y\vee a)\wedge a')=& ((x\vee a)\wedge (y\vee a))\wedge a' \\
=& (((x\vee a)\wedge (y\vee a))\vee a)\wedge a'
\end{align*}
Similar arguments show that $\partial B$ is closed under $\vee$.
\end{proof}
\end{lemma}

\begin{lemma}
Let $(a,a')$ be an AC-correspondence of a lattice $L$, and let $x\in A$. Then $\partial_{\downarrow}(x)$ is the least element in $\partial A$ that is greater than or equal to $x$. The dual statement also holds.
\begin{proof} Since $a'\geq x$ and since $\partial A$ is closed under $\wedge$, we can take $y$ to be the least element in $\partial A$ that is greater than or equal to $x$.
Since $x\ \leq \ a\vee x,\ a'$, we obtain $x\leq \partial_{\downarrow}(x)$. Then $y\leq \partial_{\downarrow}(x)$.
On the other hand, since $\partial_{\downarrow}$ is order homomorphism and $x\leq y$, it follows that $\partial_{\downarrow}(x)\leq \partial_{\downarrow}(y)=y$.
The dual statement is proved similarly.
\end{proof}
\label{fault plane and supremum}
\end{lemma}

\begin{lemma}
  Let $L$ be a lattice, and let $x\in L$. Then, \[x\in \partial A \Leftrightarrow x\prec x\vee a.\] The dual statement also holds.
  \begin{proof}
    $(\Rightarrow)$ follows from \cref{join and fault plane}. $(\Leftarrow)$ : Since $x\ngeq a$, $x\in A$. By \cref{fault plane and supremum}, $x\leq \partial_{\downarrow}(x) \lneq x\vee a$.
    The dual statement is proved similarly.
  \end{proof}
\end{lemma}

The following proposition shows the implication $(1) \Rightarrow (2)$ in \cref{mutation iff}.

\begin{proposition}
\label{necessary condition for mutation}
Let $L\xmapsto[]{\mu_{(A,B)}}L'$ be a mutation. Then it satisfies AC and $\partial$-sublattice conditions.
\begin{proof} By \cref{mutation implies AC and sublattice,AC and sublattice implies AC-correspondence}, we obtain AC, sublattice conditions. Moreover, we can take an AC-correspondence.
For arbitrary $u,v\in \partial A$, $w=u\vee v$ is in $A$.
Take $u',v'\in \partial B$ satisfying $u\prec u'$ and $v\prec v'$.
By \cref{join and fault plane}, the elements $u'$ and $v'$ are uniquely determined.
Denote $z'=u'\vee_{L} v'$. Then $w\leq_{L} z'$.
Since $z'\geq_{L'} u',v'$, it follows that $u'\vee_{L'} v'\leq_{L'} z'$.
Hence, $u'\vee_{L'} v' \in B$.
Since $w\geq_{L'} u,v$, it follows that $w\geq_{L'} u',v'$.
Then, $w\wedge_{L'} z'$ is also greater than or equal to $u',v'$ and is in $B$.
An element greater than or equal to $u',v'$ and in $B$ is greater than or equal to $z'$. Thus $w\wedge_{L'} z'\geq_{L'} z'$, and hence $w\geq_{L'} z'$.
Since $w\leq_{L} z'$, it follows that $z'$ covers $w$ on $L$.
Therefore, $w\in \partial A$, hence the $\partial$-sublattice condition holds.
\end{proof}
\end{proposition}

The following proposition shows the implication $(2) \Rightarrow (1)$ in \cref{mutation iff}.

\begin{proposition}
\label{sufficient condition for mutation}
Let $L\xmapsto[]{\mu_{(A,B)}}L'$ be a filp of a lattice $L$ satisfying AC and $\partial$-sublattice conditions. Then $L'$ is a lattice.
\begin{proof} We show that it is possible to define $x\vee_{L'} y$.
By \cref{p-sublattice implies sublattice,AC and sublattice implies AC-correspondence}, we can take an AC-correspondence $(a,a')$.
Moreover, the $\partial$-sublattice condition and \cref{AC-correspondence and semilattice} implies that $\partial A,\partial B$ are lattices.
Then, $\partial A$ and $\partial B$ are isomorphic lattices.

Case1: $x,y \in A$ 

In this case, an element $w$ satisfying $x,y\leq_{L'} w$ is not in $B$.
Then, $x,y\leq_{L} w$, hence $x\vee_{L} y \leq_{L} w$. On the other hand, $x,y\leq_{L'}, x\vee_{L} y$.
Thus, $x\vee_{L'} y=x\vee_{L} y$.

Case2: $x \in A, \ y \in B$ 

Since an element $w$ satisfying $x\leq_{L'} w$ is in $A$, an element $w'$ satisfying $x,y\leq_{L'} w'$ is also in $A$.
For any $u\in A$ with $y\leq_{L'} u$, consider a covering chain from $y$ to $u$ on $L'$.
In this chain, there exists $i$ such that $z_i\in B$ and $z_{i+1}\in A$. Thus $z_i \in \partial B$.
Hence there exists an element of $\partial B$ which is upper than $y$ and lower than $u$.

By the $\partial$-sublattice condition, $\partial B$ is closed under $\wedge_{L}$.
By the same argument as in Case1, $\partial B$ is also closed under $\wedge_{L'}$.
We can take $s$ to be the least among the elements of $\partial B$ that are greater than or equal to $y$.
Take an arbitrary covering chain from $s$ to $u$ on $L'$.
In this chain, there exists $i$ such that $z_i\in B$ and $z_{i+1}\in A$.
Let $p=z_{i+1}$ and let $q=z_{i}$. Then $p\in \partial A$.
Since $\partial A$ and $\partial B$ are order isomorphic, $p\geq_{L} \partial_{\downarrow}(s)$.
For $r\in \partial A$, $r\geq_{L} s$ is equivalent to $r\geq_{L} \partial_{\downarrow}(s)$.
Hence $x\vee_{L'} y=x\vee_{L} \partial_{\downarrow}(s)$, and we are reduced to Case1.

Case3: $x,y \in B$ 

We show that in this case $x\vee_{L'} y=x\vee_L y$.
Take an arbitrary $u \geq_{L'} x,y$. If $u\in B$, then $x\vee_{L} y \leq_{L'} u$.
Hence we may assume $u\in A$.
By the same argument as in Case2, we can take $s'\ (resp.~t')$ to be the least among the elements of $\partial B$ that are greater than or equal to $x\ (resp.~y)$.
Then, it follows from \cref{order isomorphic fault planes} that $s=\partial_{\downarrow}(s')\ (resp.~t=\partial_{\downarrow}(t'))$ is the least among the elements of $\partial A$ that are greater than or equal to $x\ (resp.~y)$ on $L'$.
By the $\partial$-sublattice condition, $s\vee t\in \partial A$.
Since $u\geq_{L'} s\vee t$, $u\geq_{L'} \partial_{\uparrow}(s\vee t)$.

By \cref{order isomorphic fault planes}, $\partial A$ and $\partial B$ are order isomorphic. Thus, $\partial_{\uparrow}(s\vee t) \geq_{L'} \partial_{\uparrow}(s),\partial_{\uparrow}(t)$.
Since $\partial_{\uparrow}(s\vee t),\partial_{\uparrow}(s),\partial_{\uparrow}(t)$ are all in $B$, $\partial_{\uparrow}(s\vee t) \geq_{L} \partial_{\uparrow}(s)\vee_{L} \partial_{\uparrow}(t)$.
Then $\partial_{\uparrow}(s\vee t) \geq_{L} x\vee_{L} y$ holds, and hence $u \geq_{L'} x\vee_{L} y$.
Therefore, $x\vee_{L'} y=x\vee_L y$.
\end{proof}
\end{proposition}

The following two propositions are important properties of mutations which immediately follow from the above lemmas and propositions.

\begin{proposition}
Let $L\xmapsto[]{\mu_{(A,B)}}L'$ be a mutation. Then $A,B,\partial A,\partial B$ are sublattices of $L$.
\begin{proof} It follows from \cref{A contains 0,mutation implies AC and sublattice,AC and sublattice implies AC-correspondence,AC-correspondence and semilattice,necessary condition for mutation}.
\end{proof}
\label{mutation and sublattice}
\end{proposition}

\begin{proposition}
\label{fault plane and lattice isomorphism}
Let $\mu_{(A,B)}$ be a mutation. Then $\partial A,\partial B$ are lattice isomorphic as sublattices of $L$.
\begin{proof} By \cref{order isomorphic fault planes,mutation and sublattice}, these two are sublattices of $L$ and are order isomorphic.
\end{proof}
\end{proposition}

In the rest of this section, we discuss the behavior of mutation under taking quotients and ideals which will be used in later sections.

\begin{lemma}
\label{prerequisite for quotient flip}
  Let $(a,a')$ be an AC-correspondence of a lattice $L$, and let $f:L\rightarrow M$ be a surjective lattice homomorphism. Then the following are equivalent:
  \begin{enumerate}
    \item $f(A) \cap f(B) = \emptyset$
    \item $f(a) \neq f(0)$
    \item $f(a') \neq f(1)$
    \item $f(a) \nleq f(a')$
  \end{enumerate}
  \begin{proof}
    It is clear that (1) $\Rightarrow$ (2) and (3).
    (2) $\Rightarrow$ (3) We prove the contrapositive. Applying $\wedge f(a)$ to both sides of the assumption $f(a')=f(1)$, we obtain $f(0)=f(a)$. A similar argument applies to (3) $\Rightarrow$ (2).
    (2) and (3) $\Rightarrow$ (4) We prove the contrapositive. Applying $\wedge f(a)$ to both sides of the assumption $f(a)\leq f(a')$, we obtain $f(a)\leq f(0)$. Since $f$ is order homomorphism, $f(a)=f(0)$.
    (4) $\Rightarrow$ (1) $f(a')$ is the supremum of $f(A)$, and $f(a)$ is the infimum of $f(B)$. Our assumption says that there exists no element $x$ such that $f(a) \leq x \leq f(a')$. Therefore $f(A) \cap f(B) = \emptyset$.
  \end{proof}
\end{lemma}

\begin{proposition}
\label{quotient flip}
Let $L\xmapsto[]{\mu_{(A,B)}}L'$ be a mutation, and let $f:L\rightarrow M$ be a surjective lattice homomorphism such that $f(A) \cap f(B) = \emptyset$. Then the following hold:
\begin{enumerate}
  \item $(f(A),f(B))$ is a flip pair, and $M\xmapsto[]{\mu_{(f(A),f(B))}}M'$ is a mutation.
  \item $f(\partial A)=\partial (f(A))$ holds.
\end{enumerate}
\begin{proof} 
Since $f(A) \cap f(B) = \emptyset$, we obtain $f(a')\neq f(1)$ by \cref{prerequisite for quotient flip}.
Since $a'\prec 1$, we get $f(a')\prec f(1)$. Thus, the AC condition is satisfied.
Since $f$ is a lattice homomorphism, the image of sublattices $A,B$ are sublattices. Then the sublattice condition is satisfied.
To show the $\partial$-sublattice condition, it is sufficient to show $f(\partial A)=\partial (f(A))$.

Let us prove $f(\partial A)\subset \partial (f(A))$. 
For $x \in \partial A$, there exists $y$ such that $y \in \partial B$ and $x\prec y$.
Since $f(A) \cap f(B) = \emptyset$, $f(x)\lneq f(y)$.
For any $z$ such that $f(x)\leq f(z)\leq f(y)$, let $w=(x\vee z)\wedge y$.
Then $f(w)=f(((x\vee z)\wedge y))=(f(x)\vee f(z))\wedge f(y)=f(z)$ holds.
On the other hand, since $x\leq x\vee z,\ y$, it follows that $x\leq w\leq y$. Since $x\prec y$, $w$ is $x$ or $y$.
Therefore $f(z)=f(w)$ is $f(x)$ or $f(y)$, which shows that $f(x)\prec f(y)$.

We next prove $f(\partial A)\supset \partial (f(A))$.
For $x' \in \partial (f(A))$, there exists $y'$ such that $y' \in \partial (f(B))$ and $x'\prec y'$.
Since $f$ is surjective, there exist $x$ and $y$ such that $f(x)=x',f(y)=y'$.
Let $p=x\wedge y,\ q=x\vee y$.
Then $p\leq q,\ f(p)=f(x)\wedge f(y)=x',f(q)=f(x)\vee f(y)=y'$.
Thus, the image of a covering chain from $p$ to $q$ under $f$ is the sequence $x',x',\cdots,x',y',\cdots,y',y'$.
Then there exist $m,n$ such that $m\prec n$ and $f(m)=x',f(n)=y'$.
Since $x'\notin f(B)$ and $y'\notin f(A)$, it follows that $m\in A$ and $n\in B$.
Therefore $m\in \partial A$ and $n\in \partial B$.
\end{proof}
\end{proposition}

\begin{remark}
In the above proposition, $M'$ is not, in general, a quotient lattice of $L'$.
In particular, in general, $f: (L,\leq_{L'})\to (M,\leq_{M'})$ is not lattice homomorphism.
We show an example in \cref{sec:Flips of Cambrian lattices}, where we set $L$ be the weak order of $S_4$, $B=\{x\in L\ |\ (2,3)\leq x\}$, and $M$ be the $A_3$ Tamari lattice \cite{MR2258260}. Then, $L'\cong L$ and $M'$ is the lattice shown in \cref{fig:A2affineTamari}. By \cref{general type Ordovician lattice and the infinite case}, $M'$ is not a quotient lattice of $L'$.
\end{remark}

\begin{proposition}
  Let $L\xmapsto[]{\mu_{(A,B)}}L'$ be a mutation, and let $S=[0_L,x]_L$. There exists $S'$ such that $S'$ is an ideal of $L'$ and is isomorphic to $S$ as a set if and only if $x\in \partial B$ holds. Moreover, if $x\in \partial B$, then $S'=[0_L',x']_{L'}$ where $x'=\partial_{\downarrow}(x)$. 
  \begin{proof}
    We first show the necessity by contradiction. Since $0_{L'}\in S$ on $L$, $x\in B$.
    By assumption, $x\notin \partial B$, hence $x$ is maximal of $S'$.
    Since $x\ngeq 0_L \in A$ on $L$, $S'$ has no maximum. Hence $S'$ cannot be an ideal, which is a contradiction.
    We next show the sufficiency.
    Since the order dual of \cref{join and fault plane}, $x'\in A$ satisfies $x'\succ_{L'} x$.
    The restriction of $\mu_{(A,B)}$ to $[0_L,x]_L$ is the flip $\mu_{(A\cap [0_L,x]_L\ ,\ B\cap [0_L,x]_L)}$ and its fault plane is $\partial A \cap [0_L,x]_L$ and $\partial B \cap [0_L,x]_L $.
    We show that this flip is a mutation by using \cref{mutation iff}.
    It is easily seen that $0_L\in \partial A \cap [0_L,x]_L \subset A\cap [0_L,x]_L$ and $0_L'\in \partial B \cap [0_L,x]_L \subset B\cap [0_L,x]_L$.
    Since the intersection of two sublattices is either empty or a sublattice, sublattice and $\partial$-sublattice conditions are satisfied.
    The AC condition is shown by taking $a=0_L'$ and $a'=x'$.
  \end{proof}
  \label{mutation of ideal}
\end{proposition}

\section{Invariants of flips}
\label{sec:Invariants of flips}
Flips modify the structure of posets only slightly.
In this section, we introduce an invariant $D$ of posets that remains unchanged under flips.
Throughout this section, we assume that $L$ is a connected poset.
Here, to say that $L$ is connected means that the underlying undirected graph of $\Hasse(L)$ is connected.

\begin{definition}
  Let $Q$ be a quiver. We define a \emph{repetition quiver} $Rep(Q)$ associated with $Q$ as follows:
  \begin{enumerate}
    \item The vertices are \( \mathbf{Z} \times Q \).
    \item The edges are \(\{(m,x) \to (n,y)\ |\ (n=m , \exists x \to y \in Q) \ or\  (n=m+1 , \exists y \to x \in Q)\}\).
  \end{enumerate}
  \begin{figure}
  \centering
  \begin{tikzpicture}
\node[draw,shape=circle,inner sep=1pt] (A00) at (0,0) {};
\node[draw,shape=circle,inner sep=1pt] (A01) at (1,1) {};
\node[draw,shape=circle,inner sep=1pt] (A02) at (0,2) {};
\node[draw,shape=circle,inner sep=1pt] (A03) at (1,3) {};
\node[draw,shape=circle,inner sep=1pt] (A10) at (2,0) {};
\node[draw,shape=circle,inner sep=1pt] (A11) at (3,1) {};
\node[draw,shape=circle,inner sep=1pt] (A12) at (2,2) {};
\node[draw,shape=circle,inner sep=1pt] (A13) at (3,3) {};
\node[draw,shape=circle,inner sep=1pt] (A20) at (4,0) {};
\node[draw,shape=circle,inner sep=1pt] (A21) at (5,1) {};
\node[draw,shape=circle,inner sep=1pt] (A22) at (4,2) {};
\node[draw,shape=circle,inner sep=1pt] (A23) at (5,3) {};
\node[draw,shape=circle,inner sep=1pt] (A30) at (6,0) {};
\node[draw,shape=circle,inner sep=1pt] (A31) at (7,1) {};
\node[draw,shape=circle,inner sep=1pt] (A32) at (6,2) {};
\node[draw,shape=circle,inner sep=1pt] (A33) at (7,3) {};

\node at (-0.5,1.5) {$\cdots$};
\node at (7.5,1.5) {$\cdots$};

\draw[->] (A00)--(A01);
\draw[->] (A02)--(A01);
\draw[->] (A02)--(A03);

\draw[->,densely dotted] (A01)--(A10);
\draw[->,densely dotted] (A01)--(A12);
\draw[->,densely dotted] (A03)--(A12);

\draw[->,very thick] (A10)--(A11);
\draw[->,very thick] (A12)--(A11);
\draw[->,very thick] (A12)--(A13);

\draw[->,densely dotted] (A11)--(A20);
\draw[->,densely dotted] (A11)--(A22);
\draw[->,densely dotted] (A13)--(A22);

\draw[->] (A20)--(A21);
\draw[->] (A22)--(A21);
\draw[->] (A22)--(A23);

\draw[->,densely dotted] (A21)--(A30);
\draw[->,densely dotted] (A21)--(A32);
\draw[->,densely dotted] (A23)--(A32);

\draw[->] (A30)--(A31);
\draw[->] (A32)--(A31);
\draw[->] (A32)--(A33);

\node[draw,shape=circle,inner sep=1pt] (B00) at (0,3.5) {};
\node[draw,shape=circle,inner sep=1pt] (B01) at (1,4.5) {};
\node[draw,shape=circle,inner sep=1pt] (B02) at (0,5.5) {};
\node[draw,shape=circle,inner sep=1pt] (B03) at (1,6.5) {};
\node[draw,shape=circle,inner sep=1pt] (B10) at (2,3.5) {};
\node[draw,shape=circle,inner sep=1pt] (B11) at (3,4.5) {};
\node[draw,shape=circle,inner sep=1pt] (B12) at (2,5.5) {};
\node[draw,shape=circle,inner sep=1pt] (B13) at (3,6.5) {};
\node[draw,shape=circle,inner sep=1pt] (B20) at (4,3.5) {};
\node[draw,shape=circle,inner sep=1pt] (B21) at (5,4.5) {};
\node[draw,shape=circle,inner sep=1pt] (B22) at (4,5.5) {};
\node[draw,shape=circle,inner sep=1pt] (B23) at (5,6.5) {};
\node[draw,shape=circle,inner sep=1pt] (B30) at (6,3.5) {};
\node[draw,shape=circle,inner sep=1pt] (B31) at (7,4.5) {};
\node[draw,shape=circle,inner sep=1pt] (B32) at (6,5.5) {};
\node[draw,shape=circle,inner sep=1pt] (B33) at (7,6.5) {};

\node at (-0.5,5) {$\cdots$};
\node at (7.5,5) {$\cdots$};

\draw[->,densely dotted] (B00)--(B01);
\draw[->] (B02)--(B01);
\draw[->,densely dotted] (B02)--(B03);

\draw[->] (B01)--(B10);
\draw[->,densely dotted] (B01)--(B12);
\draw[->,very thick] (B03)--(B12);

\draw[->,densely dotted] (B10)--(B11);
\draw[->,very thick] (B12)--(B11);
\draw[->,densely dotted] (B12)--(B13);

\draw[->,very thick] (B11)--(B20);
\draw[->,densely dotted] (B11)--(B22);
\draw[->] (B13)--(B22);

\draw[->,densely dotted] (B20)--(B21);
\draw[->] (B22)--(B21);
\draw[->,densely dotted] (B22)--(B23);

\draw[->] (B21)--(B30);
\draw[->,densely dotted] (B21)--(B32);
\draw[->] (B23)--(B32);

\draw[->,densely dotted] (B30)--(B31);
\draw[->] (B32)--(B31);
\draw[->,densely dotted] (B32)--(B33);

  \end{tikzpicture}
  \caption{Quivers and their repetition quivers. In this example, different quivers have the same repetition quiver.}
  \label{fig:repetition quiver}
\end{figure}
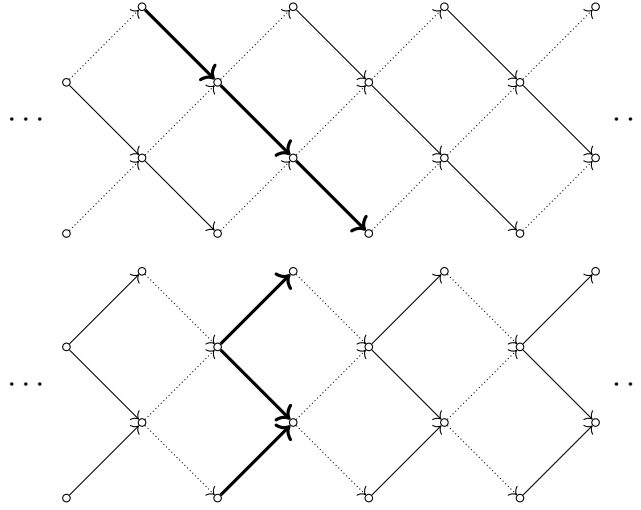
\end{definition}

\begin{lemma}
  Let $i$ be a sink or a source of a quiver $Q$, and let $Q'=\mu_i(Q)$.
  Then $Rep(Q)$ and $Rep(Q')$ are isomorphic; in particular, the following isomorphism can be defined.
\begin{equation*}
  (m,j) \mapsto
  \begin{cases}
    (m,j) & \text{if $j\neq i$,} \\
    (m-1,i) & \text{if $j=i$ and $i$ is a sink,} \\
    (m+1,i) & \text{if $j=i$ and $i$ is a source.}
  \end{cases}
\end{equation*}
\end{lemma}

In what follows, we consider in particular the case where $Q=\Hasse(L)$ for some poset $L$.
Then, $Rep(Q)$ is also a poset.

Let $x,y\in L$. We define $d(x,y)$ as follows:
\[d(x,y)=\min\{i\in \mathbf{Z}_{\geq 0} \ | \ (0,x) \leq (i,y)\}\]

\begin{remark}
  $d(x,y)$ is also the minimum number of descending covers in covering path from $x$ to $y$. 
\end{remark}

\begin{lemma}
  The following hold:
  \begin{enumerate}
    \item $d(x,y)=0 \Leftrightarrow x\leq y$
    \item $d(x,y)=d(y,x)=0 \Leftrightarrow x=y$
    \item $d(x,z)\leq d(x,y)+d(y,z)$
    \item $d(x,y)=0,d(y,x)=1 \Leftrightarrow x\prec y$
  \end{enumerate}
  \begin{proof}
    (1) and (2) are clear.
    (3) follows by concatenating a covering path from $x$ to $y$ with a covering path from $y$ to $z$.
    (4) $\Leftarrow$ is clear. We now show $\Rightarrow$.
    By assumption, $x\lneq y$.
    There exists a covering path $y=x_0,\cdots,x_n=x$ that has exactly one descending cover $x_i\succ x_{i+1}$.
    Then $x_{i+1}\leq x\lneq y\leq x_i$ holds; hence $x_{i+1}=x,x_i=y$ follows from $x_i\succ x_{i+1}$.
  \end{proof}
  \label{basic properties of d}
\end{lemma}

By (1) of \cref{basic properties of d}, $d(x,y)$ recovers the poset structure on $L$.

Let $x,y\in L$. We say that $x$ and $y$ have a \emph{cover relationship} when $x\prec y$ or $x\succ y$ holds.
By (4) of \cref{basic properties of d}, $x$ and $y$ have a cover relationship if and only if $d(x,y)+d(y,x)=1$.

\begin{definition}
  Let $L\xmapsto[]{\mu_{(A,B)}}L'$ be a flip. We call $d_{L'}$ as the \emph{flip} of $d_{L}$ by a flip pair $(A,B)$.
\end{definition}

\begin{proposition}
\label{potential epsilon}
  The following hold:
  \begin{enumerate}
    \item $d_{L'}(x,y)-d_{L}(x,y)=(\epsilon(y)-\epsilon(x))/2$, where $\epsilon$ is defined as follows:
    \[\epsilon(z)=
    \begin{cases}
      -1 & (z \in A) \\
      +1 & (z \in B)
    \end{cases}\]
    \item $d_{L'}(x,y)+d_{L'}(y,z)-d_{L'}(x,z)=d_{L}(x,y)+d_{L}(y,z)-d_{L}(x,z)$ \\
    In particular, $d_{L}(x,y)+d_{L}(y,z)-d_{L}(x,z)$ is unchanged under flips.
  \end{enumerate}
  \begin{proof}
    (1) For any covering path $x_0,\cdots,x_n$, the following equation holds.
    \begin{align*}
      &\Big( (\text{number of descending covers in } L')\\
      &\quad - (\text{number of descending covers in } L)\Big) \\
      &=\#\{i\ |\ x_i\in A,x_{i+1}\in B\}-\#\{i\ |\ x_i\in B,x_{i+1}\in A\} \\
      &=(\epsilon(y)-\epsilon(x))/2
    \end{align*}
    
    (2) straightforwardly follows from (1).
  \end{proof}
\end{proposition}

\begin{proposition}
\label{AC-correspondence and the map u}
  Let $(a,a')$ be an AC-correspondence on $L$. Then, for any $x\in L$, $d(a,a')=1=d(a,x)+d(x,a')$.
  \begin{proof}
    If $x\in A$, $d(a,x)+d(x,a')=d(a,x)$. Since $a\nleq x$ and $d(a,x)\leq d(a,0)+d(0,x)=1$, it follows that $d(a,x)=1$.
    Substituting $a'$ for $x$, we obtain $d(a,a')=1$.
    A similar argument applies to $x\in B$.
  \end{proof}
\end{proposition}

\begin{definition}
\label{The invariant D}
  Let $x,y,z\in L$. We define $D_L(x,y,z)$ by 
  \[d_{L}(x,y)+d_{L}(y,z)-d_{L}(x,z).\]
\end{definition}

\begin{lemma}
\label{D recovers d}
  Let $L$ be a connected poset. The following hold:
  \begin{enumerate}
    \item $D_L(x,y,z)\geq 0$
    \item For any flip $L'=\mu(L)$, $D_L(x,y,z)=D_{L'}(x,y,z)$ holds.
    \item If $L$ has the least element $0$, $D_L(0,y,z)=d(y,z)$ holds.
  \end{enumerate}
  \begin{proof}
    (1) follows from \cref{basic properties of d}.
    (2) follows from \cref{potential epsilon}.
    (3) follows from the definition.
  \end{proof}
\end{lemma}

The following theorem states that the poset structure obtained from $L$ by a finite sequence of flips is uniquely determined once the least element is specified.

\begin{theorem}
\label{flip and minimum}
  Let $L$ be a connected poset, and let $x\in L$.
  Then there uniquely exists a poset $L'$ obtained from $L$ by a finite sequence of flips whose least element is $x$.
  \begin{proof}
    We first say that there exists at most one $L'$ satisfying the assumptions in the theorem.
    By \cref{D recovers d}(3), we obtain $d_{L'}(y,z)=D_{L'}(x,y,z)=D_{L}(x,y,z)$, which recovers $L'$.
    We next prove the existence of $L'$. Let $S_L=\{w\in L| \ x\nleq w \}$.
    If $S_L=\emptyset$, this case is clear.
    Then we may assume that $S_L \neq \emptyset$.
    We can take $w'\in S_L$ and a covering path $x_0,\cdots,x_n$ from $x$ to $w'$.
    Then, by \cref{basic properties of d}, $-1\leq d(x,x_i+1)-d(x,x_i)\leq 1$. Hence there exist $u\prec v$ such that $d(x,u)=1,d(x,v)=0$ hold.
    We can take a flip pair $(A,B)$, where $B=\{w\in L|\ w\geq x\}$ and $A=L\setminus B$.
    Let $L_1$ be a flip of $L$ by $(A,B)$.
    Then $S_{L_1}\subset S_L$, $u\in S_L$, and $u\notin S_{L_1}\coloneqq \{w\in L_1 | \ x\nleq w \}$ hold.
    Thus, $\#\{S_{L_1}\} \lneq \#\{S_{L}\}$.
    If $\#\{S_{L_1}\}=0$, then the statement holds.
    Otherwise, by the same argument, we obtain $S_{L_2}$, and hence $\#\{S_{L_2}\} \lneq \#\{S_{L_1}\}$.
    Therefore, by repeating this operation, we obtain $S_{L_m}=\emptyset$ for some $m$, and hence the statement holds.
  \end{proof}
\end{theorem}

\begin{corollary}
\label{D and sequence of flips}
  Let $L_1=(L,\leq_{L_1})$ and $L_2=(L,\leq_{L_2})$ be connected posets with the underlying set $L$. Then the following are equivalent:
  \begin{enumerate}
    \item $L_2$ is obtained from $L_1$ by applying a finite sequence of flips.
    \item $D_{L_1}(x,y,z)=D_{L_2}(x,y,z)$ for all $x,y,z\in L$.
  \end{enumerate}
  \begin{proof}
    The implication (1) $\Rightarrow$ (2) is clear by (2) of \cref{D recovers d}.
    We will prove (2) $\Rightarrow$ (1). Fix $x\in L$. From previous theorem, there uniquely exists $L'_1=(L,\leq_{L'_1})$ such that $L'_1$ is obtained from $L_1$ by applying a finite sequence of flips and such that its least element is $x$. 
    By the same argument, there uniquely exists $L'_2=(L,\leq_{L'_2})$ such that $L'_2$ is obtained from $L_2$ by applying a finite sequence of flips and such that its least element is $x$.
    By \cref{D recovers d}(3), $D_{L'_1}$ and $D_{L'_2}$ determine $L'_1$ and $L'_2$ respectively.
    Then, since $D_{L'_1}=D_{L'_2}$, we obtain $L'_1=L'_2$.
    Therefore, by concatenating a finite sequence of flips from $L_1$ to $L'_1$ and that from $L'_2=L'_1$ to $L_2$, we obtain a finite sequence of flips from $L_1$ to $L_2$. 
  \end{proof}
\end{corollary}

\begin{definition}
  Let $L$ be a poset with the least element $0_L$, and let $a\in L$ be an atom. We define a \emph{flip on (an) atom} $\mu_{a}$ to be a flip $\mu_{(A,B)}$, where $B=\{x\geq a\}$ and $A=L\setminus B$.
\end{definition}

\begin{lemma}
  Let $L$ be a poset with the least element $0_L$, and let $a\in L$ be an atom. Then $L'=\mu_{a}(L)$ is a poset with the least element $a$.
  \begin{proof}
    We show that $x\geq_{L'} a$ holds for all $x\in L$. 
    If $x\in A$, $a\leq_{L'}0_L\leq_{L'}x$. The case $x\in B$ is clear.
  \end{proof}
  \label{flip on atom and minimum one}
\end{lemma}

\begin{lemma}
  Let $L_1 \xmapsto[]{\mu_{(A,B)}} L_2$ be a flip. If both $L_1$ and $L_2$ have the least element, then $\mu_{(A,B)}$ is a flip on atom. 
  \begin{proof}
    Since the least element of $L_1$ is also the least element of $A$, we denote it by $0_A$; similarly, since the least element of $L_2$ is also the least element of $B$, we denote it by $0_B$.
    Thus, we obtain $0_A \prec_{L_1} 0_B$, and hence $B=\{x\geq_{L_1} 0_B\}$.
  \end{proof}
  \label{flip on atom and minimum two}
\end{lemma}

\begin{definition}
  Let $L$ be a poset. We define a graph $G(L)$ as follows:
  \begin{itemize}
    \item A vertex is labelled by $L'$ such that $L'$ is obtained from $L$ by applying a finite sequence of flips and such that $L'$ has the least element. 
    \item An edge joins two vertices $L_1$ and $L_2$ if a flip $L_1 \xmapsto[]{\mu_{(A,B)}} L_2$ exists.
  \end{itemize}
  \label{the flip graph G}
\end{definition}

\begin{theorem}
  Let $L$ be a connected poset with the least element.
  Then $G(L)$ is isomorphic to the underlying undirected graph of $\Hasse(L)$.
  \begin{proof}
    It follows from \cref{flip and minimum,flip on atom and minimum one,flip on atom and minimum two}.
  \end{proof}
  \label{flip graph and Hasse quiver}
\end{theorem}

\section{Mutable lattices}
\label{sec:Mutable lattices}
In this section, we define mutable lattices and study their properties.
In particular, we will show that mutable lattices are semidistributive (\cref{mutable implies semidistributive}).

\begin{definition}
\label{locally mutable}
  Let $L$ be a lattice. We say that $L$ is \emph{locally mutable} if all of the following conditions hold:
  \begin{enumerate}
    \item For any atom $a$, there exists a coatom $a'$ such that $(a,a')$ is AC-correspondence.
    \item For any coatom $a'$, there exists an atom $a$ such that $(a,a')$ is AC-correspondence.
    \item For any AC-correspondence $(a,a')$, $\mu_{(a,a')}$ is a mutation.
  \end{enumerate}
\end{definition}

\begin{definition}
\label{mutable}
  Let $L$ be a lattice. We say that $L$ is \emph{mutable} if all lattices obtained from $L$ by a finite sequence of mutations are locally mutable lattices.
\end{definition}

\begin{lemma}
\label{locally mutable and AC-correspondence}
  Let $L$ be a locally mutable lattice. Then the following hold:
  \begin{enumerate}
    \item AC-correspondence $(a,a')$ gives a bijection $a\mapsto a'$ between atoms and coatoms.
    \item $1_L$ is the supremum of all atoms.
    \item For any $x,y\in L$, $x\vee y=1_L \Leftrightarrow y\geq \bigvee \{a\in L\ |\ a\nleq x, atom\}$ holds.
  \end{enumerate}
  \begin{proof}
    (1) If two coatoms correspond to atom $a$, then $A$ has two coatoms but does not have $1$.
    Hence, $A$ is not sublattice, which contradicts the sublattice condition.
    Thus, $a\mapsto a'$ is injective, and hence coatoms are at least as much as atoms.
    A similar argument shows that atoms are at least as much as coatoms.
    Therefore $a\mapsto a'$ is bijective.
    
    (2) We proceed by contradiction. By assumption, there exists $x\neq 1$ such that $x$ is more than or equal to all atoms.
    Hence, there exists a coatom $a'$ such that $a'$ is more than or equal to all atoms.
    However, in this case, we cannnot take atom $a$ such that $(a,a')$ is an AC-correspondence, which is a contradiction.
    
    (3) $(\Leftarrow)$ follows from (2).
    $(\Rightarrow)$ We prove the contrapositive. Let $a$ be an atom, and $a'$ is a coatom such that $(a,a')$ is an AC-correspondence. Then, for any $a\nleq x,y$, $x\vee y \leq a' \ngeq a$ holds.
    Thus, $a\nleq x\vee y$ holds, and hence $x\vee y \neq 1$.
  \end{proof}
\end{lemma}

\begin{lemma}
  Let $L$ be a locally mutable lattice. Take $a,x,y\in L$ such that $a$ is an atom, $x\leq y$, $a\nleq y$ (resp.~$a\leq x$). Also we take $L'=\mu_{a}(A)$.
  Then $[x,y]_{L}$ and $[x,y]_{L'}$ are isomorphic lattices
  ; in particular, the isomorphism $L\in z\to z\in L'$ can be defined.
  \begin{proof}
    It follows from the fact that a mutation does not change the lattice structure of $A$ (resp.~$B$).
  \end{proof}
  \label{interval preserved under mutation}
\end{lemma}

\begin{lemma}
  For a lattice $L$, The following are equivalent:
  \begin{enumerate}
    \item For all $x,y,z\in L$, the following hold (which is called a \textnormal{semidistributive} lattice):
    \[ x\vee y=x\vee z=w \Rightarrow x\vee(y\wedge z)=x\vee y\ \]
    \[ x\wedge y=x\wedge z=w \Rightarrow x\wedge(y\vee z)=x\wedge y \]
    \item For any $x,y\in L$ with $x\leq y$, $U=\{z\in L\ |\ x\vee z=y\}$ has the least element (\textnormal{i.e.,} $U$ contains the element that is the infimum of $U$ on $L$) and $U'=\{z'\in L\ |\ y\wedge z'=x\}$ has the greatest element (\textnormal{i.e.,} $U'$ contains the element that is the supremum of $U'$ on $L$).
  \end{enumerate}
  \begin{proof}
    $(1)\Rightarrow (2)$ For our assumption, $U$ is closed under $\wedge$. Then the infimum of $U$(=meet of all elements in $U$) is in $U$. A similar argument applies to $U'$.
    $(2)\Rightarrow (1)$ It suffices to show that $U$ is closed under $\wedge$. Denote $0_U$ as the least element of $U$.
    For every $z,w\in U$, $0_U \leq z\wedge w \leq z$. Since $y=x\vee 0_U \leq x\vee (z\wedge w)\leq x\vee z=y$, we obtain $z\wedge w\in U$. A similar argument applies to $U'$.
  \end{proof}
  \label{semidistributive and infimum}
\end{lemma}

\begin{theorem}
\label{mutable implies semidistributive}
  A mutable lattice is semidistributive.
  \begin{proof}
  Let $L$ be a mutable lattice.
  For every $x\leq y$, we set $U=\{z\in L\ |\ x\vee z=y\}$ and $U'=\{z'\in L\ |\ y\wedge z'=x\}$. By \cref{semidistributive and infimum}, it suffices to show that $U$ has the least element and $U'$ has the greatest element.
  We only prove the former statement since the latter can be proved in the same manner.
    It follows easily that $U\subset [0_L,y]$.
    Take arbitrary covering chain $(y_0,\cdots,y_k)$ from $y$ to $1_L$.
    Let $L_0=L,\ L_{i+1}=\mu_{y_{k-i-1}}(L_i)$, where $\mu_{a'}$ on a lattice $M$ is defined as $\mu_{a}$ such that $(a,a')$ is an AC-correspondence on $M$. Then $L'\coloneqq L_k$ has the greatest element $y$. 
    In this sequence of mutations, $y$ is always in a footwall. Hence, $[0_L,y]_L$ is always in a footwall. By \cref{interval preserved under mutation}, $[0_L,y]_L'=[0_L,y]_L$.
    Then for any $z\in [0_L,y]$, $x\vee_{L}z=y \Leftrightarrow x\vee_{L'}z=y$.
    
    Since $L'$ is mutable, we can take $w\in L'$ such that $w\leq_{L'}z \Leftrightarrow x\vee_{L'}z=y$ by applying \cref{locally mutable and AC-correspondence}.
    Moreover, $0_L \leq_{L'} z$, hence $z\geq_{L'} w\vee_{L'} 0_L$. We denote $z'=w\vee_{L'} 0_L$.
    Now we show $x\vee_{L}z'=y$.
    Since $z'\in [0_L,y]_L'$, we have $z'\in [0_L,y]_L$. Hence, since $x,z'\in [0_L,y]$ and $x\vee_{L'}z'=y$ hold, $x\vee_{L} z'=y$ holds.
    This shows that the infimum of $U$ is $z'$.
  \end{proof}
\end{theorem}

\begin{remark}
  Cambrian lattices are semidistributive \cite{MR4579952}.
\end{remark}

\begin{remark}
  We give an example of a lattice that is semidistributive but not mutable.
  \cref{fig:semidistributive and locally mutable} is both semidistributive and locally mutable. However, \cref{fig:neither semidistributive nor locally mutable}, obtained from \cref{fig:semidistributive and locally mutable} by applying a mutation once, is neither semidistributive nor locally mutable.
    \begin{figure}[htbp]
\centering
\begin{tabular}{cc}
  \begin{minipage}[t]{0.3\linewidth}
  \begin{tikzpicture}[scale=0.5]

\node[name=max,draw,inner sep=2pt,circle] at (0,0) {};
\node[name=A,draw,inner sep=2pt,circle] at (0,-2) {};
\node[name=B1,draw,inner sep=2pt,circle] at (-2,-4) {};
\node[name=B2,draw,inner sep=2pt,circle] at (2,-4) {};
\node[name=C,draw,inner sep=2pt,circle] at (0,-6) {};
\node[name=D,draw,inner sep=2pt,circle] at (4,-4) {};
\node[name=min,draw,inner sep=2pt,circle] at (0,-8) {};

\draw[black] (max)--(A);
\draw[black] (max)--(D);
\draw[black] (A)--(B1);
\draw[black] (A)--(B2);
\draw[black] (B1)--(C);
\draw[black] (B2)--(C);
\draw[black] (C)--(min);
\draw[black] (D)--(min);

\end{tikzpicture}
  \subcaption{Both semidistributive and locally mutable.}
  \label{fig:semidistributive and locally mutable}
  \end{minipage} &
  
  \begin{minipage}[t]{0.3\linewidth}
  \begin{tikzpicture}[scale=0.5]

\node[name=max,draw,inner sep=2pt,circle] at (0,0) {};
\node[name=A1,draw,inner sep=2pt,circle] at (-2,-2) {};
\node[name=A2,draw,inner sep=2pt,circle] at (2,-2) {};
\node[name=B,draw,inner sep=2pt,circle] at (0,-4) {};
\node[name=C,draw,inner sep=2pt,circle] at (0,-6) {};
\node[name=D,draw,inner sep=2pt,circle] at (4,-2) {};
\node[name=min,draw,inner sep=2pt,circle] at (0,-8) {};

\draw[black] (max)--(A1);
\draw[black] (max)--(A2);
\draw[black] (max)--(D);
\draw[black] (A1)--(B);
\draw[black] (A2)--(B);
\draw[black] (B)--(C);
\draw[black] (C)--(min);
\draw[black] (D)--(min);

\end{tikzpicture}
  \subcaption{Neither semidistributive nor locally mutable.}
  \label{fig:neither semidistributive nor locally mutable}
  \end{minipage}
\end{tabular}
\caption{}
\end{figure}

\end{remark}

\begin{example}
  There are examples of mutable lattices. 
  \begin{enumerate}
    \item A polygon 
    \item $\{0,1\}^n$, where $\{0,1\}$ denotes the two-element chain and $n$ is a positive integer
    \item The weak order of a finite Coxeter group
    \item A $A_3$ Cambrian lattice
    \item A $B_3$ Cambrian lattice
  \end{enumerate}
\end{example}

The cases (1),(2), and (3), these lattices remain isomorphic to their original forms regardless of how a mutation is applied.
The proof of the case (3) will be given in \cref{sec:Flips of the weak order of finite Coxeter groups}.

The cases (4) and (5), however, it may be mutated into a lattice that is not isomorphic to the original lattice(\cref{fig:mutation graph of A3 Ordovician lattice,fig:mutation graph of B3 Ordovician lattice}).

The case (1) includes $A_2$ and $B_2$ Cambrian lattices, and (2) includes $A_1$ Cambrian lattice.
Hence, $A_1,A_2,A_3,B_2,B_3$ Cambrian lattices are mutable.

\begin{conjecture}
\label{conj:mutable implies N-regular}
  Let $L$ be a mutable lattice. Then, in the underlying undirected graph of $\Hasse(L)$, the degree of every vertex is equal to the constant $N$ (which is called an \textnormal{$N$-regular graph}). 
\end{conjecture}
The converse is false (\cref{fig:BadCase1,fig:BadCase2}). 

For the following conjecture, we will later provide a more detailed statement (\cref{conj:structure of Ordovician lattice}).
\begin{conjecture}
  Let $Q$ be a Coxeter quiver corresponding to a finite Coxeter group. Then the Cambrian lattice $\Camb(Q)$ associated with $Q$ is mutable.
\end{conjecture}
From the viewpoint of representation theory, we propose another conjecture:
\begin{conjecture}
\label{conj:mutation and quiver with potential}
  Let $(Q,W)$ be a quiver with potential where $Q$ is a quiver such that $Q$ has no loops or 2-cycles. 
  Denote by $\mathcal{P}(Q,W)$ the Jacobian algebra associated with $(Q,W)$.
  Assume that $\dim (\mathcal{P}(Q,W)) < \infty$.
  Then there exists a mutation that maps $L=\Tors(\mathcal{P}(Q,W))$ to a lattice isomorphic to 
$L'=\Tors(\mathcal{P}(\mu_i(Q,W)))$.
Here, $\Tors(A)$ denotes the lattice of torsion classes of the category of finitely generated left $A$-modules, and $\mu_i(Q,W)$ is the mutation of $(Q,W)$ at a vertex $i$ of $Q$ \textnormal{\cite{MR2480710}}.
\end{conjecture}

The results in the rest of this section will not be used in this paper, but we include them since they might be of independent interest.

\begin{proposition}
  Let $L$ be a mutable lattice. Then the following hold:
  \begin{enumerate}
    \item For all $x\in L$, there exists exactly one $u(x)\in L$ such that $D(x,y,u(x))=0$ for all $y\in L$.
    \item $u(0)=1$. Moreover, for any AC-correspondence$(a,a')$, $u(a)=a'$.
  \end{enumerate}
  \begin{proof}
    (1) Since $L$ is a finite lattice, $\Hasse(L)$ is connected. Thus, by \cref{flip and minimum}, we obtain $L'$ such that its least element is $x$ and can be obtained from $L$ by flips on an atom.
    Since $L$ is mutable, flips on an atom are always mutations.
    Thus, $L'$ is a lattice. Hence, if we take $u(x)=1_{L'}$, the result follows. 
    We next prove uniqueness by contradiction.
    By the assumption for contradiction, there exists $x$ for which there are multiple $u(x)$ satisfying the conditions of the proposition. We denote them by $u_1,u_2$.
    Then, the following equation holds:
    \[d(x,u_1)=d(x,u_2)+d(u_2,u_1)=d(x,u_1)+d(u_1,u_2)+d(u_2,u_1)\]
    Since $d(u_1,u_2)+d(u_2,u_1)=0$, we obtain $u_1=u_2$, which is a contradiction.
    (2) follows from \cref{AC-correspondence and the map u}.
  \end{proof}
  \label{the map u}
\end{proposition}

\begin{proposition}
  Let $L$ be a mutable lattice. Then the following hold:
  \begin{enumerate}
    \item The map $u:x\mapsto u(x)$ is a bijection.
    \item The map $y\mapsto (d(x,y),y)$ gives a bijection from $L$ to the interval $[(0,x),(d(x,u(x)),u(x))]$.
  \end{enumerate}
  \begin{proof}
    (1) Since $L$ has only finitely many elements, it suffices to prove injectivity. For all $x,x'$ such that $u(x)=u(x')$, the following equation holds:
    \begin{align*}
      d(x,u(x))&=d(x,x')+d(x',u(x))=d(x,x')+d(x',u(x'))\\
      &=d(x,x')+d(x',x)+d(x,u(x'))\\
      &=d(x,x')+d(x',x)+d(x,u(x))
    \end{align*}
    Therefore $x=x'$.
    
    (2) The injectivity is clear. Since $(i,y)\in [(0,x),(d(x,u(x)),u(x))]$ if and only if $i=d(x,y)$, surjectivity follows.
  \end{proof}
\end{proposition}

\begin{proposition}
  The following hold:
  \begin{enumerate}
    \item $d(u(x),u(y))+d(u(y),u(x))=d(x,y)+d(y,x)$
    \item $u(x)$ and $u(y)$ have a cover relationship if and only if $x$ and $y$ have a cover relationship.
  \end{enumerate}
  \begin{proof}
  (1) follows from the equation below.
  \begin{align*}
  &d(u(x),u(y))+d(u(y),u(x))\\
  &=d(y,u(y))-d(y,u(x))+d(x,u(x))-d(x,u(y))\\
  &=d(y,u(y))-d(x,u(x))+d(x,y)+d(x,u(x))-d(y,u(y))+d(y,x)\\
  &=d(x,y)+d(y,x)
\end{align*}
  (2) follows from (1).
  \end{proof}
\end{proposition}

\section{Flips of the weak order of finite Coxeter groups}
\label{sec:Flips of the weak order of finite Coxeter groups}
In this section, we state the weak order of a finite Coxeter group is mutable.

Hereafter, we assume that $(W,S)$ is a finite Coxeter system.
We denote the identity element of $W$ by $id$, and the longest element of $W$ by $w_0$.
For $s_i,s_j\in S$, let $m(s_i,s_j)$ denote the order of the element $s_i s_j$.
We write the length of $p\in W$ by $l(p)$. 
The Coxeter diagram of $(W,S)$ is a graph as follows;
\begin{enumerate}
  \item Each vertices are labelled by an element of $S$.
  \item For any $s_i,s_j\in S$ satisfying $m(s_i,s_j)\geq 3$, draw an edge between $s_i$ and $s_j$, labelled by $m(s_i,s_j)$, omitting the label when $m(s_i,s_j)=3$.
\end{enumerate}

\begin{example}
  We display the Coxeter diagrams corresponding to all finite Coxeter groups in \cref{fig:finite-type Coxeter quiver}.
    \begin{figure}
    \centering
    \begin{tabular}{ccc}
      \begin{minipage}[t]{0.3\linewidth}
  \centering
  \begin{tikzpicture}[scale=0.35]

\node[name=A1,draw,inner sep=2pt,circle] at (-4,0) {};
\node[name=A2,draw,inner sep=2pt,circle] at (-2,0) {};
\node[name=A3,draw,inner sep=2pt,circle] at (0,0) {};
\node[name=A4,inner sep=2pt,circle] at (2,0) {$\cdots$};
\node[name=A5,draw,inner sep=2pt,circle] at (4,0) {};

\draw[black] (A1)--(A2)--(A3)--(A4)--(A5);
\end{tikzpicture}
  \subcaption*{Type-A}
  \label{fig:type-A Coxeter quiver}
  \end{minipage}&
  
\begin{minipage}[t]{0.3\linewidth}
\centering
  \begin{tikzpicture}[scale=0.35]

\node[name=A1,draw,inner sep=2pt,circle] at (-4,0) {};
\node[name=A2,draw,inner sep=2pt,circle] at (-2,0) {};
\node[name=A3,draw,inner sep=2pt,circle] at (0,0) {};
\node[name=A4,inner sep=2pt,circle] at (2,0) {$\cdots$};
\node[name=A5,draw,inner sep=2pt,circle] at (4,0) {};

\draw[black] (A1)--(A2) node[midway, above] {4};
\draw[black] (A2)--(A3)--(A4)--(A5);
\end{tikzpicture}
  \subcaption*{Type-B}
  \label{fig:type-B Coxeter quiver}
  \end{minipage}&
  \begin{minipage}[t]{0.3\linewidth}
  \centering
    \begin{tikzpicture}[scale=0.35]

\node[name=A1,draw,inner sep=2pt,circle] at (-4,0) {};
\node[name=A2,draw,inner sep=2pt,circle] at (-2,0) {};
\node[name=A3,draw,inner sep=2pt,circle] at (-2,-2) {};
\node[name=A4,draw,inner sep=2pt,circle] at (0,0) {};
\node[name=A5,inner sep=2pt,circle] at (2,0) {$\cdots$};
\node[name=A6,draw,inner sep=2pt,circle] at (4,0) {};

\draw[black] (A1)--(A2)--(A4)--(A5)--(A6);
\draw[black] (A2)--(A3);
\end{tikzpicture}
\subcaption*{Type-D}
  \label{fig:type-D Coxeter quiver}
  \end{minipage}\\
  
\begin{minipage}[t]{0.3\linewidth}
\centering
\begin{tikzpicture}[scale=0.35]

\node[name=A1,draw,inner sep=2pt,circle] at (-4,0) {};
\node[name=A2,draw,inner sep=2pt,circle] at (-2,0) {};
\node[name=A3,draw,inner sep=2pt,circle] at (0,0) {};
\node[name=A4,draw,inner sep=2pt,circle] at (0,-2) {};
\node[name=A5,inner sep=2pt,circle] at (2,0) {$\cdots$};
\node[name=A6,draw,inner sep=2pt,circle] at (4,0) {};

\draw[black] (A1)--(A2)--(A3)--(A5)--(A6);
\draw[black] (A3)--(A4);
\end{tikzpicture}
\subcaption*{Type-E}
  \label{fig:type-E Coxeter quiver}
\end{minipage}&

\begin{minipage}[t]{0.3\linewidth}
  \centering
  \begin{tikzpicture}[scale=0.35]

\node[name=A1,draw,inner sep=2pt,circle] at (-3,0) {};
\node[name=A2,draw,inner sep=2pt,circle] at (-1,0) {};
\node[name=A3,draw,inner sep=2pt,circle] at (1,0) {};
\node[name=A4,draw,inner sep=2pt,circle] at (3,0) {};

\draw[black] (A1)--(A2);
\draw[black] (A2)--(A3) node[midway, above] {4};
\draw[black] (A3)--(A4);
\end{tikzpicture}
\subcaption*{$F_4$}
  \label{fig:type-F Coxeter quiver}
\end{minipage}&

\begin{minipage}[t]{0.3\linewidth}
  \centering
\begin{tikzpicture}[scale=0.35]

\node[name=A1,draw,inner sep=2pt,circle] at (-1,0) {};
\node[name=A2,draw,inner sep=2pt,circle] at (1,0) {};

\draw[black] (A1)--(A2) node[midway, above] {6};
\end{tikzpicture}
\subcaption*{$G_2$}
  \label{fig:type-G Coxeter quiver}
\end{minipage}\\

       \begin{minipage}[t]{0.3\linewidth}
  \centering
  \begin{tikzpicture}[scale=0.35]

\node[name=A1,draw,inner sep=2pt,circle] at (-3,0) {};
\node[name=A2,draw,inner sep=2pt,circle] at (-1,0) {};
\node[name=A3,draw,inner sep=2pt,circle] at (1,0) {};

\draw[black] (A1)--(A2) node[midway, above] {5};
\draw[black] (A2)--(A3);
\end{tikzpicture}
\subcaption*{$H_3$}
  \label{fig:type-H3 Coxeter quiver}
\end{minipage}&
       \begin{minipage}[t]{0.3\linewidth}
  \centering
  \begin{tikzpicture}[scale=0.35]

\node[name=A1,draw,inner sep=2pt,circle] at (-3,0) {};
\node[name=A2,draw,inner sep=2pt,circle] at (-1,0) {};
\node[name=A3,draw,inner sep=2pt,circle] at (1,0) {};
\node[name=A4,draw,inner sep=2pt,circle] at (3,0) {};

\draw[black] (A1)--(A2) node[midway, above] {5};
\draw[black] (A2)--(A3);
\draw[black] (A3)--(A4);
\end{tikzpicture}
\subcaption*{$H_4$}
  \label{fig:type-H4 Coxeter quiver}
\end{minipage}&
\begin{minipage}[t]{0.3\linewidth}
  \centering
  \begin{tikzpicture}[scale=0.35]

\node[name=A1,draw,inner sep=2pt,circle] at (-1,0) {};
\node[name=A2,draw,inner sep=2pt,circle] at (1,0) {};

\draw[black] (A1)--(A2) node[midway, above] {$p$};
\end{tikzpicture}
\subcaption*{$I_2(p)$}
  \label{fig:type-I Coxeter quiver}
\end{minipage}
    \end{tabular}
    \caption{Finite-type Coxeter diagrams}
    \label{fig:finite-type Coxeter quiver}
  \end{figure}
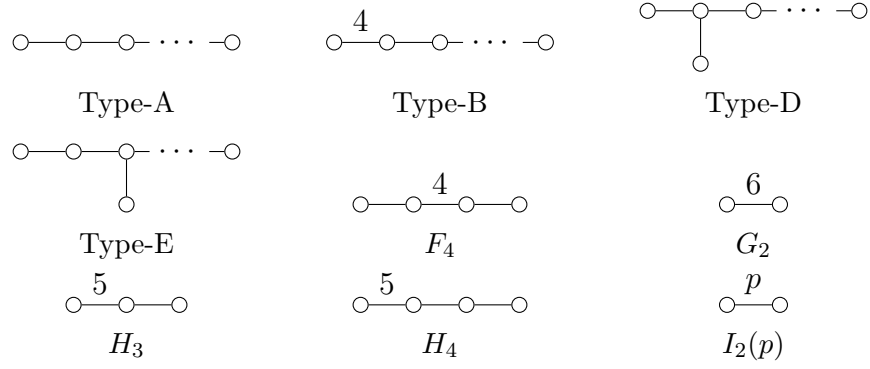
\end{example}

\begin{definition}
Let $(W,S)$ be a finite Coxeter system. The \emph{(right) weak order} of $(W,S)$ is the poset $(W,\leq)$, where $\leq$ is defined by $p\leq q$ if and only if there exists a reduced word of $p$ which is a prefix of a reduced word of $q$ (\text{i.e.,} $l(q)=l(p)+l(p^{-1}q)$). 
\end{definition}

\begin{remark}
It is well-known that the weak order of a finite Coxeter system is a lattice (See \cite[Section3.2]{MR2133266}).  
\end{remark}

\begin{lemma}
Let $x\in W$, and let $a\in S$. If $a\nleq x$, then $a\leq ax$. If $a\leq x$, then $a\nleq ax$. 
\begin{proof}
We first prove the latter.
By assumption, we can take a reduced word of $x$ that starts with $a$, then $l(ax)\lneq l(x)$.
If $ax$ has a reduced word that starts with $a$, then $l(x) = l(aax)\lneq l(ax)$, which is a contradiction. Therefore $a\nleq ax$.
Since $a^2=id$, the former follows from the latter.
\end{proof}
\end{lemma}

\begin{lemma}
\label{remove prefix}
  Let $a$ be an atom, and let $p,q\geq a$. Then, $ap\leq aq \Leftrightarrow p\leq q$.
  \begin{proof}
    This follows from the following equivalent reformulation.
    \begin{align*}
      ap\leq aq &\Leftrightarrow l(aq)=l(ap)+l((ap)^{-1}aq)\\
      &\Leftrightarrow l(q)+1=l(p)+1+l(p^{-1}q)\\
      &\Leftrightarrow l(q)=l(p)+l(p^{-1}q)\\
      &\Leftrightarrow p\leq q
    \end{align*}
  \end{proof}
\end{lemma}

\begin{lemma}
The following hold.
\begin{enumerate}
  \item If $p\prec q$, $pw_0\succ qw_0$.
  \item If $p\leq q$, $pw_0\geq qw_0$.
\end{enumerate}
\begin{proof}(2) It suffices to show $l(pw_0)=l(qw_0)+l((qw_0)^{-1}pw_0)$. This can be proved as follows.
\begin{align*}
  (RHS)&=l(qw_0)+l(w_0q^{-1}pw_0)=l(w_0)-l(q)+l(w_0)-l(w_0q^{-1}p)\\
  &=l(w_0)-l(q)+l(w_0)-l(p^{-1}qw_0)\\
  &=l(w_0)-l(q)+l(w_0)-l(w_0)+l(p^{-1}q)\\
  &=l(w_0)-l(q)+l(q^{-1}p)=l(w_0)-l(p)=(LHS)
\end{align*}
(1) follows from by (2) and computation of the length. This can be proved as follows:
\[l(w_0q^{-1}pw_0)=l(pw_0)-l(qw_0)=l(w_0)-l(p)-(l(w_0)-l(q))=l(q)-l(p)\]
\end{proof}
\label{reversing order}
\end{lemma}

The following is the main theorem of this section.

\begin{theorem}
\label{mutation of weak order}
For $a\in S$, the following hold:
\begin{enumerate}
  \item $(a,aw_0)$ is an AC-correspondence.
  \item $W'=\mu_a(W)$ is isomorphic to the weak order of $W$. In particular, $\mu_{a}$ is a mutation, and $W$ is mutable.
\end{enumerate}
\begin{proof} (1) Since $a\leq w_0$, $a\nleq aw_0$. 
Then, by \cref{remove prefix}, $x\leq aw_0$ holds for any $x\ngeq a$.
By \cref{reversing order}, exactly half of the elements in $W$ are greater than or equal to $a$. Hence, exactly half of the elements in $W$ are lower than or equal to $aw_0$. 

(2) 
$\mu_a$ reverses a direction of a cover relation $p\prec q$ if and only if $p\prec q,\ p\ngeq a,\ q\geq a$.
Consider $f(x)=ax$ for $x\in W$. Then for any cover relation $p \prec q$, $f(p)\prec f(q)$ or $f(p)\succ f(q)$ holds. 
Moreover, $f(p)\succ f(q)$ holds if and only if $p\prec q,\ p\ngeq a,\ q\geq a$.
Therefore, the relation $x\leq^{'} y \coloneqq f(x)\leq f(y)$ on $W$ is isomorphic to $W'=\mu_a(W)$.
\end{proof}
\end{theorem}

  For $U\subset S$, define $W'$ as the subgroup generated by $U$. We call $(W',U)$ a \emph{parabolic subgroup} of $(W,S)$.

\begin{lemma}
  Let $(W',U)$ be a parabolic subgroup of $(W,S)$. Then $(W',U)$ is an ideal of $(W,S)$.
  \begin{proof}
    We give only the main ideas of the proof.
    Let $w_0'$ be the longest element of $(W',U)$. Then, by minimal length coset representative (cf.~\cite[Section 1.10]{MR1066460}), $(W',U)=[id,w_0']$.
  \end{proof}
\end{lemma}

We conclude this section with the following proposition, which will be used in the proof of \cref{general type Ordovician lattice and the infinite case}.

\begin{proposition}
\label{quotient of weak order}
  Let $(W,S)$ be a finite Coxeter system, let $a\in S$, and let $f:W\to L$ be a surjective lattice homomorphism such that $f(id)=f(a)$.
  Then $L$ is a quotient lattice of a parabolic subgroup $(W',S-\{a\})$ .
  \begin{proof}
    Let $w_{!a}$ be the longest element of $W'$.
    We show $a\vee w_{!a} = w_0$ and $a\wedge w_{!a} = id$. The latter follows from $a\nleq w_{!a}$. The former follows from that $a\vee w_{!a}$ is greater than or equal to all elements in $S$.  
    For any lattice homomorphism $f$, if $f(x\wedge y)=f(x)$ then $f(y)=f(x\vee y)$. This can be proved as follows:
    \[f(x\vee y)=f(x)\vee f(y)=f(x\wedge y)\vee f(y)=f((x\wedge y)\vee y)=f(y) \]
    Since $f(id)=f(a)$, $f(w_{!a})=f(w_0)$. Then $f(x)=f(x\wedge w_{!a})$.
    Let $f'(x)=x\wedge w_{!a}$ for any $x\in W$. Then, $f'$ is order homomorphism from $(W,S)$ to $(W',S-\{a\})$, and preserves $\wedge$.
    Since no element of $S-\{a\}$ is lower than or equal to $f'(w_{!a}w_0)$, we obtain $f'(w_{!a}w_0)=id$.
    By minimal length coset representative (cf.~\cite[Section 1.10]{MR1066460}), for any $x$, there exists exactly one expression $x=x_{!a}x^{!a}$ such that $x^{!a}\leq w_{!a}w_0$ and $x_{!a}\leq w_{!a}$. Moreover, $x_{!a}=f'(x)$ and $l(x)=l(x_{!a})+l(x^{!a})$ holds.
    Let us prove the following lemma:
    \[for \ all \ y\leq w_{!a}, \ f'(x)\leq y \Leftrightarrow x\leq yw_{!a}w_0\]
    $(\Leftarrow)$ follows by applying $f'$ to both sides.
    $(\Rightarrow)$ It suffices to prove $x\leq f'(x)w_{!a}w_0\leq yw_{!a}w_0$. 
    $x\leq f'(x)w_{!a}w_0$ reduces to $x^{!a}\leq w_{!a}w_0$ by \cref{remove prefix}, and hence follows. 
    We show $f'(x)w_{!a}w_0\leq yw_{!a}w_0$ by a sequence of equivalent transformations.
    $f'(x)w_{!a}w_0\leq yw_{!a}w_0$ is equivalent to $f'(x)w_{!a}\geq yw_{!a}$, which is in turn equivalent to $f'(x)\leq y$ because both sides are in $W'$.
    Therefore the lemma follows.
    We now return to the proof of the theorem.
    By this lemma, for any $z\leq w_{!a}$, we have the following equivalences:
    \begin{align*}
      f'(x\vee y) \leq z \ &\Leftrightarrow \ x\vee y \leq zw_{!a}w_0 \\
      &\Leftrightarrow \ x \leq zw_{!a}w_0 \ \text{and} \  y \leq zw_{!a}w_0 \\
      &\Leftrightarrow f'(x)\leq z \ \text{and} \  f'(y)\leq z
    \end{align*}
    Then $f'$ preserves $\vee$, and hence $f'$ is surjective lattice homomorphism.
    Therefore, there exists a surjective lattice homomorphism $g$ such that $f=gf'$. 
  \end{proof}
\end{proposition}

\section{Flips of Cambrian lattices}
\label{sec:Flips of Cambrian lattices}
In this section, we show that a Cambrian lattice is locally mutable and can be mutated into other Cambrian lattices.
Also we state that under certain mutations, a Cambrian lattice is transformed into a structure that is not a Cambrian lattice. We call this resulting structure an Ordovician lattice and discuss its properties.

\begin{definition}[{\cite[Section 3]{MR2486939}}]
  Let $(W,S)$ be a finite-type Coxeter system, $x\in W$, and let $c=s_{i_1}\cdots s_{i_n}$ be a Coxeter element.
  The \emph{$c$-sorting word} for $x$ is the lexicographically leftmost subword of $c_{\infty}\coloneq s_{i_1}\cdots s_{i_n}|s_{i_1}\cdots s_{i_n}|s_{i_1}\cdots s_{i_n}|\cdots$ which is a reduced word of $x$.
  We regard $|$ as dividers, thus the $c$-sorting word is expressed as a sequence of subsets of $S$. 
  An element $x$ is a \emph{$c$-sortable element} if this sequence is weakly decreasing under inclusion.
  Define the congruence $\Theta_{c}$ on $W$ by $x\equiv y\Leftrightarrow \pi^c_{\downarrow}(x)=\pi^c_{\downarrow}(y)$, where $\pi^c_{\downarrow}(x)$ is the unique maximal $c$-sortable element satisfying $\pi^c_{\downarrow}(x)\leq x$.
  A \emph{Cambrian lattice} $\Camb(c)$ is the quotient lattice of the weak order of $(W,S)$ modulo the congruence $\Theta_{c}$.
\end{definition}

\begin{definition}
  Let $(W,S)$ be a finite-type Coxeter system, and let $c=s_{i_1}\cdots s_{i_n}$ be a Coxeter element.
  A \emph{Coxeter quiver} $\overrightarrow{B}$ associated with $c$ is defined as follows:
  \begin{enumerate}
    \item It is a quiver with labelled arrows whose underlying undirected graph (with labelled edges) is the Coxeter diagram of $W$.
    \item For $j\neq k$, an edge between $j$ and $k$ is oriented as $j\to k$ if $j=i_s, \ k=i_t$ and $s\lneq t$, and as $k\to j$ otherwise.
  \end{enumerate}
  By convention, the BGP-reflections of $\overrightarrow{B}$ leave the arrow labellings unchanged.
\end{definition}

We can uniquely determine $c$ from $\overrightarrow{B}$.
Thus, we define $\Camb(\overrightarrow{B})$ by $\Camb(c)$.

\begin{proposition}
\label{direction of tree}
Suppose that $\overrightarrow{B_1},\overrightarrow{B_2}$ be quivers whose underlying undirected graphs are isomorphic, and that this graph is one of the graphs shown in \cref{fig:finite-type Coxeter quiver}.
Then $\overrightarrow{B_1}$ and $\overrightarrow{B_2}$ can be obtained from each other by a finite sequence of BGP-reflections.

  \begin{proof}
    For type-A, the proof is straightforward by induction of $n$.
  The other types follow by a similar argument.
  \end{proof}
\end{proposition}

\begin{example}
  We present a figure showing all Cambrian lattices of type $A_3$.
\begin{figure}[htbp]
\centering
\begin{tabular}{ccc}
  \begin{minipage}[t]{0.3\linewidth}
  \begin{tikzpicture}[scale=0.5]

\node[name=A0,draw,inner sep=2pt,circle] at (0,0) {};
\node[name=A1,draw,inner sep=2pt,circle] at (-2,-2) {};
\node[name=A2,draw,inner sep=2pt,circle] at (0,-2) {};
\node[name=A3,draw,inner sep=2pt,circle] at (2,-2) {};
\node[name=A4,draw,inner sep=2pt,circle] at (-1,-4) {};
\node[name=A5,draw,inner sep=2pt,circle] at (4,-4) {};
\node[name=A6,draw,inner sep=2pt,circle] at (2,-6) {};
\node[name=A7,draw,inner sep=2pt,circle] at (4,-6) {};
\node[name=A8,draw,inner sep=2pt,circle] at (1,-8) {};
\node[name=A9,draw,inner sep=2pt,circle] at (4,-8) {};
\node[name=Aa,draw,inner sep=2pt,circle] at (-2,-10) {};
\node[name=Ab,draw,inner sep=2pt,circle] at (0,-10) {};
\node[name=Ac,draw,inner sep=2pt,circle] at (2,-10) {};
\node[name=Ad,draw,inner sep=2pt,circle] at (0,-12) {};

\draw[arrows = {-Stealth[scale=1]},black] (A0)--(A1);
\draw[arrows = {-Stealth[scale=1]},black] (A0)--(A2);
\draw[arrows = {-Stealth[scale=1]},black] (A0)--(A3);
\draw[arrows = {-Stealth[scale=1]},black] (A1)--(A4);
\draw[arrows = {-Stealth[scale=1]},black] (A1)--(Aa);
\draw[arrows = {-Stealth[scale=1]},black] (A2)--(A6);
\draw[arrows = {-Stealth[scale=1]},black] (A2)--(A8);
\draw[arrows = {-Stealth[scale=1]},black] (A3)--(A4);
\draw[arrows = {-Stealth[scale=1]},black] (A3)--(A5);
\draw[arrows = {-Stealth[scale=1]},black] (A4)--(Ab);
\draw[arrows = {-Stealth[scale=1]},black] (A5)--(A6);
\draw[arrows = {-Stealth[scale=1]},black] (A5)--(A7);
\draw[arrows = {-Stealth[scale=1]},black] (A6)--(A9);
\draw[arrows = {-Stealth[scale=1]},black] (A7)--(A9);
\draw[arrows = {-Stealth[scale=1]},black] (A7)--(Ab);
\draw[arrows = {-Stealth[scale=1]},black] (A8)--(Aa);
\draw[arrows = {-Stealth[scale=1]},black] (A8)--(Ac);
\draw[arrows = {-Stealth[scale=1]},black] (A9)--(Ac);
\draw[arrows = {-Stealth[scale=1]},black] (Aa)--(Ad);
\draw[arrows = {-Stealth[scale=1]},black] (Ab)--(Ad);
\draw[arrows = {-Stealth[scale=1]},black] (Ac)--(Ad);

\end{tikzpicture}
  \centering
  \captionsetup{width=0.8\textwidth}
  \caption{The $A_3$ Tamari lattice}
  \end{minipage} &
  
  \begin{minipage}[t]{0.3\linewidth}
  \begin{tikzpicture}[scale=0.5]

\node[name=A0,draw,inner sep=2pt,circle] at (0,0) {};
\node[name=A1,draw,inner sep=2pt,circle] at (-2,-2) {};
\node[name=A2,draw,inner sep=2pt,circle] at (0,-2) {};
\node[name=A3,draw,inner sep=2pt,circle] at (2,-2) {};
\node[name=A4,draw,inner sep=2pt,circle] at (-2,-4) {};
\node[name=A5,draw,inner sep=2pt,circle] at (2,-4) {};
\node[name=A6,draw,inner sep=2pt,circle] at (0,-6) {};
\node[name=A7,draw,inner sep=2pt,circle] at (-3,-7) {};
\node[name=A8,draw,inner sep=2pt,circle] at (3,-7) {};
\node[name=A9,draw,inner sep=2pt,circle] at (0,-8) {};
\node[name=Aa,draw,inner sep=2pt,circle] at (-2,-10) {};
\node[name=Ab,draw,inner sep=2pt,circle] at (0,-10) {};
\node[name=Ac,draw,inner sep=2pt,circle] at (2,-10) {};
\node[name=Ad,draw,inner sep=2pt,circle] at (0,-12) {};

\draw[arrows = {-Stealth[scale=1]},black] (A0)--(A1);
\draw[arrows = {-Stealth[scale=1]},black] (A0)--(A2);
\draw[arrows = {-Stealth[scale=1]},black] (A0)--(A3);
\draw[arrows = {-Stealth[scale=1]},black] (A1)--(A7);
\draw[arrows = {-Stealth[scale=1]},black] (A1)--(Ab);
\draw[arrows = {-Stealth[scale=1]},black] (A2)--(A4);
\draw[arrows = {-Stealth[scale=1]},black] (A2)--(A5);
\draw[arrows = {-Stealth[scale=1]},black] (A3)--(A8);
\draw[arrows = {-Stealth[scale=1]},black] (A3)--(Ab);
\draw[arrows = {-Stealth[scale=1]},black] (A4)--(A6);
\draw[arrows = {-Stealth[scale=1]},black] (A4)--(A7);
\draw[arrows = {-Stealth[scale=1]},black] (A5)--(A6);
\draw[arrows = {-Stealth[scale=1]},black] (A5)--(A8);
\draw[arrows = {-Stealth[scale=1]},black] (A6)--(A9);
\draw[arrows = {-Stealth[scale=1]},black] (A7)--(Aa);
\draw[arrows = {-Stealth[scale=1]},black] (A8)--(Ac);
\draw[arrows = {-Stealth[scale=1]},black] (A9)--(Aa);
\draw[arrows = {-Stealth[scale=1]},black] (A9)--(Ac);
\draw[arrows = {-Stealth[scale=1]},black] (Aa)--(Ad);
\draw[arrows = {-Stealth[scale=1]},black] (Ab)--(Ad);
\draw[arrows = {-Stealth[scale=1]},black] (Ac)--(Ad);

\end{tikzpicture}
  \centering
  \captionsetup{width=0.8\textwidth}
  \caption{$A_3$ Cambrian lattice but not a Tamari.}
  \end{minipage} &

  \begin{minipage}[t]{0.3\linewidth}
  \begin{tikzpicture}[scale=0.5]

\node[name=A0,draw,inner sep=2pt,circle] at (0,-12) {};
\node[name=A1,draw,inner sep=2pt,circle] at (-2,-10) {};
\node[name=A2,draw,inner sep=2pt,circle] at (0,-10) {};
\node[name=A3,draw,inner sep=2pt,circle] at (2,-10) {};
\node[name=A4,draw,inner sep=2pt,circle] at (-2,-8) {};
\node[name=A5,draw,inner sep=2pt,circle] at (2,-8) {};
\node[name=A6,draw,inner sep=2pt,circle] at (0,-6) {};
\node[name=A7,draw,inner sep=2pt,circle] at (-3,-5) {};
\node[name=A8,draw,inner sep=2pt,circle] at (3,-5) {};
\node[name=A9,draw,inner sep=2pt,circle] at (0,-4) {};
\node[name=Aa,draw,inner sep=2pt,circle] at (-2,-2) {};
\node[name=Ab,draw,inner sep=2pt,circle] at (0,-2) {};
\node[name=Ac,draw,inner sep=2pt,circle] at (2,-2) {};
\node[name=Ad,draw,inner sep=2pt,circle] at (0,0) {};

\draw[arrows = {-Stealth[scale=1]},black] (A1)--(A0);
\draw[arrows = {-Stealth[scale=1]},black] (A2)--(A0);
\draw[arrows = {-Stealth[scale=1]},black] (A3)--(A0);
\draw[arrows = {-Stealth[scale=1]},black] (A7)--(A1);
\draw[arrows = {-Stealth[scale=1]},black] (Ab)--(A1);
\draw[arrows = {-Stealth[scale=1]},black] (A4)--(A2);
\draw[arrows = {-Stealth[scale=1]},black] (A5)--(A2);
\draw[arrows = {-Stealth[scale=1]},black] (A8)--(A3);
\draw[arrows = {-Stealth[scale=1]},black] (Ab)--(A3);
\draw[arrows = {-Stealth[scale=1]},black] (A6)--(A4);
\draw[arrows = {-Stealth[scale=1]},black] (A7)--(A4);
\draw[arrows = {-Stealth[scale=1]},black] (A6)--(A5);
\draw[arrows = {-Stealth[scale=1]},black] (A8)--(A5);
\draw[arrows = {-Stealth[scale=1]},black] (A9)--(A6);
\draw[arrows = {-Stealth[scale=1]},black] (Aa)--(A7);
\draw[arrows = {-Stealth[scale=1]},black] (Ac)--(A8);
\draw[arrows = {-Stealth[scale=1]},black] (Aa)--(A9);
\draw[arrows = {-Stealth[scale=1]},black] (Ac)--(A9);
\draw[arrows = {-Stealth[scale=1]},black] (Ad)--(Aa);
\draw[arrows = {-Stealth[scale=1]},black] (Ad)--(Ab);
\draw[arrows = {-Stealth[scale=1]},black] (Ad)--(Ac);
\end{tikzpicture}
  \centering
  \captionsetup{width=0.8\textwidth}
  \caption{Upside down the lattice structure of the center.}
  \end{minipage}
\end{tabular}
\end{figure}
  
\end{example}

\begin{remark}[\cite{MR2258260}]
  Let $\overrightarrow{B}$ be a type-A (resp.~type-B) Coxeter quiver.
  If all edges of $\overrightarrow{B}$ have the same orientation, then we call $\Camb(\overrightarrow{B})$ a type-A (resp.~type-B) Tamari lattice.
  For type-A, the $A_n$ Tamari lattice is uniquely determined.
  For type-B, there are 2 types of lattice as a $B_n$ Tamari lattice, and they are opposites of each other.
\end{remark}

\begin{proposition}
  The lattice structure $\Camb(\overrightarrow{B'})$ is the dual order of $\Camb(\overrightarrow{B})$, where $\overrightarrow{B'}$ is obtained from $\overrightarrow{B}$ by reversing the orientation of every edge.
  \begin{proof}
    We provide a graphical proof only for type-A and B in \cref{dual order of Cambrian lattice in section A}.
    For other types, see \cite[Between Example 3.4 and Lemma 3.5]{MR2486939}.
  \end{proof}
  \label{dual order of Cambrian lattice in section 7}
\end{proposition}

\begin{construction}[cf.~\cite{MR2258260}]
\label{General type Cambrian lattice and ideal polygon}
  Let $\overrightarrow{B}$ be a finite-type Coxeter quiver, and let $a,b\in \Camb(\overrightarrow{B})$ be atoms.
  Then $[0,a\vee b]$ is a $(2,k)$-polygon for some $k\geq 2$.
  Hence we define the quiver with labelled arrows whose vertices are indexed by atoms of $\Camb(\overrightarrow{B})$.
  Arrows are defined as follows:
  If $k\geq 3$ holds, we include an arrow labelled by $k$ from the atom in a chain of length $k$ in $[0,a\vee b]$ to the other atom in $[0,a\vee b]$.
  Then the resulting quiver $Q$ is isomorphic to $\overrightarrow{B}$.
  
\end{construction}

In particular, we can take a bijection between the vertices of $\overrightarrow{B}$ and the atoms of $\Camb(\overrightarrow{B})$.

The following theorem relates the mutation of lattices and the BGP-reflection of quivers.
The proof of (2) and (3) was suggested by N.~Reading.

\begin{theorem}
\label{mutation of general type Cambrian lattice}
The following hold;
\begin{enumerate}
  \item Any finite-type of Cambrian lattices are locally mutable.
  \item Let $\overrightarrow{B}$ be a finite-type Coxeter quiver, and let $i\in \overrightarrow{B}$ be a sink or a source. Then $\mu_{\eta((i,i+1))}(\Camb(\overrightarrow{B})) \cong \Camb(\mu_{i}(\overrightarrow{B}))$, where $\mu_{i}$ is a BGP-reflection on $i$.
  \item Let $\overrightarrow{B},\overrightarrow{B'}$ be finite-type Coxeter quivers such that $B\cong B'$. Then $\Camb(\overrightarrow{B})$ is mapped to a lattice isomorphic to $\Camb(\overrightarrow{B'})$ by a finite sequence of mutations.
\end{enumerate}
\begin{proof}
  (1) Let $\eta_{c}:W\to \Camb(c)$ be the natural quotient map.
  Then $\eta_{c}$ sends every atom to an atom since every atom $s\in S$ is a $c$-sortable element.
  Hence, since \cref{mutation of weak order,quotient flip}, $\Camb(c)$ is a locally mutable lattice.
  \cref{prerequisite for quotient flip} verifies that the prerequisites of \cref{quotient flip} is satisfied.
  (2) From the assumption that $i\in \overrightarrow{B}$ is a sink (resp.~source), it follows that if a generator $s_i\in S$ satisfies $\eta(s_i)=a_i$, then it is final (resp.~initial) letter of a reduced word of $c$, where $a_i\in \Camb(\overrightarrow{B})$ is an atom corresponding to $i$.
  Thus, $s_{i}cs_{i}$ is also a Coxeter element corresponds to $\mu_{i}(\overrightarrow{B})$.
  Consequently, $s_{i}cs_{i}$-Cambrian lattice is isomorphic to $\Camb (\mu_{i}(\overrightarrow{B}))$.
  Then, by \cite[Proposition 7.4 and Lemma 7.5]{MR2486939}, we obtain the bijection $[Z_{s_i}]: [w]_{c}\mapsto [Z_{s_i}(w)]_{scs}$ from $c$-Cambrian fan to $s_{i}cs_{i}$-Cambrian fan, where $Z_{s_i}$ is defined between \cite[Proposition 7.4 and Lemma 7.5]{MR2486939}.
  By \cite[Proposition 4.3]{MR2486939}, $[Z_{s_i}]$ preserves a pair $\{x,y\}$ such that $x\prec y$ or $x\succ y$.
  A straightforward computation shows that $Z_{s_i}$ preserves a relation $w\leq v$ if and only if $s_{i}\leq w,v$ or $s_{i}\nleq w,v$.
  Hence, $[Z_{s_i}]$ coincides with $\mu_{\eta(s_i)}$.
  Since analogous versions of \cref{direction of tree} for the other types also hold, (3) follows from (2).
\end{proof}
\end{theorem}

\begin{remark}[Flip-Flop of Cambrian lattices]
  For type-A, the flip in (2) of \cref{mutation of general type Cambrian lattice} can also be described as a Flip-Flop \cite{ladkani2007universalderivedequivalencesposets}.
  In particular, $A=\partial A$ or $B=\partial B$ in this case.
  We give a short graphical proof for this; see \cref{Flip-Flop of Cambrian lattices in section A}.
  \label{Flip-Flop of Cambrian lattices in section 7}
\end{remark}

\begin{definition}
A lattice $L$ is said to be a \emph{type-X Ordovician order} if $L$ is isomorphic to some $L'' \in G(L')$ for some type-X Cambrian lattice $L'$.
We call an Ordovician order $L$ an \emph{type-X Ordovician lattice}\footnote{Ordovician lattices are named after the Ordovician period because Cambrian lattices are named after the Cambrian period (cf.~\cite{MR2258260}). Ordovician period is the next period of the Cambrian period.} if $L$ is a lattice.
\end{definition}

\begin{example}
  We present a type-A Ordovician lattice. This is obtained from The $A_3$ Tamari lattice by a mutation on an atom that is neither a sink nor a source.
      \begin{figure}[htbp]
\centering
\begin{tabular}{c}
  \begin{minipage}[t]{0.9\linewidth}
  \begin{tikzpicture}[scale=0.5]

\node[name=A0,draw,inner sep=2pt,circle] at (0,0) {};
\node[name=A1,draw,inner sep=2pt,circle] at (-2,-2) {};
\node[name=A2,draw,inner sep=2pt,circle] at (0,-2) {};
\node[name=A3,draw,inner sep=2pt,circle] at (2,-2) {};
\node[name=A4,draw,inner sep=2pt,circle] at (-3,-4.5) {};
\node[name=A5,draw,inner sep=2pt,circle] at (-1,-4.5) {};
\node[name=A6,draw,inner sep=2pt,circle] at (1,-4.5) {};
\node[name=A7,draw,inner sep=2pt,circle] at (-1,-6.5) {};
\node[name=A8,draw,inner sep=2pt,circle] at (1,-6.5) {};
\node[name=A9,draw,inner sep=2pt,circle] at (3,-6.5) {};
\node[name=Aa,draw,inner sep=2pt,circle] at (-2,-9) {};
\node[name=Ab,draw,inner sep=2pt,circle] at (0,-9) {};
\node[name=Ac,draw,inner sep=2pt,circle] at (2,-9) {};
\node[name=Ad,draw,inner sep=2pt,circle] at (0,-11) {};

\draw[arrows = {-Stealth[scale=1]},black] (A0)--(A1);
\draw[arrows = {-Stealth[scale=1]},black] (A0)--(A2);
\draw[arrows = {-Stealth[scale=1]},black] (A0)--(A3);
\draw[arrows = {-Stealth[scale=1]},black] (A1)--(A4);
\draw[arrows = {-Stealth[scale=1]},black] (A1)--(A8);
\draw[arrows = {-Stealth[scale=1]},black] (A2)--(A5);
\draw[arrows = {-Stealth[scale=1]},black] (A2)--(A9);
\draw[arrows = {-Stealth[scale=1]},black] (A3)--(A6);
\draw[arrows = {-Stealth[scale=1]},black] (A3)--(A7);
\draw[arrows = {-Stealth[scale=1]},black] (A4)--(A9);
\draw[arrows = {-Stealth[scale=1]},black] (A4)--(Ab);
\draw[arrows = {-Stealth[scale=1]},black] (A5)--(A7);
\draw[arrows = {-Stealth[scale=1]},black] (A5)--(Ac);
\draw[arrows = {-Stealth[scale=1]},black] (A6)--(A8);
\draw[arrows = {-Stealth[scale=1]},black] (A6)--(Aa);
\draw[arrows = {-Stealth[scale=1]},black] (A7)--(Aa);
\draw[arrows = {-Stealth[scale=1]},black] (A8)--(Ab);
\draw[arrows = {-Stealth[scale=1]},black] (A9)--(Ac);
\draw[arrows = {-Stealth[scale=1]},black] (Aa)--(Ad);
\draw[arrows = {-Stealth[scale=1]},black] (Ab)--(Ad);
\draw[arrows = {-Stealth[scale=1]},black] (Ac)--(Ad);

\end{tikzpicture}
  \centering
  \captionsetup{width=3\textwidth}
  \caption{A type-A Ordovician lattice obtained from the $A_3$ Tamari.}
  \label{fig:A2affineTamari}
  \end{minipage}
\end{tabular}
\end{figure}
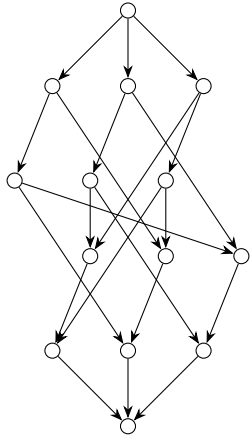
  
This Ordovician lattice is not any type of Cambrian.
This lattice has 3 mutations but all mutations send to a lattice isomorphic to the $A_3$ Tamari.
\end{example}

We conjecture that the mutations of Ordovician lattices correspond to the mutations of quivers which are introduced by S.~Fomin and A.~Zelevinsky \text{\cite{MR1887642}}.
To state the conjecture, we first introduce the weighted quiver and its mutations.
We define a \emph{weighted quiver} as a quiver each of whose vertices $i$ is labelled by $d_i\in \mathbb{Z}_{\gneq 0}$.

Hereafter, we will refer to a weighted quiver simply as a quiver when no confusions can arise.

\begin{definition}[{\cite{MR1887642} (cf.~\cite[Lemma 2.3]{MR4265942})}]
  Assume that $Q$ has neither loops nor 2-cycles. For vertex $i$, define $Q'=\mu_i(Q)$ as follows:
  \begin{enumerate}
    \item For each pair of vertices $j,k$, we add $\{\text{number of edges} \ j\to i\}\times \{\text{number of edges} \ i\to k\} \times \frac{d_{jk}d_i}{d_{ji}d_{ik}}$ copies of the edge $j\to k$, where $d_{xy}$ is GCD of $d_x$, $d_y$. 
    \item Reverse all edges incident to $i$.
    \item Resolve each 2-cycle by canceling pairs of opposite directed edges. 
  \end{enumerate}
If $i$ is either a sink or a source, $\mu_i$ is a BGP-reflection.

\begin{construction} 
  For a weighted quiver $Q$, we consider the following condition;
  \begin{itemize}
    \item[($\star$)]  For each vertices $i,j$ in the same connected component of $Q$, $\frac{d_i}{d_j}$ is equal to one of $1,2,3,\frac{1}{2}$ and $\frac{1}{3}$.
  \end{itemize}
  
  We construct a quiver $\overrightarrow{B}(Q)$ with labelled arrows from a weighted quiver $Q$ satisfying ($\star$) as follows;
\begin{enumerate}
  \item $\overrightarrow{B}(Q)=Q$ if we ignore their labels and weights.
  \item Each edge $i\to j$ is labelled as $3,4$ or $6$: $3$ if $d_i=d_j$, $4$ if $d_i=2d_j$ or $2d_i=d_j$, and $6$ if $d_i=3d_j$ or $3d_i=d_j$. 
\end{enumerate}
\label{weights of arrows and vertices}
\end{construction}

A weighted quiver $Q$ satisfying ($\star$) is said to be \emph{of Dynkin type} if $\overrightarrow{B}(Q)$ is a Coxeter quiver of a finite-type Coxeter system. See \cref{fig:Dynkin type weighted quiver} for the complete list of connected weighted quiver of Dynkin type.
  \begin{figure}
  \centering
  \begin{tabular}{cc}
  \begin{minipage}{0.4\linewidth}
  \centering
    \begin{tikzpicture}[scale=1]
      \node[draw,shape=circle,inner sep=2pt,double] (A) at (-2,0) {};
\node[draw,shape=circle,inner sep=2pt] (B) at (0,0) {};
\node[draw,shape=circle,inner sep=2pt,double] (C) at (2,0) {};

\draw[->] (A)--(B);
\draw[->] (B)--(C);
\draw[->] (C) to[bend right] (A);
    \end{tikzpicture}
    \subcaption{Before mutation}
  \end{minipage}&
  \begin{minipage}{0.4\linewidth}
  \centering
  \begin{tikzpicture}[scale=1]
\node[draw,shape=circle,inner sep=2pt,double] (A) at (-2,0) {};
\node[draw,shape=circle,inner sep=2pt] (B) at (0,0) {};
\node[draw,shape=circle,inner sep=2pt,double] (C) at (2,0) {};

\draw[->] (A)--(B);
\draw[->] (B)--(C);
\draw[->] (C) to[bend right] (A);
\draw[->,blue] (A) to[bend right] (C);
\draw[->,blue] (A) to[bend right=45] (C);
\end{tikzpicture}
\subcaption{Operation(1)}
\end{minipage}\\
\begin{minipage}{0.4\linewidth}
\centering
    \begin{tikzpicture}[scale=1]
      \node[draw,shape=circle,inner sep=2pt,double] (A) at (-2,0) {};
\node[draw,shape=circle,inner sep=2pt] (B) at (0,0) {};
\node[draw,shape=circle,inner sep=2pt,double] (C) at (2,0) {};

\draw[<-,blue] (A)--(B);
\draw[<-,blue] (B)--(C);
\draw[->] (C) to[bend right] (A);
\draw[->] (A) to[bend right] (C);
\draw[->] (A) to[bend right=45] (C);
    \end{tikzpicture}
    \subcaption{Operation(2)}
  \end{minipage}&
  \begin{minipage}{0.4\linewidth}
  \centering
  \begin{tikzpicture}[scale=1]
\node[draw,shape=circle,inner sep=2pt,double] (A) at (-2,0) {};
\node[draw,shape=circle,inner sep=2pt] (B) at (0,0) {};
\node[draw,shape=circle,inner sep=2pt,double] (C) at (2,0) {};

\draw[<-] (A)--(B);
\draw[<-] (B)--(C);
\draw[->,dotted,cyan=30] (C) to[bend right] (A);
\draw[->,dotted,cyan=30] (A) to[bend right] (C);
\draw[->] (A) to[bend right=45] (C);
\end{tikzpicture}
\subcaption{Operation(3)}
\end{minipage}
  \end{tabular}
  \caption{An example of mutation of a quiver. The vertices of weight $1,2$ are represented by a circle and a double circle, respectively.}
  \label{fig:mutation of quiver}
\end{figure}
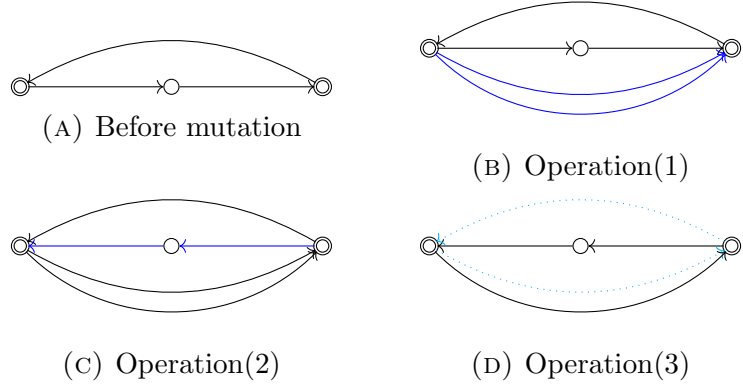
\end{definition}

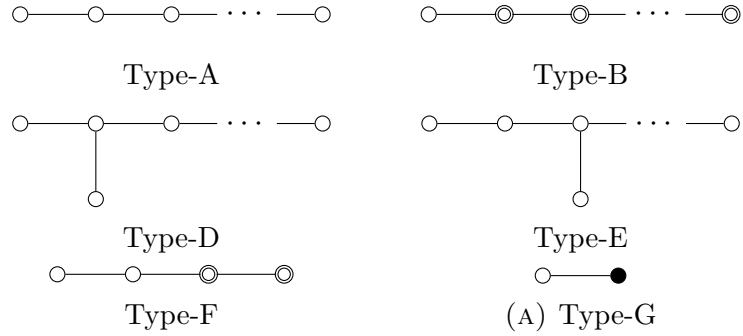
\begin{figure}
  \centering
  \begin{tabular}{cc}
  \begin{minipage}[t]{0.4\linewidth}
  \centering
  \begin{tikzpicture}[scale=0.5]

\node[name=A1,draw,inner sep=2pt,circle] at (-4,0) {};
\node[name=A2,draw,inner sep=2pt,circle] at (-2,0) {};
\node[name=A3,draw,inner sep=2pt,circle] at (0,0) {};
\node[name=A4,inner sep=2pt,circle] at (2,0) {$\cdots$};
\node[name=A5,draw,inner sep=2pt,circle] at (4,0) {};

\draw[black] (A1)--(A2)--(A3)--(A4)--(A5);
\end{tikzpicture}
  \subcaption*{Type-A}
  \label{fig:type-A weighted quiver}
  \end{minipage}&
  
\begin{minipage}[t]{0.4\linewidth}
\centering
  \begin{tikzpicture}[scale=0.5]

\node[name=A1,draw,inner sep=2pt,circle] at (-4,0) {};
\node[name=A2,draw,inner sep=2pt,circle,double] at (-2,0) {};
\node[name=A3,draw,inner sep=2pt,circle,double] at (0,0) {};
\node[name=A4,inner sep=2pt,circle] at (2,0) {$\cdots$};
\node[name=A5,draw,inner sep=2pt,circle,double] at (4,0) {};

\draw[black] (A1)--(A2)--(A3)--(A4)--(A5);
\end{tikzpicture}
  \subcaption*{Type-B}
  \label{fig:type-B weighted quiver}
  \end{minipage}
  \\
  \begin{minipage}[t]{0.4\linewidth}
  \centering
    \begin{tikzpicture}[scale=0.5]

\node[name=A1,draw,inner sep=2pt,circle] at (-4,0) {};
\node[name=A2,draw,inner sep=2pt,circle] at (-2,0) {};
\node[name=A3,draw,inner sep=2pt,circle] at (-2,-2) {};
\node[name=A4,draw,inner sep=2pt,circle] at (0,0) {};
\node[name=A5,inner sep=2pt,circle] at (2,0) {$\cdots$};
\node[name=A6,draw,inner sep=2pt,circle] at (4,0) {};

\draw[black] (A1)--(A2)--(A4)--(A5)--(A6);
\draw[black] (A2)--(A3);
\end{tikzpicture}
\subcaption*{Type-D}
  \label{fig:type-D weighted quiver}
  \end{minipage}&
\begin{minipage}[t]{0.4\linewidth}
\centering
\begin{tikzpicture}[scale=0.5]

\node[name=A1,draw,inner sep=2pt,circle] at (-4,0) {};
\node[name=A2,draw,inner sep=2pt,circle] at (-2,0) {};
\node[name=A3,draw,inner sep=2pt,circle] at (0,0) {};
\node[name=A4,draw,inner sep=2pt,circle] at (0,-2) {};
\node[name=A5,inner sep=2pt,circle] at (2,0) {$\cdots$};
\node[name=A6,draw,inner sep=2pt,circle] at (4,0) {};

\draw[black] (A1)--(A2)--(A3)--(A5)--(A6);
\draw[black] (A3)--(A4);
\end{tikzpicture}
\subcaption*{Type-E}
  \label{fig:type-E weighted quiver}
\end{minipage}
\\
\begin{minipage}[t]{0.4\linewidth}
  \centering
  \begin{tikzpicture}[scale=0.5]

\node[name=A1,draw,inner sep=2pt,circle] at (-3,0) {};
\node[name=A2,draw,inner sep=2pt,circle] at (-1,0) {};
\node[name=A3,draw,inner sep=2pt,circle,double] at (1,0) {};
\node[name=A4,draw,inner sep=2pt,circle,double] at (3,0) {};

\draw[black] (A1)--(A2)--(A3)--(A4);
\end{tikzpicture}
\subcaption*{Type-F}
  \label{fig:type-F weighted quiver}
\end{minipage}&
\begin{minipage}[t]{0.4\linewidth}
  \centering
\begin{tikzpicture}[scale=0.5]

\node[name=A1,draw,inner sep=2pt,circle] at (-1,0) {};
\node[name=A2,draw,inner sep=2pt,circle,fill=black] at (1,0) {};

\draw[black] (A1)--(A2);
\end{tikzpicture}
\subcaption{Type-G}
  \label{fig:type-G weighted quiver}
\end{minipage}
  \end{tabular}
  \caption{The underlying undirected graph of Dynkin type weighted quiver. The vertices of weight $1,2,3$ are represented by a circle, a double circle, and a filled circle, respectively.}
  \label{fig:Dynkin type weighted quiver}
\end{figure}

\begin{remark}
  The following operations on any connected component of $Q$ do not change $\overrightarrow{B}(Q)$.
  \begin{itemize}
    \item Multiply all the original weights by a constant.
    \item Replace all the original weights by their reciprocals.
  \end{itemize}
  Note that these operations do not change $\frac{d_{jk}d_i}{d_{ji}d_{ik}}$ for $i,j,k$ in the same connected component.
  \label{change of vertex weight}
\end{remark}

\begin{conjecture}
\label{conj:structure of Ordovician lattice}
  The following hold:
  \begin{enumerate}
    \item Every finite-type Ordovician order is a lattice. In particular, every finite-type Cambrian lattice is mutable.
    \item Let $L$ be an finite-type Ordovician lattice, and let $a,b$ be its atoms. Then $[0,a\vee b]$ is a $(2,m)$-polygon for some $m\geq 2$. Moreover, $L$ is \emph{polygonal}, \text{i.e.}, it satisfies that $[x,y\vee z]$ is a polygon for any $x\prec y,z$ and $[x\wedge y, z]$ is a polygon for any $x,y\prec z$ \text{\cite{MR3645055}}.
    \item Let $L$ be an Dynkin type Ordovician lattice corresponding to $\overrightarrow{B}$, and let $i\in L$ be an atom. Take a weighted quiver $Q$ satisfying ($\star$) and $\overrightarrow{B}(Q)=\overrightarrow{B}$.  Then $\mu_i(L)$ corresponds to a quiver with labelled arrows isomorphic to $\overrightarrow{B}(\mu_{i}(Q))$. Here, the correspondence between an Dynkin type Ordovician lattice and a quiver with labelled arrows is given by \cref{General type Cambrian lattice and ideal polygon}.
  \end{enumerate}
\end{conjecture}

The following theorem provides supporting evidence for (3) of \cref{conj:structure of Ordovician lattice}.

\begin{theorem}
\label{general type of Ordovician lattice and ideal polygon}
  Let $Q$ be a weighted quiver of Dynkin type, $L=\Camb(\overrightarrow{B}(Q))$, $a_i \in L$ an atom, and $L'=\mu_{a_i}(L)$.
  Then, \cref{General type Cambrian lattice and ideal polygon} can be applied to $L'$, and the resulting quiver $\overrightarrow{B'}$ is isomorphic to $\overrightarrow{B}(\mu_{i}(Q))$, where $i$ is the vertex of $Q$ that corresponds to $a_i$.
  \begin{proof}
    If $i$ is a sink or a source, then the proposition follows from \cref{mutation of general type Cambrian lattice}.
    Thus, we may now assume that $i$ is neither a sink nor a source.

    Let $x,y \in L'$ be distinct atoms of $L'$.
    If $x,y \neq 0_L$, then $[0_L',x\vee_{L'} y]_{L'} = [0_L',x\vee_{L'} y]_{L}$.
    Let $a_j,a_k\in L$ be atoms of $L$ such that $x=a_ia_j$ or $x=a_i \vee_{L} a_j$, and such that $y=a_ia_k$ or $x=a_i \vee_{L} a_k$.
    Then, $[0_L',x\vee_{L'} y]_{L} \subset [0_L, a_i \vee_{L} a_j \vee_{L} a_k]_{L}$, and hence the edges between $x,y\in Q'$ is determined by $[0_L, a_i \vee_{L} a_j \vee_{L} a_k]_{L}$.    
    Since $[0_L, a_i \vee_{L} a_j \vee_{L} a_k]_{L}$ is a Cambrian lattice that the corresponding quiver is full subquiver of $\overrightarrow{B}$ with vertices $a_i,\ a_j,\ a_k$, it suffices to consider only those Cambrian lattices $L$ such that the corresponding quiver $\overrightarrow{B}$ has exactly three vertices.
    Such $L$ can only be the following types: $A_3$, $B_3$, $A_2\times A_1$, $B_2\times A_1$, $G_2\times A_1$, and $A_1\times A_1 \times A_1$.
    Here, a Cambrian lattice of type $X \times Y$ is defined as the product lattice of a Cambrian lattice of type $X$ and a Cambrian lattice of type $Y$.
    Therefore, the proposition follows from a straightforward computation in this case. 
    On the other hand, if $y=0_L$, let $a_j\in L$ be an atom of $L$ such that $x=a_ia_j$ or $x=a_i \vee_{L} a_j$.
    Using \cref{mutation of ideal}, we obtain $[0_L',0_L \vee_{L'} (a_i \vee_{L} a_j)]_{L'} = [0_L,a_i \vee_{L} a_j]_{L}$.
    Thus, a straightforward computation shows that the edge between $x,y\in Q'$ is reversed for $x,y \in Q$.
  \end{proof}
\end{theorem}

\begin{theorem}
\label{general type Ordovician lattice and the infinite case}
Let $L$ be a Dynkin type Cambrian lattice, $a_i \in L$ be an atom such that $i$ is neither a sink nor a source, and $L'=\mu_{a_i}(L)$. Then $L'$ cannot be expressed as a quotient of the weak order of any finite Coxeter groups. 
\begin{proof} 
We proceed by contradiction. Assume that such a map $f:S\to L'$ exists.
Without loss of generality, we may assume that no atom is mapped to $0$ by $f$. If not, \cref{quotient of weak order} allows us to replace $S$ with the parabolic subgroup $S'=\{a\in S \ | \ f(a)\neq 0\}$ and $f$ with $g$, so that the assumptions are satisfied.
For any two atoms $a,b$, we define $t(a,b)= \#([0,a\vee b])$.
It is straightforward to show that $t(a,b)=4$ if $ab=ba$.
Then, for any $3$ atoms $a,b,c$, at least one of $t(a,b),\ t(b,c),\ t(c,a)$ is equal to $4$.
By assumption, we may choose $j,k\in \overrightarrow{B}(Q)$ such that both an edge from $j$ to $i$ and from $i$ to $k$ exist.
Let $a'_i, a'_j, a'_k$ be atoms of $L'$.
We define $a'_i, a'_j, a'_k$ to correspond to $i,j,k \in \mu_i(\overrightarrow{B}(Q))$, respectively.
Then, by \cref{general type of Ordovician lattice and ideal polygon,weights of arrows and vertices,General type Cambrian lattice and ideal polygon}, $a'_i, a'_j, a'_k$ satisfy the following equation: $t(a'_i,a'_j),\ t(a'_j,a'_k),\ t(a'_k,a'_i)\geq 5$.
For distinct atoms $a,b$, $f$ satisfies $f(a)\neq f(b)$.
Hence, $t(a,b)\geq t(f(a),f(b))$.
However, if we take $a,b,c$ such that $f(a)=a'_i$, $f(b)=a'_j$, and $f(c)=a'_k$, this leads to a contradiction.

\end{proof}
\end{theorem}

A Cambrian lattice is a quotient of the weak order of a finite Coxeter group \cite{MR2258260}.
On the other hand, by our theorem, an Ordovician lattice is not, in general, a quotient of the weak order of a finite Coxeter group.

\begin{remark}[$\tilde{A}_2$ Affine Tamari lattice \cite{barkley2025affinetamarilattice}]
$\tilde{A}_2$ affine Tamari lattice coincides with \cref{fig:A2affineTamari}, where \cref{fig:A2affineTamari} represents an Ordovician lattice obtained from the $A_3$ Tamari lattice by a mutation.
\end{remark}

\begin{example}
  We present a mutation graph of the $A_3$ Tamari lattice. In this graph, the vertex set is all Ordocivian lattices obtained from the $A_3$ Tamari lattice by a finite sequence of mutations.
  We draw a directed edge from $L$ to $L'$ for each mutation $\mu$ such that $L'\cong \mu(L)$ exists.
  Since we ignore the labeling, this graph is not isomorphic to the underlying undirected graph of the Hasse quiver of the $A_3$ Tamari lattice.
  If mapping atoms of lattices to vertices of quivers by \cref{general type of Ordovician lattice and ideal polygon}, we obtain \cref{fig:mutation graph of A3 Ordovician lattice}. This coincides with the mutation graph of $A_3$ Coxeter quiver. 
  \begin{figure}[htbp]
\centering

\begin{tikzpicture}[scale=0.6]

\node[draw,shape=circle,inner sep=1pt] (A1) at (1,0) {};
\node[draw,shape=circle,inner sep=1pt] (B1) at (2,0) {};
\node[draw,shape=circle,inner sep=1pt] (C1) at (3,0) {};
\node[name=A3RR,draw,inner sep=2pt,circle,fit=(A1)(B1)(C1)] {};

\draw[->] (A1)--(B1);
\draw[->] (B1)--(C1);

\node[draw,shape=circle,inner sep=1pt] (A2) at (-3,3) {};
\node[draw,shape=circle,inner sep=1pt] (B2) at (-2,3) {};
\node[draw,shape=circle,inner sep=1pt] (C2) at (-1,3) {};
\node[name=A3LR,draw,inner sep=2pt,circle,fit=(A2)(B2)(C2)] {};

\draw[<-] (A2)--(B2);
\draw[->] (B2)--(C2);

\node[draw,shape=circle,inner sep=1pt] (A3) at (-3,-3) {};
\node[draw,shape=circle,inner sep=1pt] (B3) at (-2,-3) {};
\node[draw,shape=circle,inner sep=1pt] (C3) at (-1,-3) {};
\node[name=A3RL,draw,inner sep=2pt,circle,fit=(A3)(B3)(C3)] {};

\draw[->] (A3)--(B3);
\draw[<-] (B3)--(C3);

\node[draw,shape=circle,inner sep=1pt] (A4) at (7,0) {};
\node[draw,shape=circle,inner sep=1pt] (B4) at (8,0) {};
\node[draw,shape=circle,inner sep=1pt] (C4) at (9,0) {};
\node[name=A3Circle,draw,inner sep=2pt,circle,fit=(A4)(B4)(C4)] {};

\draw[->] (A4)--(B4);
\draw[->] (B4)--(C4);
\draw[->] (C4) to[bend right] (A4);

\draw[-{Stealth[scale=1]}] (A3RR) to[bend right] (A3LR);
\draw[-{Stealth}] (A3LR) to[bend right] (A3RR);
\draw[-{Stealth}] (A3LR) to[bend right=45] (A3RR);
\draw[-{Stealth[scale=1]}] (A3RR) to[bend right] (A3RL);
\draw[-{Stealth}] (A3RL) to[bend right] (A3RR);
\draw[-{Stealth}] (A3RL) to[bend right=45] (A3RR);
\draw[-{Stealth[scale=1]}] (A3LR) to[bend right] (A3RL);
\draw[-{Stealth[scale=1]}] (A3RL) to[bend right] (A3LR);
\draw[-{Stealth[scale=1]}] (A3RR) to[bend right] (A3Circle);
\draw[-{Stealth}] (A3Circle) to[bend right] (A3RR);
\draw[-{Stealth}] (A3Circle) to[bend right=45] (A3RR);
\draw[-{Stealth}] (A3Circle) to[bend right=60] (A3RR);

\end{tikzpicture}

\caption[figure iii]{A mutation graph of Ordovician lattices obtained from the $A_3$ Tamari lattice. If multiple mutations send to the same lattices, we duplicate edges.}
\label{fig:mutation graph of A3 Ordovician lattice}
\end{figure}
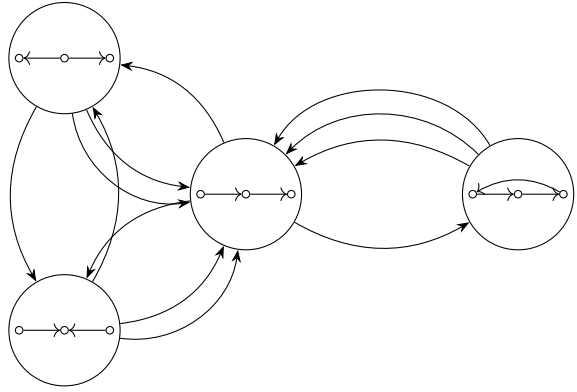
\end{example}

\begin{remark}
  The weak order of an infinite Coxeter group is, in general, not a lattice.
  However, affine Coxeter groups, a subclass of infinite Coxeter groups, are not lattices, but can be embedded to an affine Dyer lattice \cite{MR3943754,barkley2025affineextendedweakorder}.
  An affine Tamari lattice is a quotient lattice of an affine Dyer lattice \cite{barkley2025affinetamarilattice}. 
\end{remark}

\begin{example}
  We can obtain the following lattice structures and their opposites by a finite sequence of mutations from a $B_3$ Tamari lattice. These are all structures obtainable in this way.
  The rightmost lattice is a type-B Ordovician but not a type-B Cambrian. Others are type-B Cambrian.
  The rightmost lattice equals its opposite.
  Furthermore, the rightmost lattice is locally mutable and can be mutated into itself or a $B_3$ Tamari.
      \begin{figure}[htbp]
\centering
\begin{tabular}{ccc}
  \begin{minipage}[t]{0.3\linewidth}
  \begin{tikzpicture}[scale=0.3]

\node[name=min,draw,inner sep=2pt,circle] at (0,-18) {};
\node[name=a,draw,inner sep=2pt,circle] at (-3,-16) {};
\node[name=b,draw,inner sep=2pt,circle] at (0,-16) {};
\node[name=c,draw,inner sep=2pt,circle] at (3,-16) {};
\node[name=d,draw,inner sep=2pt,circle] at (5,-14) {};
\node[name=e,draw,inner sep=2pt,circle] at (-1,-12) {};
\node[name=f,draw,inner sep=2pt,circle] at (1,-12) {};
\node[name=g,draw,inner sep=2pt,circle] at (5,-12) {};
\node[name=h,draw,inner sep=2pt,circle] at (7,-12) {};
\node[name=i,draw,inner sep=2pt,circle] at (8,-10) {};
\node[name=j,draw,inner sep=2pt,circle] at (4,-9) {};
\node[name=k,draw,inner sep=2pt,circle] at (8,-8) {};
\node[name=l,draw,inner sep=2pt,circle] at (-1,-7) {};
\node[name=m,draw,inner sep=2pt,circle] at (5,-6) {};
\node[name=n,draw,inner sep=2pt,circle] at (7,-6) {};
\node[name=o,draw,inner sep=2pt,circle] at (5,-4) {};
\node[name=p,draw,inner sep=2pt,circle] at (-3,-2) {};
\node[name=q,draw,inner sep=2pt,circle] at (0,-2) {};
\node[name=r,draw,inner sep=2pt,circle] at (3,-2) {};
\node[name=max,draw,inner sep=2pt,circle] at (0,0) {};

\draw[black] (min)--(a);
\draw[black] (min)--(b);
\draw[black] (min)--(c);
\draw[black] (a)--(f);
\draw[black] (a)--(p);
\draw[black] (b)--(e);
\draw[black] (b)--(h);
\draw[black] (c)--(d);
\draw[black] (c)--(f);
\draw[black] (d)--(g);
\draw[black] (d)--(h);
\draw[black] (e)--(l);
\draw[black] (e)--(m);
\draw[black] (f)--(q);
\draw[black] (g)--(j);
\draw[black] (g)--(i);
\draw[black] (h)--(i);
\draw[black] (i)--(k);
\draw[black] (j)--(n);
\draw[black] (j)--(q);
\draw[black] (k)--(m);
\draw[black] (k)--(n);
\draw[black] (l)--(p);
\draw[black] (l)--(r);
\draw[black] (m)--(o);
\draw[black] (n)--(o);
\draw[black] (o)--(r);
\draw[black] (p)--(max);
\draw[black] (q)--(max);
\draw[black] (r)--(max);

\end{tikzpicture}
  \centering
  \captionsetup{width=0.8\textwidth}
  \caption{$B_3$ Tamari}
  \end{minipage} &
  
  \begin{minipage}[t]{0.3\linewidth}
  \begin{tikzpicture}[scale=0.3]

\node[name=min,draw,inner sep=2pt,circle] at (0,-18) {};
\node[name=a,draw,inner sep=2pt,circle] at (-3,-16) {};
\node[name=b,draw,inner sep=2pt,circle] at (0,-16) {};
\node[name=c,draw,inner sep=2pt,circle] at (3,-16) {};
\node[name=d,draw,inner sep=2pt,circle] at (0,-14) {};
\node[name=e,draw,inner sep=2pt,circle] at (3,-14) {};
\node[name=f,draw,inner sep=2pt,circle] at (3,-12) {};
\node[name=g,draw,inner sep=2pt,circle] at (6,-12) {};
\node[name=h,draw,inner sep=2pt,circle] at (0,-10) {};
\node[name=i,draw,inner sep=2pt,circle] at (3,-10) {};
\node[name=j,draw,inner sep=2pt,circle] at (-3,-8) {};
\node[name=k,draw,inner sep=2pt,circle] at (0,-8) {};
\node[name=l,draw,inner sep=2pt,circle] at (3,-8) {};
\node[name=m,draw,inner sep=2pt,circle] at (0,-6) {};
\node[name=n,draw,inner sep=2pt,circle] at (6,-6) {};
\node[name=o,draw,inner sep=2pt,circle] at (0,-4) {};
\node[name=p,draw,inner sep=2pt,circle] at (-3,-2) {};
\node[name=q,draw,inner sep=2pt,circle] at (0,-2) {};
\node[name=r,draw,inner sep=2pt,circle] at (3,-2) {};
\node[name=max,draw,inner sep=2pt,circle] at (0,0) {};

\draw[black] (min)--(a);
\draw[black] (min)--(b);
\draw[black] (min)--(c);
\draw[black] (a)--(j);
\draw[black] (a)--(q);
\draw[black] (b)--(d);
\draw[black] (b)--(e);
\draw[black] (c)--(g);
\draw[black] (c)--(q);
\draw[black] (d)--(f);
\draw[black] (d)--(h);
\draw[black] (e)--(f);
\draw[black] (e)--(g);
\draw[black] (f)--(i);
\draw[black] (g)--(n);
\draw[black] (h)--(j);
\draw[black] (h)--(k);
\draw[black] (i)--(k);
\draw[black] (i)--(l);
\draw[black] (j)--(p);
\draw[black] (k)--(m);
\draw[black] (l)--(m);
\draw[black] (l)--(n);
\draw[black] (m)--(o);
\draw[black] (n)--(r);
\draw[black] (o)--(p);
\draw[black] (o)--(r);
\draw[black] (p)--(max);
\draw[black] (q)--(max);
\draw[black] (r)--(max);

\end{tikzpicture}
  \centering
  \captionsetup{width=0.8\textwidth}
  \caption{A $B_3$ Cambrian but not a Tamari.}
  \end{minipage} &

  \begin{minipage}[t]{0.3\linewidth}
\begin{tikzpicture}[scale=0.3]

\node[name=min,draw,inner sep=2pt,circle] at (0,-18) {};
\node[name=a,draw,inner sep=2pt,circle] at (-3,-16) {};
\node[name=b,draw,inner sep=2pt,circle] at (0,-16) {};
\node[name=c,draw,inner sep=2pt,circle] at (3,-16) {};
\node[name=d,draw,inner sep=2pt,circle] at (4,-14) {};
\node[name=e,draw,inner sep=2pt,circle] at (0,-12) {};
\node[name=f,draw,inner sep=2pt,circle] at (5,-12) {};
\node[name=g,draw,inner sep=2pt,circle] at (-5,-10) {};
\node[name=h,draw,inner sep=2pt,circle] at (0,-10) {};
\node[name=i,draw,inner sep=2pt,circle] at (4,-10) {};
\node[name=j,draw,inner sep=2pt,circle] at (-5,-8) {};
\node[name=k,draw,inner sep=2pt,circle] at (0,-8) {};
\node[name=l,draw,inner sep=2pt,circle] at (4,-8) {};
\node[name=m,draw,inner sep=2pt,circle] at (0,-6) {};
\node[name=n,draw,inner sep=2pt,circle] at (5,-6) {};
\node[name=o,draw,inner sep=2pt,circle] at (4,-4) {};
\node[name=p,draw,inner sep=2pt,circle] at (-3,-2) {};
\node[name=q,draw,inner sep=2pt,circle] at (0,-2) {};
\node[name=r,draw,inner sep=2pt,circle] at (3,-2) {};
\node[name=max,draw,inner sep=2pt,circle] at (0,0) {};

\draw[black] (min)--(a);
\draw[black] (min)--(b);
\draw[black] (min)--(c);
\draw[black] (a)--(g);
\draw[black] (a)--(j);
\draw[black] (b)--(e);
\draw[black] (b)--(l);
\draw[black] (c)--(d);
\draw[black] (c)--(k);
\draw[black] (d)--(f);
\draw[black] (d)--(m);
\draw[black] (e)--(h);
\draw[black] (e)--(o);
\draw[black] (f)--(i);
\draw[black] (f)--(l);
\draw[black] (g)--(k);
\draw[black] (g)--(p);
\draw[black] (h)--(j);
\draw[black] (h)--(r);
\draw[black] (i)--(n);
\draw[black] (i)--(q);
\draw[black] (j)--(p);
\draw[black] (k)--(m);
\draw[black] (l)--(n);
\draw[black] (m)--(q);
\draw[black] (n)--(o);
\draw[black] (o)--(r);
\draw[black] (p)--(max);
\draw[black] (q)--(max);
\draw[black] (r)--(max);

\end{tikzpicture}
  \centering
  \captionsetup{width=0.8\textwidth}
  \caption{A type-B Ordovician but not a type-B Cambrian.}
  \end{minipage}
\end{tabular}
\end{figure}
 
\end{example}

\begin{example}
  We present a mutation graph of a $B_3$ Tamari lattice. In this graph, the vertex set is all type-B Ordocivian lattices obtained from a $B_3$ Tamari lattice by a finite sequence of mutations. We draw a directed edge from $L$ to $L'$ for each mutation $\mu$ such that $L'\cong \mu(L)$ exists.
  Since we ignore the labeling, this graph is not isomorphic to the underlying undirected graph of the Hasse quiver of $B_3$ Tamari lattice.
  If mapping atoms of lattices to vertices of quivers by \cref{general type of Ordovician lattice and ideal polygon}, we obtain \cref{fig:mutation graph of B3 Ordovician lattice}. This coincides with the mutation graph of $B_3$ Coxeter quiver. 
  \begin{figure}[htbp]
\centering
\begin{tikzpicture}[scale=0.4]

\node[draw,shape=circle,inner sep=2pt] (A1) at (4,8) {};
\node[draw,shape=circle,inner sep=2pt,double] (B1) at (6,8) {};
\node[draw,shape=circle,inner sep=2pt,double] (C1) at (8,8) {};
\node[name=B3RR,draw,inner sep=2pt,circle,fit=(A1)(B1)(C1)] {};

\draw[->] (A1)--(B1);
\draw[->] (B1)--(C1);

\node[draw,shape=circle,inner sep=2pt] (A2) at (-6,6) {};
\node[draw,shape=circle,inner sep=2pt,double] (B2) at (-4,6) {};
\node[draw,shape=circle,inner sep=2pt,double] (C2) at (-2,6) {};
\node[name=B3LR,draw,inner sep=2pt,circle,fit=(A2)(B2)(C2)] {};

\draw[<-] (A2)--(B2);
\draw[->] (B2)--(C2);

\node[draw,shape=circle,inner sep=2pt] (A3) at (-6,-6) {};
\node[draw,shape=circle,inner sep=2pt,double] (B3) at (-4,-6) {};
\node[draw,shape=circle,inner sep=2pt,double] (C3) at (-2,-6) {};
\node[name=B3RL,draw,inner sep=2pt,circle,fit=(A3)(B3)(C3)] {};

\draw[->] (A3)--(B3);
\draw[<-] (B3)--(C3);

\node[draw,shape=circle,inner sep=2pt] (A4) at (4,-8) {};
\node[draw,shape=circle,inner sep=2pt,double] (B4) at (6,-8) {};
\node[draw,shape=circle,inner sep=2pt,double] (C4) at (8,-8) {};
\node[name=B3LL,draw,inner sep=2pt,circle,fit=(A4)(B4)(C4)] {};

\draw[<-] (A4)--(B4);
\draw[<-] (B4)--(C4);

\node[draw,shape=circle,inner sep=2pt] (A5) at (12,0) {};
\node[draw,shape=circle,inner sep=2pt,double] (B5) at (14,0) {};
\node[draw,shape=circle,inner sep=2pt,double] (C5) at (16,0) {};
\node[name=B3Circle,draw,inner sep=2pt,circle,fit=(A5)(B5)(C5)] {};

\draw[->] (A5)--(B5);
\draw[->] (B5)--(C5);
\draw[->] (C5) to[bend right] (A5);

\draw[-{Stealth}] (B3RR) to[bend right] (B3LR);
\draw[-{Stealth}] (B3RR) to[bend right] (B3RL);
\draw[-{Stealth}] (B3LR) to[bend right] (B3RL);
\draw[-{Stealth}] (B3RL) to[bend right] (B3LL);
\draw[-{Stealth}] (B3LR) to[bend right] (B3LL);
\draw[-{Stealth}] (B3LR) to[bend right] (B3RR);
\draw[-{Stealth}] (B3RL) to[bend right] (B3RR);
\draw[-{Stealth}] (B3RL) to[bend right] (B3LR);
\draw[-{Stealth}] (B3LL) to[bend right] (B3RL);
\draw[-{Stealth}] (B3LL) to[bend right] (B3LR);
\draw[-{Stealth}] (B3RR) to[bend right] (B3Circle);
\draw[-{Stealth}] (B3LL) to[bend right] (B3Circle);
\draw[-{Stealth}] (B3Circle) to[bend right] (B3RR);
\draw[-{Stealth}] (B3Circle) to[bend right] (B3LL);
\draw[-{Stealth}] (B3Circle) to[looseness=2,loop right] (B3Circle);

\end{tikzpicture}
\caption{A mutation graph of Ordovician lattices obtained from a $B_3$ Tamari lattice.}
\label{fig:mutation graph of B3 Ordovician lattice}
\end{figure}
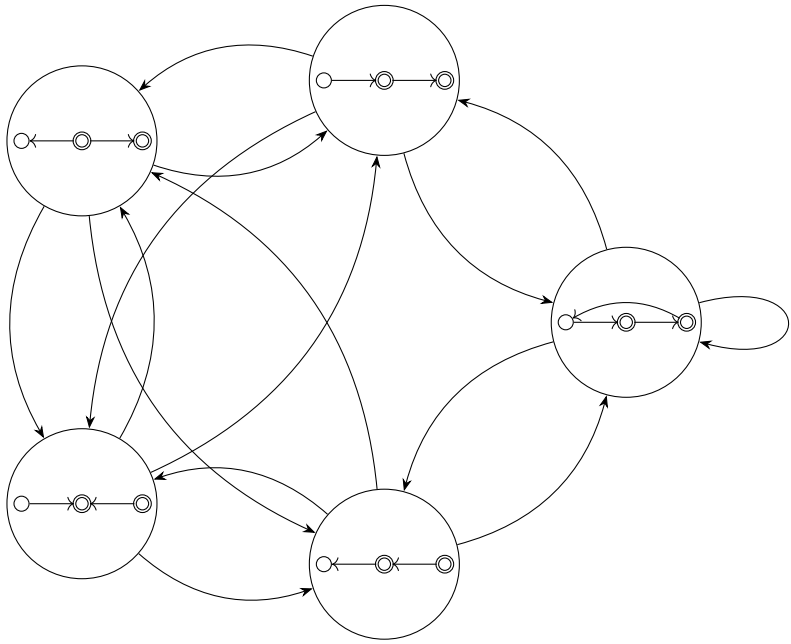
\end{example}

\begin{remark}[Other types of Ordovician lattices]
  We suggest that a \emph{type-D Ordovician lattice}, which is defined as obtained from a $D_{n}$ Cambrian lattice by a finite sequence of mutations, can also be obtained from $\tilde{A}_{n-1}$ affine Tamari lattice by a finite sequence of mutations.
  Indeed, $\tilde{A}_2$ affine Tamari lattice is obtained from the $A_3$ Tamari lattice by a mutation, and $A_3$ Coxeter group coincides with $D_3$ Coxeter group.
\end{remark}

\appendix

\section{A graphical description of flips of Cambrian lattices}
\crefalias{section}{appendix}
\label{sec:A graphical description of flips of Cambrian lattices}
In this section, we give a graphical proof of Type-A,B of Theorem.
This is author's original proof.

\begin{definition}[\cite{MR2258260}]
\label{type-A labelled polygon}
Fix $n$ as a positive integer.
We define the polygon $Q(\overrightarrow{B})$ to be a $(n+3)$-gon together with a labeling $0,1,\cdots,n+2$ of its vertices, where the labeling is given as follows:

Assign $1,2,\cdots,n+1$ to an upper or lower position.
For each $i\in\{1,2,\cdots,n+1\}$, if $i$ is assigned to the upper side, we denote it by $\overline{i}$.
If $i$ is assigned to the lower side, we denote it by $\underline{i}$.
We place $0$ at the leftmost position, $n+2$ at the rightmost position.
On the upper side (resp.~lower side) path, the elements assigned to the upper side (resp.~lower side) are arranged from left to right in increasing order.
In addition, we define a Coxeter quiver $\overrightarrow{B}$ to correspond to $Q(\overrightarrow{B})$.
We assign an orientation to $\overrightarrow{B}$ as follows:
For each $i\in\{1,2,\cdots,n-1\}$, $i\rightarrow i+1$ if $i+1$ is assigned to the upper side, $i\rightarrow i+1$ if $i+1$ is assigned to the lower side.
\begin{figure}[htbp]
  \centering
  \begin{tikzpicture}[scale=0.5]
  \node[name=0,draw,inner sep=2pt,circle] at (0,0) {0};
  \node[name=1,draw,inner sep=2pt,circle] at (2,2) {$\overline{1}$};
  \node[name=2,draw,inner sep=2pt,circle] at (4,3) {$\overline{2}$};
  \node[name=3,draw,inner sep=2pt,circle] at (6,-3) {$\underline{3}$};
  \node[name=4,draw,inner sep=2pt,circle] at (8,3) {$\overline{4}$};
  \node[name=5,draw,inner sep=2pt,circle] at (10,-3) {$\underline{5}$};
  \node[name=6,draw,inner sep=2pt,circle] at (12,2) {$\overline{6}$};
  \node[name=7,draw,inner sep=2pt,circle] at (14,0) {7};
  \node[name=q1,draw,inner sep=2pt,circle] at (3,0) {1};
  \node[name=q2,draw,inner sep=2pt,circle] at (5,0) {2};
  \node[name=q3,draw,inner sep=2pt,circle] at (7,0) {3};
  \node[name=q4,draw,inner sep=2pt,circle] at (9,0) {4};
  \node[name=q5,draw,inner sep=2pt,circle] at (11,0) {5};
  \node[name=qm1] at (4,0) {};
  \node[name=qm2] at (6,0) {};
  \node[name=qm3] at (8,0) {};
  \node[name=qm4] at (10,0) {};

  \draw (0)--(1);
  \draw (1)--(2);
  \draw (2)--(4);
  \draw (4)--(6);
  \draw (6)--(7);
  \draw (0)--(3);
  \draw (3)--(5);
  \draw (5)--(7);
  \draw[<-] (q1)--(q2);
  \draw[->] (q2)--(q3);
  \draw[<-] (q3)--(q4);
  \draw[->] (q4)--(q5);
  \draw[dotted] (2)--(qm1);
  \draw[dotted] (3)--(qm2);
  \draw[dotted] (4)--(qm3);
  \draw[dotted] (5)--(qm4);
\end{tikzpicture}
  \caption{A polygon and type-A Coxeter quiver correspond to a Cambrian lattice}
  \label{fig:type-A polygon and quiver}
\end{figure}
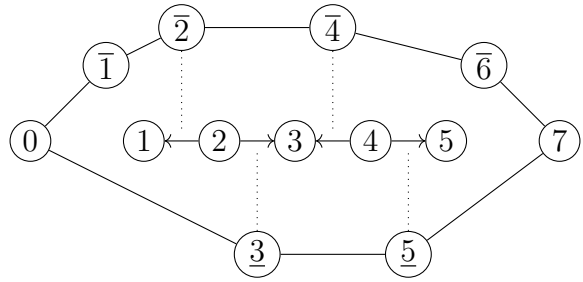
\end{definition}

We simply write $\overline{j}$ (resp.~$\underline{j}$) to indicate that $j$ is assigned to the upper side (resp.~lower side).
Moreover, we simply denote $B$ by the underlying undirected graph of $\overrightarrow{B}$.

\begin{definition}[\cite{MR2258260}]
For each $\sigma \in S_{n+1}$, we define
\[\sigma= \begin{pmatrix} 1 & \cdots & n+1 \\ \sigma(1) & \cdots & \sigma(n+1) \end{pmatrix},\]
and we construct $\eta$ which sends $\sigma$ to a triangulation of $Q(\overrightarrow{B})$;
\begin{enumerate}
  \item Let $C=\{0,n+2\}\cup \{j\in \{1,2,\cdots,n+1\}\ |\ \underline{j}\}$, and let $i=1$ be a counter.
  \item Draw a zigzag which through all elements in $C$ in increasing order.
  \item Remove $\sigma(i)$ from $C$ if $\sigma(i) \in C$; otherwise, add $\sigma(i)$ to $C$.
  \item Increase $i$ by $1$ if $i\neq n+1$; otherwise, terminate the procedure.
  \item Return to Step 2.
\end{enumerate}
$\eta$ induces a lattice structure on the weak order of $S_{n+1}$ in triangulations of $n+3$-gon, which is called \emph{(Type $A_n$) Cambrian lattice} \cite{MR2258260}.
\end{definition}

\begin{example}
We present an example. In $Q(\overrightarrow{B})$, $1,2,4,6$ are assigned to the upper side, and $3,5$ are assigned to the lower side. Let $\sigma= \begin{pmatrix} 1 & 2 & 3 & 4 & 5 & 6 \\ 1 & 4 & 3 & 6 & 5 & 2 \end{pmatrix} $.
Then $\eta(\sigma)$ is determined as \cref{fig:the map eta}:
  \begin{figure}[htbp]
    \centering
    \begin{tabular}{ccc}
    \begin{minipage}[t]{0.3\linewidth}
    \begin{tikzpicture}[scale=0.22]
  \node[name=0,draw,inner sep=2pt,circle] at (0,0) {0};
  \node[name=1,draw,inner sep=2pt,circle] at (2,3) {1};
  \node[name=2,draw,inner sep=2pt,circle] at (4,5) {2};
  \node[name=3,draw,inner sep=2pt,circle] at (6,-3) {3};
  \node[name=4,draw,inner sep=2pt,circle] at (8,5) {4};
  \node[name=5,draw,inner sep=2pt,circle] at (10,-3) {5};
  \node[name=6,draw,inner sep=2pt,circle] at (12,3) {6};
  \node[name=7,draw,inner sep=2pt,circle] at (14,0) {7};

  \draw (0)--(1);
  \draw (1)--(2);
  \draw (2)--(4);
  \draw (4)--(6);
  \draw (6)--(7);
  \draw (0)--(3);
  \draw (3)--(5);
  \draw (5)--(7);
  
\end{tikzpicture}
\end{minipage} &
  \begin{minipage}[t]{0.3\linewidth}
    \begin{tikzpicture}[scale=0.22]
  \node[name=0,draw,inner sep=2pt,circle] at (0,0) {0};
  \node[name=1,draw,inner sep=2pt,circle] at (2,3) {1};
  \node[name=2,draw,inner sep=2pt,circle] at (4,5) {2};
  \node[name=3,draw,inner sep=2pt,circle] at (6,-3) {3};
  \node[name=4,draw,inner sep=2pt,circle] at (8,5) {4};
  \node[name=5,draw,inner sep=2pt,circle] at (10,-3) {5};
  \node[name=6,draw,inner sep=2pt,circle] at (12,3) {6};
  \node[name=7,draw,inner sep=2pt,circle] at (14,0) {7};

  \draw (0)--(1);
  \draw (1)--(2);
  \draw (2)--(4);
  \draw (4)--(6);
  \draw (6)--(7);
  \draw (0)--(3);
  \draw (3)--(5);
  \draw (5)--(7);
  \draw[blue] (0)--(3)--(5)--(7);
  
\end{tikzpicture}
\end{minipage} &
\begin{minipage}[t]{0.3\linewidth}
    \begin{tikzpicture}[scale=0.22]
  \node[name=0,draw,inner sep=2pt,circle] at (0,0) {0};
  \node[name=1,draw,inner sep=2pt,circle] at (2,3) {1};
  \node[name=2,draw,inner sep=2pt,circle] at (4,5) {2};
  \node[name=3,draw,inner sep=2pt,circle] at (6,-3) {3};
  \node[name=4,draw,inner sep=2pt,circle] at (8,5) {4};
  \node[name=5,draw,inner sep=2pt,circle] at (10,-3) {5};
  \node[name=6,draw,inner sep=2pt,circle] at (12,3) {6};
  \node[name=7,draw,inner sep=2pt,circle] at (14,0) {7};

  \draw (0)--(1);
  \draw (1)--(2);
  \draw (2)--(4);
  \draw (4)--(6);
  \draw (6)--(7);
  \draw (0)--(3);
  \draw (3)--(5);
  \draw (5)--(7);
  \draw[cyan] (0)--(3);
  \draw[blue] (0)--(1)--(3)--(5)--(7);
  
\end{tikzpicture}
\end{minipage}\\ 
\begin{minipage}[t]{0.3\linewidth}
    \begin{tikzpicture}[scale=0.22]
  \node[name=0,draw,inner sep=2pt,circle] at (0,0) {0};
  \node[name=1,draw,inner sep=2pt,circle] at (2,3) {1};
  \node[name=2,draw,inner sep=2pt,circle] at (4,5) {2};
  \node[name=3,draw,inner sep=2pt,circle] at (6,-3) {3};
  \node[name=4,draw,inner sep=2pt,circle] at (8,5) {4};
  \node[name=5,draw,inner sep=2pt,circle] at (10,-3) {5};
  \node[name=6,draw,inner sep=2pt,circle] at (12,3) {6};
  \node[name=7,draw,inner sep=2pt,circle] at (14,0) {7};

  \draw (0)--(1);
  \draw (1)--(2);
  \draw (2)--(4);
  \draw (4)--(6);
  \draw (6)--(7);
  \draw (0)--(3);
  \draw (3)--(5);
  \draw (5)--(7);
  \draw[cyan] (0)--(3);
  \draw[cyan] (3)--(5);
  \draw[blue] (0)--(1)--(3)--(4)--(5)--(7);
  
\end{tikzpicture}
\end{minipage}&
\begin{minipage}[t]{0.3\linewidth}
    \begin{tikzpicture}[scale=0.22]
  \node[name=0,draw,inner sep=2pt,circle] at (0,0) {0};
  \node[name=1,draw,inner sep=2pt,circle] at (2,3) {1};
  \node[name=2,draw,inner sep=2pt,circle] at (4,5) {2};
  \node[name=3,draw,inner sep=2pt,circle] at (6,-3) {3};
  \node[name=4,draw,inner sep=2pt,circle] at (8,5) {4};
  \node[name=5,draw,inner sep=2pt,circle] at (10,-3) {5};
  \node[name=6,draw,inner sep=2pt,circle] at (12,3) {6};
  \node[name=7,draw,inner sep=2pt,circle] at (14,0) {7};

  \draw (0)--(1);
  \draw (1)--(2);
  \draw (2)--(4);
  \draw (4)--(6);
  \draw (6)--(7);
  \draw (0)--(3);
  \draw (3)--(5);
  \draw (5)--(7);
  \draw[cyan] (0)--(3);
  \draw[cyan] (3)--(5);
  \draw[cyan] (1)--(3)--(4);
  \draw[blue] (0)--(1)--(4)--(5)--(7);
  
\end{tikzpicture}
\end{minipage}&
\begin{minipage}[t]{0.3\linewidth}
    \begin{tikzpicture}[scale=0.22]
  \node[name=0,draw,inner sep=2pt,circle] at (0,0) {0};
  \node[name=1,draw,inner sep=2pt,circle] at (2,3) {1};
  \node[name=2,draw,inner sep=2pt,circle] at (4,5) {2};
  \node[name=3,draw,inner sep=2pt,circle] at (6,-3) {3};
  \node[name=4,draw,inner sep=2pt,circle] at (8,5) {4};
  \node[name=5,draw,inner sep=2pt,circle] at (10,-3) {5};
  \node[name=6,draw,inner sep=2pt,circle] at (12,3) {6};
  \node[name=7,draw,inner sep=2pt,circle] at (14,0) {7};

  \draw (0)--(1);
  \draw (1)--(2);
  \draw (2)--(4);
  \draw (4)--(6);
  \draw (6)--(7);
  \draw (0)--(3);
  \draw (3)--(5);
  \draw (5)--(7);
  \draw[cyan] (0)--(3);
  \draw[cyan] (3)--(5);
  \draw[cyan] (1)--(3)--(4);
  \draw[cyan] (5)--(7);
  \draw[blue] (0)--(1)--(4)--(5)--(6)--(7);
  
\end{tikzpicture}
\end{minipage}\\
\begin{minipage}[t]{0.3\linewidth}
    \begin{tikzpicture}[scale=0.22]
  \node[name=0,draw,inner sep=2pt,circle] at (0,0) {0};
  \node[name=1,draw,inner sep=2pt,circle] at (2,3) {1};
  \node[name=2,draw,inner sep=2pt,circle] at (4,5) {2};
  \node[name=3,draw,inner sep=2pt,circle] at (6,-3) {3};
  \node[name=4,draw,inner sep=2pt,circle] at (8,5) {4};
  \node[name=5,draw,inner sep=2pt,circle] at (10,-3) {5};
  \node[name=6,draw,inner sep=2pt,circle] at (12,3) {6};
  \node[name=7,draw,inner sep=2pt,circle] at (14,0) {7};

  \draw (0)--(1);
  \draw (1)--(2);
  \draw (2)--(4);
  \draw (4)--(6);
  \draw (6)--(7);
  \draw (0)--(3);
  \draw (3)--(5);
  \draw (5)--(7);
  \draw[cyan] (0)--(3);
  \draw[cyan] (3)--(5);
  \draw[cyan] (1)--(3)--(4);
  \draw[cyan] (5)--(7);
  \draw[cyan] (4)--(5)--(6);
  \draw[blue] (0)--(1)--(4)--(6)--(7);
  
\end{tikzpicture}
\end{minipage}&
\begin{minipage}[t]{0.3\linewidth}
    \begin{tikzpicture}[scale=0.22]
  \node[name=0,draw,inner sep=2pt,circle] at (0,0) {0};
  \node[name=1,draw,inner sep=2pt,circle] at (2,3) {1};
  \node[name=2,draw,inner sep=2pt,circle] at (4,5) {2};
  \node[name=3,draw,inner sep=2pt,circle] at (6,-3) {3};
  \node[name=4,draw,inner sep=2pt,circle] at (8,5) {4};
  \node[name=5,draw,inner sep=2pt,circle] at (10,-3) {5};
  \node[name=6,draw,inner sep=2pt,circle] at (12,3) {6};
  \node[name=7,draw,inner sep=2pt,circle] at (14,0) {7};

  \draw (0)--(1);
  \draw (1)--(2);
  \draw (2)--(4);
  \draw (4)--(6);
  \draw (6)--(7);
  \draw (0)--(3);
  \draw (3)--(5);
  \draw (5)--(7);
  \draw[cyan] (0)--(3);
  \draw[cyan] (3)--(5);
  \draw[cyan] (1)--(3)--(4);
  \draw[cyan] (5)--(7);
  \draw[cyan] (4)--(5)--(6);
  \draw[cyan] (1)--(4);
  \draw[blue] (0)--(1)--(2)--(4)--(6)--(7);
  
\end{tikzpicture}
\end{minipage}&
\begin{minipage}[t]{0.3\linewidth}
    \begin{tikzpicture}[scale=0.22]
  \node[name=0,draw,inner sep=2pt,circle] at (0,0) {0};
  \node[name=1,draw,inner sep=2pt,circle] at (2,3) {1};
  \node[name=2,draw,inner sep=2pt,circle] at (4,5) {2};
  \node[name=3,draw,inner sep=2pt,circle] at (6,-3) {3};
  \node[name=4,draw,inner sep=2pt,circle] at (8,5) {4};
  \node[name=5,draw,inner sep=2pt,circle] at (10,-3) {5};
  \node[name=6,draw,inner sep=2pt,circle] at (12,3) {6};
  \node[name=7,draw,inner sep=2pt,circle] at (14,0) {7};

  \draw (0)--(1);
  \draw (1)--(2);
  \draw (2)--(4);
  \draw (4)--(6);
  \draw (6)--(7);
  \draw (0)--(3);
  \draw (3)--(5);
  \draw (5)--(7);
  \draw[cyan] (0)--(3);
  \draw[cyan] (3)--(5);
  \draw[cyan] (1)--(3)--(4);
  \draw[cyan] (5)--(7);
  \draw[cyan] (4)--(5)--(6);
  \draw[cyan] (1)--(4);
  \draw[cyan] (0)--(1)--(2)--(4)--(6)--(7);
  
\end{tikzpicture}
\end{minipage}
\end{tabular}

    \caption{The map from a symmetric group to triangulations of a polygon}
    \label{fig:the map eta}
  \end{figure}
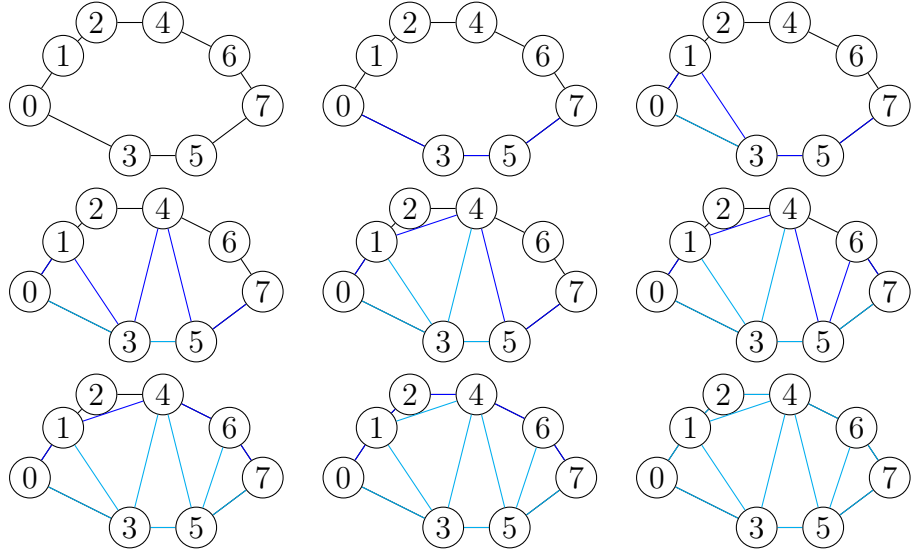
\end{example}

In a Cambrian lattice, $x\prec y$ or $x\succ y$ if and only if $y$ is obtained from $x$ by replacing one diagonal line.
Hence we can represent a cover relation $x\prec y$ in a Cambrian lattice by 2 crossing diagonals, in particular by a quadrilateral with these two diagonals.

\begin{lemma}
  Let $0\leq a\leq b\leq c\leq d\leq n+2$. For $x\prec y$ presented by a quadrilateral $abcd$, the following hold:
  \begin{enumerate}
    \item If $b,c$ are assigned to the opposite sides, then $bc,ad$ are diagonals; otherwise, $ac,bd$ are diagonals.
    \item If $\overline{b}$, then the diagonal with endpoint $b$ is smaller than the other diagonal; otherwise, diagonal with endpoint $a$ is smaller than the other diagonal.
  \end{enumerate}
  \begin{proof}
    (1) follows from \cref{type-A labelled polygon}.
    (2) We define $T,U\subset$ $\{1$, $2$, $\cdots$, $n+1\}$ as follows:
    \[T=\{i \in \{1,2,\cdots,n+1\} \setminus \{a,b,c,d\}\ | \ \underline{i}\} \sqcup \{i\in \{a,d\}\ | \ \overline{i}\},\]
    we wrtie $T=\{t_1,t_2,\cdots,t_m\}$, where $t_1\leq t_2\leq \cdots t_m$.
    \[U=\{i \in \{1,2,\cdots,n+1\} \setminus \{a,b,c,d\}\ | \ \underline{i}\} \sqcup \{i\in \{a,d\}\ | \ \overline{i}\},\]
    we write $U=\{u_1,u_2,\cdots,u_l\}$, where $u_1\leq u_2\leq \cdots u_l$.
    We define $\sigma_1,\sigma_2 \in S_{n+1}$ as follows:
    \begin{align*}
      \sigma_1=&(t_1, t_2, \cdots, t_m, b, c, u_1, u_2, \cdots, u_l),\\
      \sigma_2=&(t_1, t_2, \cdots, t_m, c, b, u_1, u_2, \cdots, u_l)
    \end{align*}
    Then the cover relation between $\eta(\sigma_1),\eta(\sigma_2)$ corresponds to the quadrilateral $abcd$.
  \end{proof} 
\label{diagonals and ordering}
\end{lemma}

Using \cref{kernel of eta} \cite{MR2258260}, we can determine which pairs of elements are mapped to the same element by $\eta$.
\begin{definition}[\cite{MR2258260}]
Let $\sigma \in S_{n+1}$. We define $i\lneq j\lneq k$ to be a \emph{$\overline{2}31$-pattern} if $\overline{j}$ and $\sigma^{-1}(j)\leq \sigma^{-1}(k)\leq \sigma^{-1}(i)$.
A \emph{$31\underline{2}$-pattern} is defined in the same manner.
\end{definition}

\begin{definition}[\cite{MR2258260}]
  Let $i\lneq j\lneq k$ be a $\overline{2}31$-pattern of $\sigma \in S_{n+1}$. We define a \emph{$\overline{2}31 \rightarrow \overline{2}13$-move} to be an operation that exchanges $i$ and $k$.
  A \emph{$31\underline{2}\rightarrow 13\underline{2}$-move} is defined similarly.
\end{definition}

\begin{remark}[\cite{MR2258260}]
\label{kernel of eta}
  The value $\eta(\sigma)$ is invariant under a $\overline{2}31 \rightarrow \overline{2}13$-move, a $31\underline{2}\rightarrow 13\underline{2}$-move applied to $\sigma$.
  Conversely, if $\eta(\sigma_{1})=\eta(\sigma_{2})$, then $\sigma_{1}$ and $\sigma_{2}$ can be transformed into each other by a finite sequence of $\overline{2}31 \rightarrow \overline{2}13$-moves and $31\underline{2}\rightarrow 13\underline{2}$-moves and their inverses.
\end{remark}

In particular, $\eta(\sigma)$ remains unchanged regardless of whether $1,n+1$ are assigned to the upper or lower side.
Hence, a Cambrian lattice is determined from $\overrightarrow{B}$.
We denote it by $\Camb(\overrightarrow{B})$.

\begin{proposition}[\cref{dual order of Cambrian lattice in section 7}]
  The lattice structure $\Camb(\overrightarrow{B'})$ is the dual order of $\Camb(\overrightarrow{B})$, where $\overrightarrow{B'}$ is obtained from $\overrightarrow{B}$ by reversing the orientation of every edge.
  \begin{proof}
    $Q(\overrightarrow{B'})$ is obtained from $Q(\overrightarrow{B})$ by swapping all numbers between the upper and lower sides. Then the result follows from \cref{diagonals and ordering}.
    A similar argument applies to type-B.
  \end{proof}
  \label{dual order of Cambrian lattice in section A}
\end{proposition}

For type-A, the original proof of \cref{mutation of general type Cambrian lattice}(2,3) is as follows:
\begin{proof} 
(2) If necessary, we assume that $1$ (resp.~$n+1$) and $2$ (resp.~$n$) lie on opposite sides.
Since $i\in \overrightarrow{B}$ is a sink or a source, $Q(\mu_{i}(\overrightarrow{B}))$ is obtained from $Q(\overrightarrow{B})$ by swapping $i$ and $i+1$.

For each cover relation $x\prec y$ in $\Camb(\overrightarrow{B})$, let $abcd$ be the corresponding quadrilateral.
In $\Camb (\mu_{i}(\overrightarrow{B}))$, $x\prec y$ corresponds to the quadrilateral $a'b'c'd'$ obtained from $abcd$ by replacing each occurrence $i+1$ with $i$ and each occurrence $i$ with $i+1$.

By \cref{diagonals and ordering}, which of the two diagonals corresponding to $a'b'c'd'$ is larger than the other diagonal depends only the sides and the order relations of $a',b',c',d'$. 
Hence, if $abcd$ and $a'b'c'd'$ have exactly the same sides and order relations,
then the ordering of the two diagonals (i.e., which diagonal is larger) remains unchanged.
From this, a straightforward computation shows the following: if the ordering of the two diagonals are changed, then $\{a,b,c,d\} \supset \{i,i+1\}$.
We may now assume $\{a,b,c,d\} \supset \{i,i+1\}$.
We prove this using \cref{diagonals and ordering}.
Our assumption implies that $i$ and $i+1$ lie on opposite sides.
Without loss of generality, we set $a\lneq c, b\lneq d, a\lneq b$.

If $a=i$, then $b=i+1$, and hence the ordering is unchanged.
If $b=i$, then $d=i+1$, and hence the ordering is reversed.
Otherwise $\{c,d\}=\{i,i+1\}$, and hence the ordering is unchanged.
Therefore the ordering is changed if and only if one of the diagonal is $(i\edge i+1)$.

Whether an element $\eta(\sigma)$ contains a diagonal $(i\edge i+1)$ or not corresponds to whether $\sigma^{-1}(i)\le \sigma^{-1}(i+1)$ (or $\sigma^{-1}(i)\ge \sigma^{-1}(i+1)$), and the the correspondence depends on the side of $i$.
Whether $\sigma^{-1}(i)\le \sigma^{-1}(i+1)$ (or $\sigma^{-1}(i)\ge \sigma^{-1}(i+1)$) also corresponds to whether $\sigma \geq (i,i+1)$ or not.
Hence, $\Camb(\mu_{i}(\overrightarrow{B}))$ is a mutation of $\Camb(\overrightarrow{B})$ on the atom $\eta((i,i+1))$.

(3) follows from (2) and \cref{direction of tree}.
\end{proof}

\begin{proposition}[\cref{Flip-Flop of Cambrian lattices in section 7}]
  The flip in (2) can also be described as a Flip-Flop \cite{ladkani2007universalderivedequivalencesposets}.
  In particular, $A=\partial A$ or $B=\partial B$ in this case.
  \begin{proof}
    In this case, exactly one of $A$ and $B$ corresponds to the set of triangulations that contain $(i\edge i+1)$.
  \end{proof}
  \label{Flip-Flop of Cambrian lattices in section A}
\end{proposition}

Hereafter, We apply a similar argument for type-B.

\begin{definition}
  Let $T=\{-n, -(n-1), \cdots, -2, -1, 1, 2, \cdots, n-1, n\}$. We define a type-B Coxeter group by $W_n=\{\sigma:T\to T\ |\ \text{bijective}$, $\sigma(-i) = -\sigma(i)\}$.
\end{definition}

We can take its generators by
$s_1=(-1,1)$, $s_2=(-2,-1)(1,2)$, $s_3=(-3,-2)(2,3)$, $\cdots$, $s_n=(-n,-(n-1))(n-1,n)$.
From now on, we assume that $S=\{s_1,\cdots,s_n\}$.

\begin{definition}[\cite{MR2258260}]
Fix a positive integer $n$.
We define the polygon $Q(\overrightarrow{B})$ to be a $(2n+2)$-gon together with a labeling 
$-(n+1)$, $-n$, $\cdots$, $-2$, $-1$, $1$, $2$, $\cdots$, $n$, $n+1$ of its vertices, where the labeling is given as follows:

Assign $-n,-(n-1),\cdots,-2,-1,1,2,\cdots,n-1,n$ to an upper or lower position such that $i$ and $-i$ lie on opposite sides.
For each $i\in\{-n,-(n-1),\cdots,-2,-1,1,2,\cdots,n-1,n\}$, if $i$ is assigned to the upper side, we denote it by $\overline{i}$.
If $i$ is assigned to the lower side, we denote it by $\underline{i}$.

We place $-(n+1)$ at the leftmost position, $n+1$ at the rightmost position.
On the upper side (resp.~lower side) path, the elements assigned to the upper side (resp.~lower side) are arranged from left to right in increasing order.

In addition, we define a type-B Coxeter quiver $\overrightarrow{B}$ to correspond to $Q(\overrightarrow{B})$.
We assign an orientation to $\overrightarrow{B}$ as follows:
For each $i\in\{1,2,\cdots,n-2,n-1\}$, $i\rightarrow i+1$ if $i+1$ is assigned to the upper side, $i\rightarrow i+1$ if $i+1$ is assigned to the lower side.
\begin{figure}[htbp]
  \centering
\begin{tikzpicture}[scale=0.5]
  \node[name=-4,draw,inner sep=2pt,circle] at (0,0) {$-4$};
  \node[name=-3,draw,inner sep=2pt,circle] at (2,2) {$\overline{-3}$};
  \node[name=-2,draw,inner sep=2pt,circle] at (4,3) {$\overline{-2}$};
  \node[name=-1,draw,inner sep=2pt,circle] at (6,-3) {$\underline{-1}$};
  \node[name=1,draw,inner sep=2pt,circle] at (8,3) {$\overline{1}$};
  \node[name=2,draw,inner sep=2pt,circle] at (10,-3) {$\underline{2}$};
  \node[name=3,draw,inner sep=2pt,circle] at (12,-2) {$\underline{3}$};
  \node[name=4,draw,inner sep=2pt,circle] at (14,0) {$4$};
  \node[name=q1,draw,inner sep=2pt,circle] at (7,0) {$1$};
  \node[name=q2,draw,inner sep=2pt,circle] at (9,0) {$2$};
  \node[name=q3,draw,inner sep=2pt,circle] at (11,0) {$3$};
  \node[name=qm1] at (8,0) {};
  \node[name=qm2] at (10,0) {};

  \draw (-4)--(-3);
  \draw (-3)--(-2);
  \draw (-2)--(1);
  \draw (1)--(4);
  \draw (-4)--(-1);
  \draw (-1)--(2);
  \draw (2)--(3);
  \draw (3)--(4);
  \draw[<-] (q1)--(q2) node[midway, below] {4};
  \draw[->] (q2)--(q3);
  \draw[dotted] (1)--(qm1);
  \draw[dotted] (2)--(qm2);
\end{tikzpicture}
  \caption{A type-B Coxeter quiver $Q(\overrightarrow{B})$ and a polygon correspond to a type-B Cambrian lattice}
  \label{fig:type-B polygon and quiver}
\end{figure}
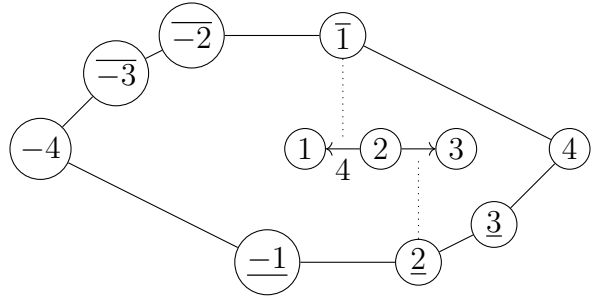
\end{definition}

As in the case of type-A, $\eta_{B}(\sigma)$ remains unchanged regardless of whether $-n,n$ are assigned to the upper or lower side \cite{MR2258260}.
Hence, a type-B Cambrian lattice is determined from $\overrightarrow{B}$.
We denote it by $\Camb(\overrightarrow{B})$.

\begin{definition}[\cite{MR2258260}]
  For each $\sigma \in W_{n}$, we define 
  \[\sigma= \begin{pmatrix} -n & \cdots & -1 & 1 & \cdots & n  \\ \sigma(-n) & \cdots & \sigma(-1) & \sigma(1) & \cdots & \sigma(n) \end{pmatrix},\]
  and construct $\eta_{B}$ which sends $\sigma$ to a point-symmetric triangulation of $Q(\overrightarrow{B})$;

\begin{enumerate}
  \item Let $C=\{-(n+1),n+1\}$ $\cup$ $\{j\in \{-n,-(n-1),\cdots,-2,-1$, $1,2,\cdots,n-1,n\}\ |\ \underline{j}\}$, and let $i=-n$ be a counter.
  \item Draw a zigzag which through all elements in $C$ in increasing order.
  \item Remove $\sigma(i)$ from $C$ if $\sigma(i) \in C$; otherwise, add $\sigma(i)$ to $C$.
  \item Increase $i$ by $1$ if $i\neq -1,n+1$. If $i=-1$, change it to $1$. If $i=n+1$, we terminate the procedure. These cases exhaust all possibilities.
  \item Return to Step 2.
\end{enumerate}

$\eta_{B}$ induce a lattice structure on the weak order of $W_{n}$ in point-symmetric triangulations of $2n+2$-gon, which is called \emph{(Type $B_n$) Cambrian lattice} \cite{MR2258260}.
No atoms are send to $0$ by $\eta_{B}$. 

\end{definition}
As in the case of type-A Cambrian lattices, no atom is mapped to $0$ by $\eta_{B}$ \cite{MR2258260}.

We call the diagonals $(-i\edge i)$ for some $i\in\{1,2,\cdots,n+1\}$ a \emph{long edge}.
For any $j\neq i,-i$, we call the pair of diagonals $(i\edge j),(-i\edge -j)$ a \emph{short edge}.

In a type-B Cambrian lattice, $x\prec y$ or $x\succ y$ if and only if $y$ is obtained from $x$ by replacing one long edge (resp.~short edge) by another long edge (resp.~short edge).
Hence we can represent a cover relation $x\prec y$ in a type-B Cambrian lattice by 2 crossing long edges(which construct one quadrilateral) or 2 crossing short edges(which constuct two quadrilaterals that are point-symmetric to each other).

For type-B, the original proof of \cref{mutation of general type Cambrian lattice}(2,3) is as follows:
\begin{proof} 
(2) Take any two elements $x,y$ such that $x\prec y$ or $x\succ y$.
The method in \cref{diagonals and ordering} also applies to check the ordering of $x,y$.
When there are two quadrilaterals corresponding to $\{x,y\}$, we may choose a quadrilateral arbitrary.

Case1: a flip by vertex $1$

Between $Q(\overrightarrow{B}),Q(\mu_{i}(\overrightarrow{B}))$, the vertices $-1,1$ are swapped.
Like type-A, the ordering of $x,y$ is reversed if and only if $(-1\edge 1)$ is one of the diagonals.

Whether an element $\eta_{B}(\sigma)$ contains diagonal $(-1\edge 1)$ or not is corresponds to whether $\sigma^{-1}(-1)\le \sigma^{-1}(1)$ (or $\sigma^{-1}(-1)\ge \sigma^{-1}(1)$), and the the correspondence depends on the side of $1$.
Whether $\sigma^{-1}(-1)\le \sigma^{-1}(1)$ (or $\sigma^{-1}(-1)\ge \sigma^{-1}(1)$) also corresponds to whether $\sigma \geq (-1,1)$ or not.
Hence, $\Camb(\mu_{i}(\overrightarrow{B}))$ is a mutation of $\Camb(\overrightarrow{B})$ on the atom $\eta((-1,1))$ in this case.

Case2: a flip by a vertex $i\geq 2$ 

If necessary, we assume that $-n(resp.~n)$ and $-n+1(resp.~n-1)$ lie on opposite sides. 
Since $i\in \overrightarrow{B}$ is a sink or a source, $Q(\mu_{i}(\overrightarrow{B}))$ is obtained from $Q(\overrightarrow{B})$ by swapping $-i+1$ with $-i$, and $i-1$ with $i$.

As in the case of type-A, the ordering of $x,y$ is reversed if and only if $(-(i+1)\edge -i)$ (or equivalently, $(i\edge i+1)$) is one of the diagonals.

Whether an element $\eta_{B}(\sigma)$ contains diagonal $(i-1\edge i)$ or not is corresponds to whether $\sigma^{-1}(i-1)\le \sigma^{-1}(i)$ (or $\sigma^{-1}(i-1)\ge \sigma^{-1}(i)$), and the the correspondence depends on the side of $i-1$.
Whether $\sigma^{-1}(i-1)\le \sigma^{-1}(i)$ (or $\sigma^{-1}(i-1)\ge \sigma^{-1}(i)$) also corresponds to whether $\sigma \geq (-(i+1),-i)(i,i+1)$ or not.
Therefore, part (2) is proved.

(3) follows from (2) and \cref{direction of tree}.
\end{proof}

\begin{conjecture}
  Let $L$ be a type-A or type-B Ordovician order with label induced by a type-A or type-B Cambrian lattice, and let $x\in L$.
  Assume that $L$ is mutable.
  Then, $u(x)$, defined in \cref{the map u}, is obtained from $x$ by a rotation that sends each vertex to the next vertex in the clockwise direction.
\end{conjecture}

\section{Lattice structures on the $A_3$ associahedron}
\crefalias{section}{appendix}
\label{sec:Lattice structures on the A_3 associahedron}
In this section, we give a complete list of all lattice structures on $A_3$ associahedron.
In particular, we give an example that is not mutable but N-regular.

$A_3$ associahedron is obtained by gluing two copies of \cref{fig:part_of_associahedron} along the parts indicated by the blue dots and lines.
\begin{figure}[htbp]
  \centering

  \caption{}
  \label{fig:part_of_associahedron}
\end{figure}

Using \cref{cycle in lattice}, we obtain the following lemma:

\begin{lemma}
  Let $K,L$ be poset structures on \cref{fig:part_of_associahedron} used to construct $A_3$ associahedron, so that a lattice structure on $A_3$ associahedron is obtained by gluing $K$ and $L$ along the parts indicated by the blue dots and lines.
  Then, every possible $K (resp.~L)$ is of one of the 11 types shown in \cref{fig:possible poset structure} and their opposites:
\begin{figure}[htbp]
%
\caption{}
\label{fig:possible poset structure}
\end{figure}

\end{lemma}

\begin{proposition}
  There are exactly $8$ types of lattice structures on $A_3$ associahedron.
  \begin{enumerate}
    \item Three $A_3$ Cambrian lattices
    \item $\tilde{A}_2$ Affine Tamari lattice
    \item \cref{fig:BadCase1} and its opposite
    \item \cref{fig:BadCase2} and its opposite
  \end{enumerate}
  \begin{proof}
    From now on, we denote by $X--Y$ the result of gluing $X$ and $Y$.
    Moreover, we denote by $X'$ the opposite of $X$.
    Only the following candidates are consistent with \cref{cycle in lattice}:
    $A1--A1'$, $B1--B1'$, $A2--A2'$, $A2--B3'$, $B3--B3'$, $B2--C4'$, $B2--D1$, $C1--C1$, $C2--C2$, $C2--C2'$ \text{ and their opposites}.
    Among these candidates, $A1--A1'$, $B1--B1'$, $C1--C1$, and $C1'--C1'$ are Ordovician lattices obtained from the $A_3$ Tamari lattice.
    The cases $A2--A2'$, $B3--B3'$, $C2--C2'$, $C2--C2$, and $B2--D1$ are not lattices, since there is no least upper bound of black elements that are shown in \cref{Gluings that do not form a lattice}.
    Also their opposites are not lattices.
      \begin{figure}[htbp]
\centering
\begin{tabular}{ccccc}
  \begin{minipage}[t]{0.2\linewidth}
  \centering
  \begin{tikzpicture}[scale=0.3]
\node[name=min,draw,inner sep=2pt,circle] at (0,-10) {};
\node[name=a,draw,inner sep=2pt,circle,fill=black] at (-2,-8) {};
\node[name=b,draw,inner sep=2pt,circle] at (0,-8) {};
\node[name=c,draw,inner sep=2pt,circle] at (2,-8) {};
\node[name=d,draw,inner sep=2pt,circle] at (-2,-6) {};
\node[name=e,draw,inner sep=2pt,circle] at (2,-6) {};
\node[name=f,draw,inner sep=2pt,circle,fill=black] at (4,-6) {};
\node[name=g,draw,inner sep=2pt,circle] at (-2,-4) {};
\node[name=h,draw,inner sep=2pt,circle] at (2,-4) {};
\node[name=i,draw,inner sep=2pt,circle] at (4,-4) {};
\node[name=j,draw,inner sep=2pt,circle] at (-2,-2) {};
\node[name=k,draw,inner sep=2pt,circle] at (0,-2) {};
\node[name=l,draw,inner sep=2pt,circle] at (2,-2) {};
\node[name=max,draw,inner sep=2pt,circle] at (0,0) {};

\draw[black] (min)--(a);
\draw[black] (min)--(b);
\draw[black] (min)--(c);
\draw[black] (a)--(d);
\draw[black] (a)--(e);
\draw[black] (b)--(g);
\draw[black] (b)--(h);
\draw[black] (c)--(f);
\draw[black] (c)--(i);
\draw[black] (d)--(g);
\draw[black] (d)--(k);
\draw[black] (e)--(i);
\draw[black] (e)--(k);
\draw[black] (f)--(h);
\draw[black] (f)--(l);
\draw[black] (g)--(j);
\draw[black] (h)--(j);
\draw[black] (i)--(l);
\draw[black] (j)--(max);
\draw[black] (k)--(max);
\draw[black] (l)--(max);

\end{tikzpicture}
  \subcaption*{$A2--A2'$}
  \end{minipage} &
  
  \begin{minipage}[t]{0.16\linewidth}
  \centering
  \begin{tikzpicture}[scale=0.3]
\node[name=min,draw,inner sep=2pt,circle] at (0,-12) {};
\node[name=a,draw,inner sep=2pt,circle,fill=black] at (-2,-10) {};
\node[name=b,draw,inner sep=2pt,circle] at (0,-10) {};
\node[name=c,draw,inner sep=2pt,circle,fill=black] at (2,-10) {};
\node[name=d,draw,inner sep=2pt,circle] at (-2,-8) {};
\node[name=e,draw,inner sep=2pt,circle] at (0,-8) {};
\node[name=f,draw,inner sep=2pt,circle] at (-2,-6) {};
\node[name=g,draw,inner sep=2pt,circle] at (2,-6) {};
\node[name=h,draw,inner sep=2pt,circle] at (0,-4) {};
\node[name=i,draw,inner sep=2pt,circle] at (2,-4) {};
\node[name=j,draw,inner sep=2pt,circle] at (-2,-2) {};
\node[name=k,draw,inner sep=2pt,circle] at (0,-2) {};
\node[name=l,draw,inner sep=2pt,circle] at (2,-2) {};
\node[name=max,draw,inner sep=2pt,circle] at (0,0) {};

\draw[black] (min)--(a);
\draw[black] (min)--(b);
\draw[black] (min)--(c);
\draw[black] (a)--(d);
\draw[black] (a)--(h);
\draw[black] (b)--(e);
\draw[black] (b)--(g);
\draw[black] (c)--(e);
\draw[black] (c)--(j);
\draw[black] (d)--(f);
\draw[black] (d)--(g);
\draw[black] (e)--(l);
\draw[black] (f)--(i);
\draw[black] (f)--(k);
\draw[black] (g)--(i);
\draw[black] (h)--(j);
\draw[black] (h)--(k);
\draw[black] (i)--(l);
\draw[black] (j)--(max);
\draw[black] (k)--(max);
\draw[black] (l)--(max);

\end{tikzpicture}
  \subcaption*{$B3--B3'$}
  \end{minipage} &

  \begin{minipage}[t]{0.16\linewidth}
  \centering
\begin{tikzpicture}[scale=0.3]
\node[name=min,draw,inner sep=2pt,circle] at (0,-10) {};
\node[name=a,draw,inner sep=2pt,circle,fill=black] at (-2,-8) {};
\node[name=b,draw,inner sep=2pt,circle] at (0,-8) {};
\node[name=c,draw,inner sep=2pt,circle] at (2,-8) {};
\node[name=d,draw,inner sep=2pt,circle] at (-2,-6) {};
\node[name=e,draw,inner sep=2pt,circle] at (0,-6) {};
\node[name=f,draw,inner sep=2pt,circle,fill=black] at (2,-6) {};
\node[name=g,draw,inner sep=2pt,circle] at (-2,-4) {};
\node[name=h,draw,inner sep=2pt,circle] at (0,-4) {};
\node[name=i,draw,inner sep=2pt,circle] at (2,-4) {};
\node[name=j,draw,inner sep=2pt,circle] at (-2,-2) {};
\node[name=k,draw,inner sep=2pt,circle] at (0,-2) {};
\node[name=l,draw,inner sep=2pt,circle] at (2,-2) {};
\node[name=max,draw,inner sep=2pt,circle] at (0,0) {};

\draw[black] (min)--(a);
\draw[black] (min)--(b);
\draw[black] (min)--(c);
\draw[black] (a)--(d);
\draw[black] (a)--(e);
\draw[black] (b)--(f);
\draw[black] (b)--(g);
\draw[black] (c)--(d);
\draw[black] (c)--(h);
\draw[black] (d)--(i);
\draw[black] (e)--(g);
\draw[black] (e)--(j);
\draw[black] (f)--(h);
\draw[black] (f)--(k);
\draw[black] (g)--(k);
\draw[black] (h)--(l);
\draw[black] (i)--(j);
\draw[black] (i)--(l);
\draw[black] (j)--(max);
\draw[black] (k)--(max);
\draw[black] (l)--(max);

\end{tikzpicture}
  \subcaption*{$C2--C2'$}
  \end{minipage} &

  \begin{minipage}[t]{0.16\linewidth}
  \centering
\begin{tikzpicture}[scale=0.3]
\node[name=min,draw,inner sep=2pt,circle] at (0,-10) {};
\node[name=a,draw,inner sep=2pt,circle] at (-2,-8) {};
\node[name=b,draw,inner sep=2pt,circle] at (0,-8) {};
\node[name=c,draw,inner sep=2pt,circle,fill=black] at (2,-8) {};
\node[name=d,draw,inner sep=2pt,circle] at (-2,-6) {};
\node[name=e,draw,inner sep=2pt,circle] at (0,-6) {};
\node[name=f,draw,inner sep=2pt,circle] at (2,-6) {};
\node[name=g,draw,inner sep=2pt,circle,fill=black] at (-2,-4) {};
\node[name=h,draw,inner sep=2pt,circle] at (0,-4) {};
\node[name=i,draw,inner sep=2pt,circle] at (2,-4) {};
\node[name=j,draw,inner sep=2pt,circle] at (-2,-2) {};
\node[name=k,draw,inner sep=2pt,circle] at (0,-2) {};
\node[name=l,draw,inner sep=2pt,circle] at (2,-2) {};
\node[name=max,draw,inner sep=2pt,circle] at (0,0) {};

\draw[black] (min)--(a);
\draw[black] (min)--(b);
\draw[black] (min)--(c);
\draw[black] (a)--(e);
\draw[black] (a)--(h);
\draw[black] (b)--(e);
\draw[black] (b)--(i);
\draw[black] (c)--(d);
\draw[black] (c)--(f);
\draw[black] (d)--(i);
\draw[black] (d)--(k);
\draw[black] (e)--(g);
\draw[black] (f)--(h);
\draw[black] (f)--(k);
\draw[black] (g)--(j);
\draw[black] (g)--(l);
\draw[black] (h)--(j);
\draw[black] (i)--(l);
\draw[black] (j)--(max);
\draw[black] (k)--(max);
\draw[black] (l)--(max);

\end{tikzpicture}
  \subcaption*{$C2--C2$}
  \end{minipage} &
  
  \begin{minipage}[t]{0.16\linewidth}
  \centering
\begin{tikzpicture}[scale=0.3]
\node[name=min,draw,inner sep=2pt,circle] at (0,-12) {};
\node[name=a,draw,inner sep=2pt,circle,fill=black] at (-2,-10) {};
\node[name=b,draw,inner sep=2pt,circle] at (0,-10) {};
\node[name=c,draw,inner sep=2pt,circle,fill=black] at (2,-10) {};
\node[name=d,draw,inner sep=2pt,circle] at (-2,-8) {};
\node[name=e,draw,inner sep=2pt,circle] at (0,-8) {};
\node[name=f,draw,inner sep=2pt,circle] at (-2,-6) {};
\node[name=g,draw,inner sep=2pt,circle] at (0,-6) {};
\node[name=h,draw,inner sep=2pt,circle] at (-2,-4) {};
\node[name=i,draw,inner sep=2pt,circle] at (0,-4) {};
\node[name=j,draw,inner sep=2pt,circle] at (-2,-2) {};
\node[name=k,draw,inner sep=2pt,circle] at (0,-2) {};
\node[name=l,draw,inner sep=2pt,circle] at (2,-2) {};
\node[name=max,draw,inner sep=2pt,circle] at (0,0) {};

\draw[black] (min)--(a);
\draw[black] (min)--(b);
\draw[black] (min)--(c);
\draw[black] (a)--(d);
\draw[black] (a)--(g);
\draw[black] (b)--(e);
\draw[black] (b)--(f);
\draw[black] (c)--(e);
\draw[black] (c)--(l);
\draw[black] (d)--(f);
\draw[black] (d)--(i);
\draw[black] (e)--(j);
\draw[black] (f)--(h);
\draw[black] (g)--(i);
\draw[black] (g)--(l);
\draw[black] (h)--(j);
\draw[black] (h)--(k);
\draw[black] (i)--(k);
\draw[black] (j)--(max);
\draw[black] (k)--(max);
\draw[black] (l)--(max);

\end{tikzpicture}
  \subcaption*{$B2--D1$}
  \end{minipage}
\end{tabular}
\caption{Gluings that do not form a lattice}
\label{Gluings that do not form a lattice}
\end{figure}
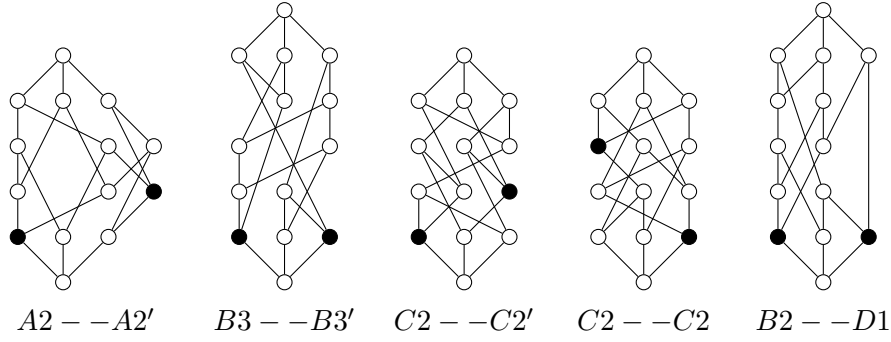

On the other hand, it follows from a straightforward computation that $A2--B3'$, $B2--C4'$ are lattices.
Also their opposites are lattices.
\begin{figure}[htbp]
\centering
\begin{tabular}{cc}
  \begin{minipage}[t]{0.45\linewidth}
  \centering
  \begin{tikzpicture}[scale=0.5]

\node[name=min,draw,inner sep=2pt,circle] at (0,-12) {};
\node[name=a,draw,inner sep=2pt,circle] at (-2,-10) {2};
\node[name=b,draw,inner sep=2pt,circle] at (0,-10) {};
\node[name=c,draw,inner sep=2pt,circle] at (2,-10) {1'};
\node[name=d,draw,inner sep=2pt,circle] at (2,-8) {};
\node[name=e,draw,inner sep=2pt,circle] at (4,-8) {};
\node[name=f,draw,inner sep=2pt,circle] at (2,-6) {};
\node[name=g,draw,inner sep=2pt,circle] at (0,-4) {};
\node[name=h,draw,inner sep=2pt,circle] at (2,-4) {};
\node[name=i,draw,inner sep=2pt,circle] at (4,-4) {};
\node[name=j,draw,inner sep=2pt,circle] at (-2,-2) {1'};
\node[name=k,draw,inner sep=2pt,circle] at (0,-2) {};
\node[name=l,draw,inner sep=2pt,circle] at (2,-2) {2};
\node[name=max,draw,inner sep=2pt,circle] at (0,0) {};

\draw[black] (min)--(a);
\draw[black] (min)--(b);
\draw[black] (min)--(c);
\draw[black] (a)--(g);
\draw[black] (a)--(j);
\draw[black] (b)--(h);
\draw[black] (b)--(j);
\draw[black] (c)--(d);
\draw[black] (c)--(e);
\draw[black] (d)--(f);
\draw[black] (d)--(g);
\draw[black] (e)--(h);
\draw[black] (e)--(i);
\draw[black] (f)--(i);
\draw[black] (f)--(k);
\draw[black] (g)--(k);
\draw[black] (h)--(l);
\draw[black] (i)--(l);
\draw[black] (j)--(max);
\draw[black] (k)--(max);
\draw[black] (l)--(max);

\end{tikzpicture}
  \caption{$A2--B3'$}
  \label{fig:BadCase1}
  \end{minipage} &
  
  \begin{minipage}[t]{0.45\linewidth}
  \centering
  \begin{tikzpicture}[scale=0.5]

\node[name=min,draw,inner sep=2pt,circle] at (0,-12) {};
\node[name=a,draw,inner sep=2pt,circle] at (-2,-10) {};
\node[name=b,draw,inner sep=2pt,circle] at (0,-10) {};
\node[name=c,draw,inner sep=2pt,circle] at (2,-10) {1};
\node[name=d,draw,inner sep=2pt,circle] at (0,-8) {};
\node[name=e,draw,inner sep=2pt,circle] at (2,-8) {};
\node[name=f,draw,inner sep=2pt,circle] at (-2,-6) {};
\node[name=g,draw,inner sep=2pt,circle] at (0,-6) {};
\node[name=h,draw,inner sep=2pt,circle] at (0,-4) {};
\node[name=i,draw,inner sep=2pt,circle] at (2,-4) {};
\node[name=j,draw,inner sep=2pt,circle] at (-2,-2) {1};
\node[name=k,draw,inner sep=2pt,circle] at (0,-2) {};
\node[name=l,draw,inner sep=2pt,circle] at (2,-2) {};
\node[name=max,draw,inner sep=2pt,circle] at (0,0) {};

\draw[black] (min)--(a);
\draw[black] (min)--(b);
\draw[black] (min)--(c);
\draw[black] (a)--(f);
\draw[black] (a)--(g);
\draw[black] (b)--(d);
\draw[black] (b)--(j);
\draw[black] (c)--(d);
\draw[black] (c)--(e);
\draw[black] (d)--(l);
\draw[black] (e)--(g);
\draw[black] (e)--(i);
\draw[black] (f)--(h);
\draw[black] (f)--(j);
\draw[black] (g)--(h);
\draw[black] (h)--(k);
\draw[black] (i)--(k);
\draw[black] (i)--(l);
\draw[black] (j)--(max);
\draw[black] (k)--(max);
\draw[black] (l)--(max);

\end{tikzpicture}
  \caption{$B2--C4'$}
  \label{fig:BadCase2}
  \end{minipage}
\end{tabular}
\end{figure}

  \end{proof}
\end{proposition}

\begin{proposition}
\cref{fig:BadCase1} has two mutations, and is mutated into \cref{fig:BadCase2} or the opposite of \cref{fig:BadCase1}.
\cref{fig:BadCase2} has one mutation, and is mutated into \cref{fig:BadCase1}.
  \begin{proof}
    This can be checked directly.
    \cref{fig:BadCase1,fig:BadCase2} exhibit the atom (resp.~coatom) corresponding to a mutation.
  \end{proof}
\end{proposition}

\cref{fig:BadCase1,fig:BadCase2} are 3-regular but not locally mutable, then there are counterexamples to the converse of \cref{conj:mutable implies N-regular}.

\vspace{5mm}
\noindent {\bf{Acknowledgments.}} This research was supported by the Science Tokyo Support Program for Doctoral Students, funded by the Universities for International Research Excellence.

This paper is based on the author's master's thesis, completed under the supervision of Hironori Oya. The author is deeply grateful to Hironori Oya for his continuous and kind guidance throughout the course of this research.

The author would like to express sincere gratitude to Sota Asai, Baptiste Rognerud, and Nathan Reading for their invaluable advice and helpful comments on this research.

\end{document}